\documentclass[10pt,a4paper]{article}
\usepackage[utf8]{inputenc}
\usepackage[T1]{fontenc}
\usepackage[left=1in,right=2in,top=1in,bottom=1in]{geometry}
\usepackage{charter}
\usepackage[english]{babel}

\usepackage{amsmath}
\usepackage{amsfonts}
\usepackage{amssymb}
\usepackage{mathtools}
\usepackage{blkarray}
\usepackage{multirow}
\usepackage[affil-sl]{authblk}
\usepackage{wasysym}
\usepackage{graphicx}
\usepackage{xcolor}
\definecolor{darkred}{RGB}{150, 0, 0}
\definecolor{darkgreen}{RGB}{0, 150, 0}
\definecolor{darkblue}{RGB}{0, 0, 150}
\usepackage[breaklinks,colorlinks,citecolor=darkgreen,linkcolor=darkred,urlcolor=darkblue,bookmarks=false]{hyperref}
\usepackage{bookmark}
\usepackage[vlined,ruled,commentsnumbered,linesnumbered,norelsize]{algorithm2e}
\usepackage{amsthm}
\usepackage{thmtools}
\usepackage{thm-restate}
\usepackage[textsize=footnotesize]{todonotes}
\usepackage{enumerate}
\usepackage{etoolbox}
\usepackage{tikz}
\usepackage{booktabs}
\usepackage{placeins}
\usepackage{dsfont}
\usepackage[inline]{enumitem}
\usepackage{xspace}
\usepackage[font=footnotesize,labelfont=bf,skip=2pt]{caption}
\usepackage{subcaption}
\usepackage{calc}
\usepackage{siunitx}

\makeatletter
\newcounter{subsubparagraph}[subparagraph]
\renewcommand\thesubsubparagraph{%
  \thesubparagraph.\@arabic\c@subsubparagraph}
\newcommand\subsubparagraph{%
  \@startsection{subsubparagraph}
    {6}
    {\parindent}
    {3.25ex \@plus 1ex \@minus .2ex}
    {-1em}
    {\normalfont\normalsize\bfseries\em}}
\newcommand\l@subsubparagraph{\@dottedtocline{6}{10em}{5em}}
\newcommand{\subsubparagraphmark}[1]{}
\makeatother

\DeclareMathAlphabet{\mathrm}{OT1}{bch}{m}{n}

\DeclareMathOperator{\bigo}{\mathcal{O}}

\DeclareMathOperator{\ord}{ord}
\DeclareMathOperator{\supp}{supp}

\DeclareMathOperator{\ancestors}{Anc}

\DeclareMathOperator{\ent}{Ent}
\DeclareMathOperator{\restr}{restr}

\DeclareSIUnit{\byte}{B}

\newcommand{\symmetricgroup}[1]{\mathcal{S}_{#1}}

\newcommand{\blank}{\ensuremath{\_}}
\newcommand{\binary}{\{0, 1\}}
\newcommand{\binaryblank}{\{0, 1, \blank\}}

\newcommand{\xymresack}[1]{\mathcal{X}_{#1}}

\newcommand{\xymretope}[1]{\mathcal{X}_{#1}}
\newcommand{\id}{\mathrm{id}}
\newcommand{\ihat}{\hat{\imath}}
\newcommand{\icheck}{\check{\imath}}
\newcommand{\func}[1]{\ensuremath{\text{\textsc{#1}}}\xspace}
\newcommand{\True}{\func{True}}
\newcommand{\False}{\func{False}}
\newcommand{\Feasible}{\func{Feasible}}
\newcommand{\Infeasible}{\func{Infeasible}}

\newcommand{\Break}{\textbf{break}}

\newcommand{\gen}[1]{\langle #1 \rangle}
\newcommand{\sym}[1]{\symmetricgroup{#1}}
\newcommand{\orbit}[2]{\mathrm{orbit}(#2, #1)}
\newcommand{\stab}[2]{\mathrm{stab}(#1, #2)}
\newcommand{\STAB}[2]{\mathrm{STAB}(#1, #2)}

\renewcommand{\P}{\ensuremath{\mathrm{P}}}

\newcommand{\coNP}{\ensuremath{\mathrm{coNP}}}

\newcommand{\eqp}[1]{=_{#1}}

\newcommand{\precp}[1]{\prec_{#1}}
\newcommand{\preceqp}[1]{\preceq_{#1}}
\newcommand{\succp}[1]{\succ_{#1}}
\newcommand{\succeqp}[1]{\succeq_{#1}}

\newcommand{\queue}{\mathcal{Q}}
\newcommand{\tree}{\mathcal{T}}

\newcommand{\eqconj}[2]{\ensuremath{E_{#1}^{#2}}}
\newcommand{\infconj}[2]{\ensuremath{D_{#1}^{#2}}}

\newcommand{\R}{\mathds{R}}
\newcommand{\B}[1]{\{0,1\}^{#1}}

\newcommand{\perm}{\gamma}
\newcommand{\group}{\Gamma}
\newcommand{\inv}[1]{{#1}^{-1}}
\newcommand{\pig}{\perm\in\group}

\newcommand{\define}{\coloneqq}
\newcommand{\card}[1]{|#1|}
\newcommand{\T}{^{\top}}
\newcommand{\sprod}[2]{{#1}\T {#2}}
\newcommand{\st}{:}
\newcommand{\solver}[1]{\texttt{#1}}
\newcommand{\scip}{\solver{SCIP}\xspace}

\theoremstyle{plain}
\newtheorem{theorem}{Theorem}[section]
\newtheorem{lemma}[theorem]{Lemma}
\newtheorem{proposition}[theorem]{Proposition}
\newtheorem{corollary}[theorem]{Corollary}
\newtheorem*{claim*}{Claim}
\newtheorem{claim}[theorem]{Claim}
\newtheorem{observation}[theorem]{Observation}
\theoremstyle{definition}
\newtheorem{remark}[theorem]{Remark}

\newtheorem{example}[theorem]{Example}
\newtheorem{property}[theorem]{Property}

\newenvironment{claimproof}
    {\begin{proof}}{ \end{proof}}

\usetikzlibrary{shapes.geometric}
\usetikzlibrary{arrows}

\tikzstyle{cond} = [draw, thick, diamond, inner sep=0pt, text
width=2mm*sqrt(2),
align=center]
\tikzstyle{root} = [draw, thick, regular polygon, regular polygon sides=4,
inner sep=0pt,
text width=2mm, align=center]
\tikzstyle{loos} = [draw, thick, regular polygon, regular polygon sides=4,
inner sep=0pt,
text width=2mm, align=center]
\tikzstyle{necc} = [draw, thick, circle, inner sep=0pt, text
width=2mm*sqrt(2),
align=center]
\tikzset{edge/.style = {thick, ->,> = latex'}}

\let\OLDthebibliography\thebibliography
\renewcommand\thebibliography[1]{
  \OLDthebibliography{#1}
  \setlength{\parskip}{.2em}
  \setlength{\itemsep}{.2em plus 0.3ex}
}

\def\hyph{-\penalty0\hskip0pt\relax}
\author{Jasper van Doornmalen}
\author{Christopher Hojny}
\title{Efficient Propagation Techniques for Handling Cyclic Symmetries in
  Binary Programs}
\affil{Combinatorial Optimization Group, Eindhoven University of Technology\\
  email: \{m.j.v.doornmalen, c.hojny\}@tue.nl}
\date{}

\begin{document}
\maketitle

\begin{abstract}
  The presence of
  symmetries of binary programs typically degrade the performance of
  branch-and-bound solvers.
  In this article, we derive efficient variable fixing algorithms to
  discard symmetric solutions from the search space based on propagation
  techniques for cyclic groups.
  Our algorithms come with the guarantee to find all possible variable
  fixings that can be derived from symmetry arguments, i.e., one cannot
  find more variable fixings than those found by our algorithms.
  Since every permutation symmetry group of a binary program has cyclic
  subgroups,  the derived algorithms can be used to handle symmetries in
  any symmetric binary program.
  In experiments we also provide numerical evidence that our
  algorithms handle symmetries more efficiently than other variable fixing
  algorithms for cyclic symmetries.%

\end{abstract}
\section{Introduction}

We consider binary programs~$\max\{ \sprod{c}{x} \st Ax \leq b,\;
x \in \B{n}\}$, with~$A \in \R^{m \times  n}$,
$b \in \R^m$, and~$c \in \R^n$ for some positive integers~$m$ and~$n$.
A standard method to solve binary programs is branch-and-bound,
which iteratively explores the search space by splitting the initial binary
program into subproblems, see Land and Doig~\cite{LandDoig1960}.
Although branch-and-bound can solve binary
programs with thousands of variables and constraints rather efficiently,
the performance of branch-and-bound usually degrades drastically if
symmetries are present because it unnecessarily explores symmetric
subproblems.
Such a \emph{symmetry} is a permutation~$\perm$ of~$[n] \define
\{1, \dots, n\}$ that acts on a vector~$x \in \R^n$ by permuting its
coordinates, i.e., $\perm(x) \define
(x_{\inv{\perm}(1)},\dots,x_{\inv{\perm}(n)})$, and that adheres to the
following two properties:
\begin{enumerate*}[label=\textbf{\emph{\small(\roman*)}},
ref=\emph{\small(\roman*)}]
\item it maps feasible solutions to feasible solutions, i.e., $Ax \leq b$
  if and only if~$A\perm(x) \leq b$, and
\item it preserves the objective value, i.e., $\sprod{c}{x} =
  \sprod{c}{\perm(x)}$.
\end{enumerate*}
Two solutions~$x$ and~$y$ are symmetric if there exists a symmetry~$\perm$
such that~$y = \perm(x)$.

Various methods to remove symmetric parts from the search space have
been proposed in the literature, ranging, among others, from cutting
planes, variable fixing or branching rules, or propagation methods, see
below for references.
The common ground of many of these methods is to impose a lexicographic
order on the solution space and to exclude solutions that are not
lexicographically maximal in their symmetry class.
This approach removes all symmetric copies of a solution, and thus, handles
all symmetries.
However, deciding whether a solution is lexicographically maximal in its
symmetry class is coNP-complete~\cite{babai1983canonical}.
This makes lexicographic order based methods often computationally
expensive, since no generally applicable polynomial-time algorithms for
such methods exist, unless~$\text{P} = \text{coNP}$.
For this reason, one typically weakens the requirement of removing all
symmetric copies, or investigates symmetry handling methods for particular
groups~$\group$ for which methods exist that run in polynomial time.

In this article, we follow the latter approach by investigating propagation
techniques, whose idea is as follows.
If we are given a subproblem, some of the variables might have been fixed,
e.g., due to branching decisions.
For a symmetry~$\perm$, a propagation algorithm looks for further
variables that need to be fixed to guarantee that a solution~$x$ that
adheres to the fixings of the subproblem is not lexicographically smaller
than the permuted vector~$\perm(x)$.
Of course, if we are given a set of permutations~$\Pi$, then this
propagation step can be carried out for every~$\perm \in \Pi$.
Since symmetry groups may have size~$2^{\Omega(n)}$, however,
blindly applying propagation for each individual permutation is
computationally intractable.

Although the full symmetric group has exponential order, polynomial time
propagation algorithms for certain actions of full symmetric groups have
been designed~\cite{BendottiEtAl2021,KaibelEtAl2011}.
To the best of our knowledge, however, it seems that no efficient
propagation algorithms for cyclic groups, i.e., groups generated by a
single permutation, are known.
At first glance, finding algorithms for cyclic groups seems to be trivial
as cyclic shifts have a very simple structure.
But despite the simplicity of cyclic shifts, we have no understanding of
the structure of  binary points being lexicographically maximal with
respect to cyclic group actions.
In fact, the structure of these points is rather complicated and does not
seem to follow an obvious pattern, see~\cite[Chap.~3.2.2]{Loos2010}.
It has been an open problem for at least ten years to gain further insights
into the structure of lexicographically maximal points for cyclic groups.

We believe that this is an important gap, because every permutation
group~$\group$ has cyclic subgroups.
Thus, instead of applying propagation for individual permutations, we can
apply propagation for corresponding cyclic subgroups to find stronger
reductions.
In particular, knowledge on cyclic groups can be used for \emph{every}
symmetric binary program, whereas algorithms for symmetric groups need
assumptions on~$\group$.
We emphasize that, although cyclic groups~$\group$ are generated by a single
permutation~$\perm$, they also might have superpolynomial size%
\footnote{For example, if~$\perm$ has~$k$ disjoint cycles of mutually
distinct prime lengths.}.
That is, efficient algorithms for cyclic groups are not immediate.

\paragraph{Literature Review}
Handling symmetries in binary programs via propagation is not a novel
technique.
It originates from constraint programming, and symmetry handling
techniques in this context are discussed, among others,
in~\cite{CrawfordGinsbergLuksRoy1996,KatsirelosNarodytskaWalsh2009,KatsirelosEtAl2010,NarodytskaWalsh2013}.
For binary programs,
Bendotti et al.~\cite{BendottiEtAl2021} describe an algorithm to find
variable fixings for certain actions of symmetric groups.
Further fixings can be found if the variables affected by the symmetric
group are contained in set packing or partitioning constraints, see Kaibel
et al.~\cite{KaibelEtAl2011}.
Moreover, if instead of an entire group only the action of a single
permutation is considered, propagation algorithms for so-called symresacks
can be used~\cite{BestuzhevaEtal2021OO,HojnyPfetsch2019}.
These algorithms are \emph{complete}
in the sense that they find all possible
symmetry-based variable fixings derivable from a set of fixed variables.
In contrast to this, orbital fixing~\cite{Margot2003,OstrowskiEtAl2011}
can be used for arbitrary groups, however, without any guarantee on
completeness.
Margot~\cite{Margot2002,Margot2003} presents isomorphism pruning, a
propagation technique to prune nodes of a branch-and-bound tree that do
not contain lexicographically maximal solutions.

Besides propagation, further methods for handling symmetries in binary
programs exist such as cutting
planes~\cite{Friedman2007,Hojny2020,HojnyPfetsch2019,KaibelPfetsch2008,Liberti2008,Liberti2012a,Liberti2012,LibertiOstrowski2014,Salvagnin2018},
branching rules~\cite{OstrowskiAnjosVannelli2015,OstrowskiEtAl2011},
or model reformulations~\cite{FischettiLiberti2012}.

\paragraph{Contribution}
The aim of this article is to devise efficient propagation algorithms for
cyclic groups.
To this end, we derive an auxiliary result for arbitrary sets~$\Pi$ of
permutations first.
As mentioned above, we can find variable fixings by propagating each
individual permutation~$\perm \in \Pi$ using symresack propagation.
As mentioned in~\cite{BestuzhevaEtal2021OO}, a single symresack can be
propagated in~$O(n)$ time.
Thus, we can
find all fixings derived from all individual permutations
in~$\bigo(n^2\card{\Pi})$ time, because there
are~$n$ potential fixings and each might trigger another round of
propagating~$\Pi$.
In Section~\ref{sec:completepermset}, we improve this running time
to~$\bigo(n\card{\Pi})$ by introducing suitable data structures and a
careful analysis of dependencies among the permutations in~$\Pi$.
This result forms the basis for our efficient algorithms for cyclic groups
that we derive in Section~\ref{sec:lexleadersgroup}.
To this end, we provide a novel characterization of lexicographically
maximal elements in certain cyclic groups.
This characterization is then used to derive our efficient algorithms for a
broad class of cyclic groups.
We in particular show that our algorithms find all possible variable
fixings, i.e., they are as strong as possible.
In Section~\ref{sec:computational}, we report on numerical results on a
broad test set containing both instances with many cyclic symmetries and
general benchmark instances.
If cyclic symmetries are the dominant type of symmetries, these experiments
show that our methods outperform the immediate approach of handling all
permutations individually.
For ease of presentation, we skip most proofs in the discussion of
Section~\ref{sec:completepermset}; the missing proofs are then provided in
Appendix~\ref{sec:proofs}.

\paragraph{Basic Definitions and Notation}
Throughout this article, we assume that~$n$ is a positive integer.
Given~$k \in [n+1] \define \{ 1, \dots, n + 1 \}$
and vectors~$x, y \in \R^n$, we say that~$x =_k y$ if
and only if~$x_i = y_i$ for all~$i \in [k-1]$.
To decide whether~$x$ and~$y$ can be distinguished \emph{up to position~$k$},
we write~$x \succ_k y$ if and only if there exists~$i \in [k-1]$
such that~$x =_i y$ and~$x_i > y_i$.
The relation~$x \succeq_k y$ holds if~$x =_k y$ or~$x \succ_k y$.
These relations define the \emph{partial lexicographic order up to $k$}.
When~$k = n+1$, we write~$=$, $\succ$, and~$\succeq$ instead of~$=_k$,
$\succ_k$, and~$\succeq_k$, respectively.
In this case, we say that~$x$ is \emph{equal to}, \emph{lexicographically
  greater}, and \emph{lexicographically not smaller} than~$y$,
respectively.

Let~$\sym{n}$ be the \emph{symmetric group} on~$[n]$.
For~$\perm \in \sym{n}$, the set of all binary vectors that are
lexicographically not smaller than their images~$\perm(x)$ is denoted by~$\xymresack{\perm} \define \{x \in \B{n}
\st x \succeq \perm(x)\}$.
Moreover, for~$\Pi \subseteq \symmetricgroup{n}$, denote
$\xymretope{\Pi} \define \bigcap_{\perm \in \Pi}
\xymresack{\perm}$.
Analogously, we define~$\xymresack{\perm}^{(k)}$ and~$\xymretope{\Pi}^{(k)}$
if we use the relation~$\succeq_k$ instead of~$\succeq$.
If the set~$\Pi$ defines a group, we typically use the
symbol~$\Gamma$ to denote this.

If the generating permutations of $\Gamma$ are $\gamma_1, \dots, \gamma_m$
for some $m \in \mathbb N$, then this is denoted
with angle brackets~$\Gamma \define \gen{\gamma_1, \dots, \gamma_m}$.
If permutations are defined explicitly, we always use the cycle
representation.
For disjoint sets~$I_0, I_1 \subseteq [n]$, we define~%
$F(I_0, I_1) \define \{ x \in \binary^n : x_i = 0\ \text{for}\ i \in I_0\
\text{and}\ x_i = 1\ \text{for}\ i \in I_1 \}$.
The sets~$I_0$ and~$I_1$ thus define the indices of binary (solution)
vectors that are fixed to~0 and~1, respectively.
The situation where the entry~$x_i$, $i \in [n]$, of a vector~$x$
is fixed to a value~$b \in \binary$ is called a \emph{fixing},
and we denote this by a tuple~$f = (i, b) \in [n] \times \binary$.
By a slight abuse of terminology, we say in the following that entry~$i$
gets fixed rather than entry~$x_i$ of vector~$x$ to keep notation short.
The \emph{converse fixing} of $f = (i, b)$
is denoted by~$\bar f \define (i, 1-b)$.
A set of fixings $C \subseteq [n] \times \binary$
is called a \emph{conjunction},
and we define~%
$
V(C)
\define
\{
x \in \binary^n
:
x_i = b \
\text{for}\
(i, b) \in C
\}
$
as the set of binary vectors respecting the fixings in conjunction~$C$.

To handle symmetries, the main goal of this article is to find, given a set
of initial fixings~$I_0$ and~$I_1$ larger sets~$I'_0$ and~$I'_1$
with~$\xymretope{\Pi} \cap F(I'_0, I'_1) = \xymretope{\Pi} \cap F(I_0,
I_1)$.
Once we have identified such sets~$I'_0$ and~$I'_1$, the variables in~$I'_0
\setminus I_0$ and~$I'_1 \setminus I_1$ can be fixed to~0 and~1,
respectively.
Thus, we can derive variable fixings based on symmetry considerations.
To obtain the strongest effect, we are interested in sets~$I'_0$
and~$I'_1$ being as large as possible.
Note that the largest sets, denoted~$I^\star_0$ and~$I^\star_1$, are
unique:
Let~$C$ be the conjunction representing the fixings~$(I_0,I_1)$.
Let~$X_\Pi(C) \define \{ x \in \xymretope\Pi :
x_i = b\ \text{for}\ (i, b) \in C \}$
be the set of binary vectors in $\xymretope\Pi$ respecting the
fixings in~$C$.
For a set~$X \subseteq \binary^n$, let~$\mathcal{C}(X)
\define \{ (i, b) \in [n] \times \binary :
x_i = b\ \text{for all}\ x \in X \}$ be the set of fixings in~$X$.
Then~$\mathcal{C}(X_\psi(C))$ is the unique largest conjunction of fixings for
initial fixings $C$ with respect to $\xymretope\Pi$, from which we
derive~$I^\star_0$ and~$I^\star_1$.

For a subgroup~$\group$ of~$\sym{n}$, denoted~$\group \leq \sym{n}$, we
frequently use stabilizer subgroups.
Given a set~$I \subseteq [n]$, the \emph{pointwise stabilizer}
is~$\STAB{I}{\group} \define \{ \pig \st \perm(i) = i,\; i \in I\}$.
The \emph{setwise stabilizer} is~$\stab{I}{\group} \define \{ \pig \st
\perm(i) \in I,\; i \in I\}$.
For singleton sets, we write~$\STAB{i}{\group}$ and~$\stab{i}{\group}$
instead of~$\STAB{\{i\}}{\group}$ and~$\stab{\{i\}}{\group}$, respectively.
The \emph{orbit} of a solution~$x$ with respect to a group~$\group$
is~$\{\perm(x) \st \pig\}$.
Last, for a permutation~$\gamma \in \symmetricgroup{n}$,
the \emph{restriction of $\gamma$ to $I$},~$\delta = \restr(\gamma, I)$,
satisfies $\delta(i) = \gamma(i)$ for $i \in I$
and $\delta(i) = i$ for $i \notin I$.
For groups $\Gamma \leq \symmetricgroup{n}$, we denote
$\restr(\Gamma, I) \define \{ \restr(\gamma, I) : \gamma \in \Gamma \}$.
Note that~$\delta \in \sym{n}$ if and only if~$\gamma(I) = I$,
and that~$\restr(\Gamma, I) \leq \sym{n}$
if and only if~$I$ corresponds to the union of orbits of elements from $I$.

\section{Propagation of Individual Permutations In a Set}
\label{sec:completepermset}

The main goal of this article is to devise efficient propagation algorithms
that enforce a solution to be lexicographically maximal in its orbit
with respect to a cyclic group.
As we will see in the next section, the main workhorse of these algorithms
is an efficient subroutine that, for a given set of permutations~$\Pi$,
propagates~$x \succeq \perm(x)$ for all~$\perm \in \Pi$.
To make this statement precise, we introduce the following terminology and
notation.

Let $\Pi \subseteq \symmetricgroup{n}$,
and $I_0, I_1 \subseteq [n]$ be disjoint.
Our aim is to find larger
sets $I_0', I_1'$
with
$
\xymretope\Pi \cap F(I_0, I_1)
=
\xymretope\Pi \cap F(I_0', I_1')
$
by iteratively \emph{applying valid fixings}
for the constraints $x \succeq \gamma(x)$
for each $\gamma \in \Pi$.
A \emph{fixing} is a tuple $(i, b) \in [n] \times \binary$
that encodes the situation where the value of entry~$i$ is fixed to~$b$.
We say that a fixing is \emph{valid} for a permutation~$\gamma \in \Pi$
and a disjoint set of fixings $I_0, I_1 \subseteq [n]$
if all $x \in \xymretope\gamma \cap F(I_0, I_1)$
have~$x_i = b$.
Such a fixing $(i, b)$ is \emph{applied} if entry $i$ is added to the
index set $I_b$.
This way, the simple Observation~\ref{obs:fixing} below shows how
additional fixings can be found.
If no further valid fixing can be found by considering any individual
constraint $x \succeq \gamma(x)$ for $\gamma \in \Pi$,
then this is a
\emph{complete set of fixings for each permutation in $\Pi$},
denoted by $I_0', I_1'$.
We want to stress that these do not need to be the complete set of fixings
for~$\xymretope\Pi \cap F(I_0, I_1)$: more fixings could exist, as we will
demonstrate in Example~\ref{ex:indivpropnotcomplete}.

Using this terminology, this section's goal is to find an efficient
algorithm to determine the complete set of fixings for all~$\perm \in
\Pi$.
As mentioned in the introduction, a trivial running time of such an
algorithm is~$\bigo(n^2 \card{\Pi})$.
By introducing suitable data structures and implications among the
different permutations in~$\Pi$, however, we show that the running time can
be reduced to~$\bigo(n \card{\Pi})$.
To develop our algorithm, we make use of the following simple observation.

\begin{observation}
  \label{obs:fixing}
  Let~$\Pi \subseteq \sym{n}$ and~$I_0, I_1 \subseteq [n]$ be disjoint.
  Suppose we want to propagate~$x \succeq \perm(x)$ for all~$\perm \in
  \Pi$.
  Then, $i \in [n] \setminus (I_0 \cup I_1)$ can be added to~$I_0$
  (resp.~$I_1$) if and only if every~$x \in F(I_0, I_1)$ with~$x_i = 1$
  (resp.~$x_i = 0$) satisfies~$x \prec \perm(x)$ for some~$\perm \in \Pi$.
\end{observation}

Consequently, if~$\xymretope{\Pi} \cap F(I_0, I_1) \neq \emptyset$
and~$\xymretope{\Pi} \cap F(I_0 \cup \{i\}, I_1) = \emptyset$, we know
that~$i$ can be fixed to~1 (and analogously we can argue for 0-fixings).
Since adding~$i$ to~$I_0$ makes the latter set empty, we refer to
such a fixing as an \emph{infeasibility} fixing.
To algorithmically exploit Observation~\ref{obs:fixing}, we are thus
interested in finding infeasibility fixings~$(i,b)$ as~$(i, 1-b)$ is then a
valid fixing.
For our algorithm, it will turn out that also considering sets of fixings that
lead to infeasibility, if applied simultaneously, are of importance.
As mentioned in the introduction, these sets are referred to as
\emph{conjunctions}.
\emph{Inf-conjunctions} are sets of fixings that yield infeasibility if
all fixings of the set are applied.
Moreover, we specify special types of inf-conjunctions.
Let~$k \in [n + 1]$ and $x \in \binary^n$.
Note that~$x \precp{k} \gamma(x)$ implies~$x \prec \gamma(x)$,
and that equivalence holds if $k = n + 1$.
A \emph{$k$-inf-conjunction} is a conjunction
$C \subseteq [n] \times \binary$
such that
all $x \in F(I_0, I_1)$
with $x_i = b$ for $(i, b) \in C$
have $x \precp{k} \gamma(x)$.
Note that $C$ is also an inf-conjunction for $\gamma$.

\begin{algorithm}[!tbp]
\caption{Determine the complete set of fixings for each individual
constraint $x \succeq \gamma(x)$ for all $\gamma \in \Pi$.}
\label{alg:propagateallindividualcons}
\SetKwInOut{Input}{input}
\SetKwInOut{Output}{output}
\SetKwBlock{RepeatWithoutUntil}{repeat}{}
\Input{set of permutations $\Pi \subseteq \symmetricgroup{n}$,
and initial set of fixings $(I_0, I_1)$}
\Output{\Infeasible if an empty inf-conjunction for some~$\gamma \in \Pi$
is found by the algorithm,
or \Feasible and the set of fixings that is complete for each individual
permutation in $\Pi$.}
\lIf{$F(I_0, I_1) = \emptyset$}
{
  \Return\Infeasible
}
\vspace{.5em}
$t \gets 0$;\
$(I_0^t, I_1^t) \gets (I_0, I_1)$\;
\lForEach{$\gamma \in \Pi$}{
  $i_\gamma \gets 1$%
}
\vspace{.5em}
\While{there is a $\gamma \in \Pi$
not satisfying sufficient conditions for completeness}
{
  \label{alg:propagateallindividualcons:select}
  $i_\gamma \gets i_\gamma + 1$;\quad
  $t \gets t + 1$\;
  \label{alg:propagateallindividualcons:indexincr}
  \RepeatWithoutUntil{
    \lIf{there is a $\delta \in \Pi$
      with $i_\delta$-inf-conjunction $\emptyset$}
    {
      \Return \Infeasible
    }
    \label{alg:propagateallindividualcons:inf}
    \ElseIf{
      there is a $\delta \in \Pi$ with
      $i_\delta$-inf-conjunction $\{(i, b)\}$ and $i \notin I_{1-b}^t$
    }
    {
      \label{alg:propagateallindividualcons:whilefix}
      Apply fixing $(i, 1-b)$:
      $
      (I_b^{t + 1}, I_{1-b}^{t + 1})
      \gets
      (I_b^{t}, I_{1-b}^{t} \cup \{ i \})
      $;\quad
      $t \gets t + 1$\;
      \label{alg:propagateallindividualcons:varfixing}
    }
    \lElse
    {
      \Break\ repeat-loop%
      \label{alg:propagateallindividualcons:break}
    }
  }
}
\vspace{.5em}
\Return \Feasible, $(I_0^t, I_1^t)$\;
\end{algorithm}

Algorithm~\ref{alg:propagateallindividualcons} describes how additional
fixings can be found.
To simplify the analysis,
we maintain a timestamp~$t$, starting at~0.
Also, for each permutation~$\gamma \in \Pi$,
the index until which the partial lexicographic is considered
is~$i_\gamma$, which is initialized at~1.
If a time-specification is needed,
the value of $i_\gamma$ at time~$t$ is denoted by $i_\gamma^t$.
The set of fixings at this time is denoted by~$I_0^t$ and~$I_1^t$.
The idea of our algorithm
is to iterate over permutations from $\Pi$ for which we can potentially
find further variable fixings.
It checks whether there exists a permutation~$\perm$ in this list that
admits an inf-conjunction consisting of at most a single element:
If there is an empty inf-conjunction for~$\perm$, then~$\xymresack{\perm}
\cap F(I_0, I_1) = \emptyset$ and the algorithm terminates since
infeasibility has been determined.
Otherwise, for all inf-conjunctions~$\{(i,b)\}$ that can be found for one
permutation in the list, the algorithm applies the fixing~$(i, 1-b)$.
To be able to find inf-conjunctions efficiently, the algorithm does not
check for the existence of arbitrary inf-conjunctions.
Instead, only inf-conjunctions are checked that certify infeasibility for a
partial lexicographic order.
To make this precise, we introduce the following terminology.

Note that Algorithm~\ref{alg:propagateallindividualcons} is not
practically applicable yet, because it does not specify details on how to
execute it.
In the remainder of this section, we provide these missing details.
In particular, we derive structural properties of inf-conjunctions and
develop efficient data structures that allow us to execute the algorithm
in~$\bigo(n \card\Pi)$ time.
Before doing so, we provide an example that illustrates the execution of
this algorithm, and prove that this algorithm is correct if it terminates.

\begin{example}
\label{ex:algexample}
Let~$\gamma_1 = (1, 6, 8, 4, 7, 2, 5)$, $\gamma_2 = (1, 3, 6, 2, 4, 5)$,
$\Pi = \{ \gamma_1, \gamma_2 \}$,
and let the initial fixings be~$I_0 = \{ 4, 6 \}$ and~$I_1 = \{ 5 \}$
encoded by~$x = (\blank, \blank, \blank, 0, 1, 0, \blank, \blank)$,
where~$\blank$ represents an unfixed entry. We execute a few steps of the
algorithm, and the fixing updates are shown in
Figure~\ref{fig:ex:algexample}.
More precisely, we discuss which permutations are selected at each
iteration in Line~\ref{alg:propagateallindividualcons:select},
and which of the cases of
Lines~\ref{alg:propagateallindividualcons:inf}--%
\ref{alg:propagateallindividualcons:break} applies.
Later we specify how the selection conditions work algorithmically, and
how $k$-inf-conjunctions can be detected.

In the first iteration, we select $\gamma_1$, set $i_{\gamma_1} \gets 2$.
There is a $2$-inf-conjunction $\{ (1, 0) \}$ for $\gamma_1$, since
choosing $x_1 \gets 0$ yields $x \precp{2} \gamma_1(x)$.
Hence, we apply fixing $(1, 1)$, which fixes entry $1$ to value 1.
Any remaining $i_\delta$-inf-conjunction for $\delta \in \Pi$ needs at
least two fixings, so we continue with the next iteration.
Again, select $\gamma_1$ and set $i_{\gamma_1} \gets 3$.
Since $x_2, x_7$ are both unfixed, no $3$-inf-conjunction of cardinality
less than two exists.
Set~$i_{\gamma_1} \gets 4$, and we encounter a fixed point $3$.
Set~$i_{\gamma_1} \gets 5$, we have~$(x_4, x_8) = (0, \blank)$.
If the value of~$x_2$ and~$x_7$ is the same, then $x_8$ must become 0, as
well.
Set $i_{\gamma_1} \gets 6$, we encounter~$(x_5, x_2) = (1, \blank)$.
In this case, if the columns $x$ and $\gamma_1(x)$ are equal up to entry
$5$, and $x_2 = 0$,
then no $6$-inf-conjunction for $\gamma_1$ with cardinality 1 can be found.
Otherwise, if $x_2 = 1$, we can continue.
Choose $i_{\gamma_1} \gets 7$. Then, $(x_6, x_1) = (0, 1)$,
which means that $x \precp{7} \gamma_1(x)$ if for all entries $i < 6$ we
have that the value of $x_i$ is the same as~$\gamma_1(x)_i$.
If~$x_7 = 1$, to ensure~$x \succeqp{6} \gamma_1(x)$,
we must have $x_2 = 1$ and $x_8 = 0$,
but in that case~$x \precp{7} \gamma(x)$, so~$\{(7, 1)\}$ is a
7-inf-conjunction for~$\gamma_1$.
Hence, apply fixing $(7, 0)$.

Similar steps can be applied to permutation $\gamma_2$,
but no further fixings can be deduced. Namely, if $x_2 \gets 0$ then $x_3
\gets 1$ and we find $x \succp{6} \gamma_2(x)$, and if $x_2 \gets 1$ then
$x \succp{3} \gamma_2(x)$.

\begin{figure}[!tbp]
\begin{equation*}
\makeatletter\setlength\BA@colsep{3.2pt}\makeatother
\begin{blockarray}{*{3}{cc}}
  \begin{block}{*{3}{cc}}
    \BAmulticolumn{2}{>{$\footnotesize}c<{$}}{$x$} &
    \BAmulticolumn{2}{>{$\footnotesize}c<{$}}{$\gamma_1(x)$} &
    \BAmulticolumn{2}{>{$\footnotesize}c<{$}}{$\gamma_2(x)$} \\
  \end{block}
  \begin{block}{[*{3}{>{$\tiny}l<{\hspace{-1em}$}r}]}
    (1) & \blank & (5) &      1 & (5) &      1 \\
    (2) & \blank & (7) & \blank & (6) &      0 \\
    (3) & \blank & (3) & \blank & (1) & \blank \\
    (4) &      0 & (8) & \blank & (2) & \blank \\
    (5) &      1 & (2) & \blank & (4) &      0 \\
    (6) &      0 & (1) & \blank & (3) & \blank \\
    (7) & \blank & (4) &      0 & (7) & \blank \\
    (8) & \blank & (6) &      0 & (8) & \blank \\
  \end{block}
\end{blockarray}
\overset{
\begin{subarray}{c}
x_1 \gets 1
\end{subarray}}{\leadsto}
\begin{blockarray}{*{3}{cc}}
  \begin{block}{*{3}{cc}}
    \BAmulticolumn{2}{>{$\footnotesize}c<{$}}{$x$} &
    \BAmulticolumn{2}{>{$\footnotesize}c<{$}}{$\gamma_1(x)$} &
    \BAmulticolumn{2}{>{$\footnotesize}c<{$}}{$\gamma_2(x)$} \\
  \end{block}
  \begin{block}{[*{3}{>{$\tiny}l<{\hspace{-1em}$}r}]}
    (1) &      1 & (5) &      1 & (5) &      1 \\
    (2) & \blank & (7) & \blank & (6) &      0 \\
    (3) & \blank & (3) & \blank & (1) &      1 \\
    (4) &      0 & (8) & \blank & (2) & \blank \\
    (5) &      1 & (2) & \blank & (4) &      0 \\
    (6) &      0 & (1) &      1 & (3) & \blank \\
    (7) & \blank & (4) &      0 & (7) & \blank \\
    (8) & \blank & (6) &      0 & (8) & \blank \\
  \end{block}
\end{blockarray}
\overset{
\begin{subarray}{c}
x_7 \gets 0
\end{subarray}}{\leadsto}
\begin{blockarray}{*{3}{cc}}
  \begin{block}{*{3}{cc}}
    \BAmulticolumn{2}{>{$\footnotesize}c<{$}}{$x$} &
    \BAmulticolumn{2}{>{$\footnotesize}c<{$}}{$\gamma_1(x)$} &
    \BAmulticolumn{2}{>{$\footnotesize}c<{$}}{$\gamma_2(x)$} \\
  \end{block}
  \begin{block}{[*{3}{>{$\tiny}l<{\hspace{-1em}$}r}]}
    (1) &      1 & (5) &      1 & (5) &      1 \\
    (2) & \blank & (7) &      0 & (6) &      0 \\
    (3) & \blank & (3) & \blank & (1) &      1 \\
    (4) &      0 & (8) & \blank & (2) & \blank \\
    (5) &      1 & (2) & \blank & (4) &      0 \\
    (6) &      0 & (1) &      1 & (3) & \blank \\
    (7) &      0 & (4) &      0 & (7) &      0 \\
    (8) & \blank & (6) &      0 & (8) & \blank \\
  \end{block}
\end{blockarray}
\end{equation*}
\caption{Figure for Example~\ref{ex:algexample}.
Each matrix shows the state of the algorithm at a point.
The columns of the matrix correspond to the vectors $x$, and the permuted
vectors $\gamma_1(x)$ and $\gamma_2(x)$, and for each entry in the matrix
the variable index is written left to it. The left matrix is the initial
state. The second matrix is after application of $x_1 \gets 1$, and the
third matrix is after application of $x_7 \gets 0$.}
\label{fig:ex:algexample}
\end{figure}%
\end{example}

We next show that Algorithm~\ref{alg:propagateallindividualcons} indeed
finds the complete set of fixings.

\begin{lemma}
If Algorithm~\ref{alg:propagateallindividualcons} terminates,
then it correctly detects infeasibility,
or finds the complete set of fixings for each permutation in $\Pi$.
\end{lemma}
\begin{proof}
First of all, we show that the algorithm never adds incorrect fixings.
If the algorithm applies a fixing $(i, b)$, then this fixing is due to
Line~\ref{alg:propagateallindividualcons:varfixing},
so there is a $\delta \in \Pi$
such that~$\{ (i, 1-b) \}$ is an $i_\delta$-inf-conjunction.
Hence, all~$x \in \xymretope\delta^{(i_\delta)} \cap F(I_0^t, I_1^t)
\supseteq \xymretope\delta \cap F(I_0^t, I_1^t)$
have~$x_i = b$, showing that~$(i, b)$ is a valid fixing.

If feasibility is returned at time~$t$,
then the set of fixings $(I_0^t, I_1^t)$
satisfies sufficient conditions for completeness for all
permutations~$\gamma \in \Pi$.
Together with the first paragraph of this proof,
$(I_0^t, I_1^t)$ defines a complete set of fixings for all permutations
$\gamma \in \Pi$, and $
\xymretope\Pi \cap F(I_0, I_1) =
\xymretope\Pi \cap F(I_0^t, I_1^t)
$.

Otherwise infeasibility is returned.
If this is due to the first line, then there does not exist any vector
satisfying the initial sets of fixings $I_0, I_1$, so no valid solution
exists for any $\gamma \in \Pi$.
Otherwise it is due to Line~\ref{alg:propagateallindividualcons:inf}
at time~$t$.
Then the empty set is a valid $i_\delta$-inf-conjunction
for some $\delta \in \Pi$.
Therefore, the empty set is an inf-conjunction,
and thus infeasibility is found without applying any fixing.
Hence, infeasibility is correctly returned.
\end{proof}

We now proceed with the missing detail to turn
Algorithm~\ref{alg:propagateallindividualcons} into a practically
applicable algorithm.
A crucial step of Algorithm~\ref{alg:propagateallindividualcons} is to
identify $k$-inf-conjunctions. To this end,
recall that~$x \prec \perm(x)$ if and only if there exists~$j \in [n]$ such
that~$x_i = \perm(x)_i$ for all~$i \in {[j-1]}$, and~$x_j = 0$ as well
as~$\perm(x)_j = 1$.
To generate $k$-inf-conjunctions, we thus also introduce conjunctions
ensuring equality up to a certain index.
A \emph{$k$-eq-conjunction}
is a conjunction $C \subseteq [n] \times \binary$
such that, for all $x \in F(I_0, I_1)$
with $x_i = b$ for $(i, b) \in C$,
we have ${x \preceqp{k} \gamma(x)}$,
and $x \eqp{k} \gamma(x)$ holds at least for one such vector.
We call it a $k$-eq-conjunction, because all vectors $x \in F(I_0, I_1)$
satisfying the conjunction~$C$ in~$\xymretope\gamma$
(i.e.,~$x \in \xymretope\gamma \cap F(I_0, I_1) \cap V(C)$)
satisfy~$x \eqp{k} \gamma(x)$.
We use such $k$-eq-conjunctions of a permutation to
determine~$(k+1)$-inf-conjunctions,
by exploiting that~$x \preceqp{k} \gamma(x)$
for this conjunction, and additionally ensuring~$x_k < \gamma(x)_k$.
This way, all vectors $x \in F(I_0, I_1)$ respecting the new conjunction
either have that $x \precp{k} \gamma(x)$,
or that~$x \eqp{k} \gamma(x)$ and $x_k < \gamma(x)_k$.
In either case~$x \precp{k + 1} \gamma(x)$ holds,
so the new conjunction is a $(k+1)$-inf-conjunction.
For each permutation, denote the set of
inclusionwise minimal $i_\gamma$-inf-conjunctions
at time~$t$
by~$\infconj{\gamma}{t}$,
and,
likewise,
the set of inclusionwise minimal $i_\gamma$-eq-conjunctions
by~$\eqconj{\gamma}{t}$.
Note that considering minimal conjunctions is sufficient to guarantee
correctness of Algorithm~\ref{alg:propagateallindividualcons}.

To efficiently store and update all information needed for
Algorithm~\ref{alg:propagateallindividualcons},
we encode $\infconj{\gamma}{t}$ and $\eqconj{\gamma}{t}$
by a tree structure.
For each permutation $\gamma \in \Pi$ and time~$t$,
we define an \emph{implication tree},
which is a directed rooted tree~%
$\tree_\gamma^t = (V_\gamma^t, A_\gamma^t)$
with four types of vertices that partition $V_\gamma^t$:
$V_{\mathrm{root}, \gamma}^t \define \{ r_\gamma \}$,
$V_{\mathrm{cond}, \gamma}^t$,
$V_{\mathrm{necc}, \gamma}^t$, and~%
$V_{\mathrm{loose}, \gamma}^t$
for
the set with the root vertex $r_\gamma$,
the \emph{conditional fixing vertices},
the \emph{necessary fixing vertices},
and the \emph{loose end vertices},
respectively.
Each \emph{fixing vertex}~%
$v \in V_{\mathrm{cond}, \gamma}^t \cup V_{\mathrm{necc}, \gamma}^t$
has an associated fixing, denoted by $f_v \in [n] \times \binary$.
We call this structure an implication tree,
because it encodes conjunctions by implications of the type ``if a
set of fixings is applied, then also another fixing must be applied''.
When walking along a directed rooted path,
the vertices in this path describe (dependent) implications as follows:
If we encounter a necessary fixing vertex, then the
associated fixing must be applied. If we
encounter a conditional fixing vertex, then we continue following the path
only if that fixing has already been applied.
Added to this, if we encounter a loose end vertex on our walk, then all
previously applied fixings ensure that~$x \eqp{i_\gamma} \gamma(x)$
is found for all solution vectors~$x$ that respect the given fixings, and
the fixings on the loose end vertex.
When illustrating (parts of) implication trees,
we draw conditional fixing vertices as diamonds, necessary fixing
vertices as circles, loose end vertices as squares,
and no outline is used in case of the root vertex or if the type is
not important for the illustration.
For fixing vertices, its fixings are written next to the vertex.

In Example~\ref{ex:algexample}, after applying the fixing $x_1 \gets 1$,
we have the following implications for $\gamma_1$:
If we fix entry~$2$ to zero, then entries~$7$ and~$8$ must be fixed to
zero, as well. Likewise, if we fix entry~$7$ to one,
then entry~$2$ must be one,
and entry~$8$ must be zero,
and we find~$x \eqp{6} \gamma(x)$ for all solution vectors respecting
these fixings.
The associated implication tree is shown in
Figure~\ref{fig:ex:exampletree}.

\begin{figure}[!tbp]
\centering
\begin{tikzpicture}[x=12mm, y=3mm, font=\footnotesize]
\node[] (root) at (0, 0) {$r_{\gamma_1}$};

\node[cond, label=above:{$(2, 0)$}] (20) at (1, 1) {};
\node[necc, label=above:{$(7, 0)$}] (70) at (2, 1) {};
\node[necc, label=above:{$(8, 0)$}] (80) at (3, 1) {};

\node[cond, label=below:{$(7, 1)$}] (71) at (1, -1) {};
\node[necc, label=below:{$(2, 1)$}] (21) at (2, -1) {};
\node[necc, label=below:{$(8, 0)$}] (80p) at (3, -1) {};
\node[loos] (loos) at (4, -1) {};

\draw[edge] (root) to (20);
\draw[edge] (20) to (70);
\draw[edge] (70) to (80);

\draw[edge] (root) to (71);
\draw[edge] (71) to (21);
\draw[edge] (21) to (80p);
\draw[edge] (80p) to (loos);
\end{tikzpicture}
\caption{Implication tree of permutation $\gamma_1$ in
Example~\ref{ex:algexample} with $i_{\gamma_1} = 6$ and $x_1 = 1$.}
\label{fig:ex:exampletree}
\end{figure}
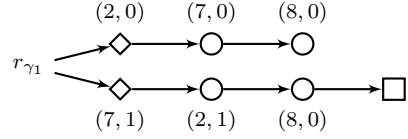

Note that symmetry-based implications can always be encoded by \emph{some}
implication tree.
In the following, however, we will see that, if using the right encoding
and update strategies, the implication trees have a \emph{particular}
structure.
This structure will allow us to implement
Algorithm~\ref{alg:propagateallindividualcons} efficiently.
We first list properties on the relation between our
particular implication trees and the sets~$\infconj{\gamma}{t}$
and~$\eqconj{\gamma}{t}$.
Moreover, we give some properties of the implication trees that are
maintained by the algorithm.
We prove that the relations and properties are maintained using induction.
We also show how Algorithm~\ref{alg:propagateallindividualcons} can
be executed, and how the data structures are updated consistently such
that the mentioned properties indeed hold.
Last, the running time of these methods is analyzed.

\begin{property}[Relation between~$\tree_\gamma^t$ and
~$\eqconj{\gamma}{t}$]
\label{prop:eqconj}
Each eq-conjunction $C \in \eqconj{\gamma}{t}$
corresponds to a loose end vertex $v \in V_{\mathrm{loose}, \gamma}^t$ and
vice versa.
Given a loose end vertex $v \in V_{\mathrm{loose}, \gamma}^t$,
the fixings of its conditional ancestors
correspond to a minimal $i_\gamma^t$-eq-conjunction.
Symbolically,
let $\ancestors(v)$
yield all (improper) ancestors of a vertex~$v$
(i.e., including $v$ itself).
For $v \in V_{\mathrm{loose}, \gamma}^t$ then
$\{ f_u : u \in \ancestors(v) \cap V_{\mathrm{cond}, \gamma}^t \} \in
\eqconj{\gamma}{t}$.
This is a bijection, so any $C \in \eqconj{\gamma}{t}$
has a corresponding loose end vertex yielding this eq-conjunction.
\end{property}

We say that a conjunction~$C \subseteq [n] \times \binary$
is \emph{incompatible} with a set of fixings~$(I_0, I_1)$
if no vector exists that satisfies the fixings of~$(I_0, I_1)$ and~$C$,
i.e., $F(I_0, I_1) \cap V(C) = \emptyset$.
We distinguish two types of inf-conjunctions: conjunctions
$C \in \infconj{\gamma}{t}$
that are incompatible with the fixings~$(I_0^t, I_1^t)$
and those that are not.
\begin{property}[Relation between~$\tree_\gamma^t$ and
~$\infconj{\gamma}{t}$]
\label{prop:infconj}
For $F(I_0^t, I_1^t) \neq \emptyset$,
each~$C \in \infconj{\gamma}{t}$ that is not incompatible with the fixings
corresponds to a necessary fixing vertex~%
$v \in V_{\mathrm{necc}, \gamma}^t$
and vice versa.
For~$v \in V_{\mathrm{necc}, \gamma}^t$,
$\{ f_u : u \in \ancestors(v) \cap V_{\mathrm{cond}, \gamma}^{t} \}
\cup \{ \bar f_v \} \in \infconj{\gamma}{t}$.
In other words,
infeasibility follows if the converse of~$f_v$ and all fixings associated
with the conditional ancestors of $v$ are applied.
Regarding the minimal incompatible inf-conjunctions in~$D_\gamma^t$,
these are single conjunctions for each~$i \in [n]$:~%
$\{ (i, 0), (i, 1) \}$ if $i \notin I_0^t \cup I_1^t$,
or~$\{ (i, 0) \}$ if~$i \in I_1^t$,
or~$\{ (i, 1) \}$ if $i \in I_0^t$.
For the special case where~$F(I_0^t, I_1^t) = \emptyset$
(i.e.,~$D_\gamma^t = \{ \emptyset \}$), the tree is marked infeasible.
\end{property}

To show that these properties of implication trees hold, we introduce some
notation.
Let
$
C_v
\define
\{
f_u :
u \in \ancestors(v)
\cap
(
V_{\mathrm{necc}, \gamma}^t
\cup V_{\mathrm{cond}, \gamma}^t
)
\}
$
be the conjunction of all fixings found on the path from the
root~$r_\gamma$ to~$v$.
Also, for fixing $f = (i, b)$,
let $\ent(f) = i$ be the \emph{entry} of fixing $f$,
and for a set of fixings~$C \subseteq [n] \times \binary$,
let~$\ent(C) \define \{ \ent(f) : f \in C \}$.
Using this notation, we show four auxiliary properties that we need to
show the former properties.

\begin{property}
\label{prop:looseend}
Loose end vertices will occur only as leaves of the tree.
\end{property}

\begin{property}
\label{prop:entries}
If a loose end vertex $v \in V_{\mathrm{loose}, \gamma}^t$ exists,
$\ent(C_u) \subseteq \ent(C_v)$ for all $u \in V_{\gamma}^t$.
Also, any loose end vertex $v \in V_{\mathrm{loose}, \gamma}^t$ satisfies
$\ent(C_v) = \bigcup_{i \in [i_\gamma^t - 1]} \{ i, \gamma^{-1}(i) \}
\setminus (I_0^t \cup I_1^t)$.
\end{property}

\begin{property}
\label{prop:implicationtreeoutdegreetwo}
Any vertex $\hat v$
in $\tree_\gamma^t$ with outdegree larger than one
has outdegree two, and its children $\hat u_1, \hat u_2$
are conditional vertices.
In turn, $\hat u_1$ and $\hat u_2$ have outdegree one,
and their children (resp.~$\hat w_1$ and~$\hat w_2$)
are necessary fixing vertices
with $f_{\hat u_1} = \bar f_{\hat w_2}$
and $f_{\hat u_2} = \bar f_{\hat w_1}$,
where fixings $f_{\hat u_1}$ and $f_{\hat u_2}$ are on different entries.
This is depicted in Figure~\ref{fig:prop:implicationtreeoutdegreetwo}.
\end{property}

\begin{property}
\label{prop:rootedpathnotwoentries}
In any rooted path $P$ in the implication tree $\tree_\gamma^t$,
the fixings of all fixing vertices are on different entries,
and no entry is in $I_0^t$ or $I_1^t$.
\end{property}

These properties ensure that if a conjunction (either an
$k$-inf-conjunction or $k$-eq-conjunction) is encoded by an implication
tree, then no inclusionwise smaller or larger conjunction of the same type
is represented by the same implication tree.

\begin{lemma}
\label{lem:noinclusionwisecontainmentsinf}
Suppose that
Property~\ref{prop:implicationtreeoutdegreetwo}
and~\ref{prop:rootedpathnotwoentries}
hold for~$\gamma \in \Pi$ at time~$t$.
For all distinct~$v, v' \in V_{\mathrm{necc}, \gamma}^{t}$,
\[
\{ f_u : u \in \ancestors(v) \cap V_{\mathrm{cond}, \gamma}^{t} \}
\cup \{ \bar f_{v} \}
\not\subsetneq
\{ f_u : u \in \ancestors(v') \cap V_{\mathrm{cond}, \gamma}^{t} \}
\cup \{ \bar f_{v'} \}.
\]
In other words, the conjunctions implied by $v$ and $v'$ are inclusionwise
not contained in each other.
\end{lemma}
\begin{proof}
Suppose $v, v' \in V_{\mathrm{necc}, \gamma}^{t}$,
respectively representing conjunctions $C, C'$
with $C \subsetneq C'$.
Let $u$ be the first common ancestor of $v$ and $v'$.
By Property~\ref{prop:rootedpathnotwoentries}, $u \notin \{ v, v' \}$,
so by Property~\ref{prop:implicationtreeoutdegreetwo}, $u$ is a vertex of
outdegree 2.
Without loss of generality, identify $u$ with $\hat v$ in the property,
and let $v$ be in the subtree of $\hat u_1$ and $v'$ be in the subtree
of $\hat u_2$. Then $f_{\hat u_1} \in C$, so by
Property~\ref{prop:rootedpathnotwoentries} and $C \subsetneq C'$ we must
have that $v' = \hat w_2$.
But then~$|C'| \leq |C|$, contradicting~$C \subsetneq C'$.
\end{proof}

\begin{lemma}
\label{lem:noinclusionwisecontainmentseq}
Suppose that
Property~\ref{prop:looseend},~%
\ref{prop:implicationtreeoutdegreetwo}
and~\ref{prop:rootedpathnotwoentries} hold for~$\gamma \in \Pi$ at
time~$t$.
For all distinct~$v, v' \in V_{\mathrm{loose}, \gamma}^{t}$,
$
\{ f_u : u \in \ancestors(v) \cap V_{\mathrm{cond}, \gamma}^t \}
\not\subsetneq
\{ f_u : u \in \ancestors(v') \cap V_{\mathrm{cond}, \gamma}^t \}
$.
In other words, the conjunctions implied by $v$ and $v'$ are inclusionwise
not contained in each other.
\end{lemma}
\begin{proof}
Suppose $v$ and~$v' \in V_{\mathrm{loose}, \gamma}^t$,
representing conjunctions $C$ and~$C'$, respectively, with~${C \subsetneq C'}$.
Vertices $v$ and $v'$ are in different subtrees,
since they are both leaves due to Property~\ref{prop:looseend}.
Let~$u$ be the first common ancestor of $v$ and $v'$,
so by Property~\ref{prop:implicationtreeoutdegreetwo}, $u$ is a vertex of
outdegree two. Without loss of generality, identify $u$ with $\hat v$ in
the property, and let $v$ be in the subtree of $\hat u_1$ and $v'$ in
the subtree of $\hat u_2$.
Then $f_{\hat u_1} \in C$.
But vertex $\hat w_2$ has fixing~$\bar f_{\hat u_1}$, so
Property~\ref{prop:rootedpathnotwoentries}
yields~$f_{\hat u_1} \notin C'$. This contradicts $C \subsetneq C'$.
\end{proof}

We now proceed by induction to show that the aforementioned
properties hold.
To this end, we show that they hold at initialization,
and that they are maintained during any step of the algorithm.

\paragraph{Initial state}
At initialization (i.e.,~$t=0$),
the index $i_\gamma$ is set to one for all~$\gamma \in \Pi$.
Because $\eqp{1}$ is a tautology, the empty set is the only
minimal $i_\gamma$-eq-conjunction for all~$\gamma \in \Pi$:
$\eqconj{\gamma}{0} = \{ \emptyset \}$.
Similarly,~$\precp{1}$ is a contradictory operator,
so at initialization the only $i_\gamma$-inf-conjunctions
are the inf-conjunctions that are incompatible with $I_0, I_1$.
An implication tree that consists of the root vertex that is the parent of
a single loose end vertex encodes these $1$-eq-conjunctions and
$1$-inf-conjunctions, and respects all aforementioned properties.
This is depicted in Figure~\ref{fig:initialstate}.

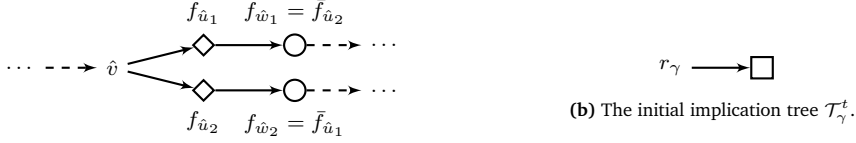
\begin{figure}[!tbp]
\begin{subfigure}{.5\textwidth}
\centering
\begin{tikzpicture}[x=12mm, y=3mm, font=\footnotesize]
\node[] (beforestart) at (-1, 0) {$\dots$};
\node[] (start) at (0, 0) {$\hat v$};
\node[cond, label=above:{$f_{\hat u_1}$}] (a0) at (1, 1) {};
\node[cond, label=below:{$f_{\hat u_2}$}] (b1) at (1, -1) {};
\node[necc, label=above:{$f_{\hat w_1} = \bar f_{\hat u_2}$}]
(b0) at (2, 1) {};
\node[necc, label=below:{$f_{\hat w_2} = \bar f_{\hat u_1}$}]
(a1) at (2, -1) {};
\node[] (l0) at (3, 1) {$\dots$};
\node[] (l1) at (3, -1) {$\dots$};

\draw[edge, dashed] (beforestart) to (start);
\draw[edge] (start) to (a0);
\draw[edge] (a0) to (b0);
\draw[edge, dashed] (b0) to (l0);
\draw[edge] (start) to (b1);
\draw[edge] (b1) to (a1);
\draw[edge, dashed] (a1) to (l1);
\end{tikzpicture}
\caption{Implication tree $\tree_\gamma^t$ around vertex $v$ with
outdegree two.}
\label{fig:prop:implicationtreeoutdegreetwo}
\end{subfigure}%
\begin{subfigure}{.5\textwidth}
\centering
\begin{tikzpicture}[x=12mm, y=3mm, font=\footnotesize]
\node[] (root) at (0, 0) {$r_\gamma$};
\node[loos] (end) at (1, 0) {};

\draw[edge] (root) to (end);
\end{tikzpicture}
\caption{The initial implication tree $\tree_\gamma^t$.}
\label{fig:initialstate}
\end{subfigure}%
\caption{Illustrations for Property~\ref{prop:implicationtreeoutdegreetwo}
and the initial state.}
\end{figure}

\paragraph{Selecting the permutation}
Proposition~\ref{lem:terminationcondition} shows
sufficient conditions for completeness of fixings in~$\xymretope\gamma$
for each $\gamma \in \Pi$, expressed in the state of the algorithm.

\begin{restatable}[Sufficient conditions for %
completeness]{proposition}{terminatecondition}
\label{lem:terminationcondition}
\label{prop:terminationcondition}
Consider~$\gamma \in \Pi$, let~$t$ be some time index,
and~$\emptyset \notin \infconj{\gamma}{t}$.
Suppose that, if~$\{(i,b)\} \in \infconj{\gamma}{t}$,
then $i \in I_{1-b}^t$.
Then, the set of fixings $(I_0^t, I_1^t)$ is complete
for~$x \succeq \gamma(x)$
if
\begin{enumerate*}[label=\textbf{\emph{\small(P\arabic*.)}},
ref=\emph{\small(P\arabic*)}]
\item $\eqconj{\gamma}{t} = \emptyset$, or
\label{prop:terminationcondition:1}
\item $i_\gamma^t > n$, or
\label{prop:terminationcondition:2}
\item all of the following:
\label{prop:terminationcondition:3}
$\emptyset \notin \eqconj{\gamma}{t}$,
and $i_\gamma^t \notin I_0^t$,
and $\gamma^{-1}(i_\gamma^t) \notin I_1^t$,
and $\gamma(i_\gamma^t) > i_\gamma^t$,
and $\gamma^{-1}(i_\gamma^t) > i_\gamma^t$.
\end{enumerate*}
\end{restatable}

The proof of this proposition, and of the other propositions in this
section, are given in Appendix~\ref{sec:proofs}.
These conditions can be used in
Line~\ref{alg:propagateallindividualcons:select}
of the algorithm, because the conditions of the second sentence are
satisfied at initialization, and at the end of an iteration due to
the loop of Line~\ref{alg:propagateallindividualcons:whilefix}.
Due to the induction hypothesis, Properties~\ref{prop:eqconj},
\ref{prop:infconj} and~\ref{prop:looseend} are satisfied,
so the following corollary shows the same conditions in terms of the
implication trees.

\begin{corollary}[Sufficient conditions for completeness]
\label{cor:terminationcondition}
Consider $\gamma \in \Pi$, and let $t$ be some time index.
Suppose that the root of $\tree_\gamma^t$ has no necessary fixing vertex
as child.
The set of fixings $(I_0^t, I_1^t)$ is complete for $\gamma \in \Pi$ if
\begin{enumerate*}[label=\textbf{\emph{\small(C\arabic*.)}},
ref=\emph{\small(C\arabic*)}]
\item
\label{cor:terminationcondition:C1}
There is no loose end vertex in $\tree_\gamma^t$, or
\label{def:sufficientcompletenessconditions:1}
\item
\label{cor:terminationcondition:C2}
$i_\gamma^t > n$, or
\item
\label{cor:terminationcondition:C3}
Every rooted path to a loose end vertex contains a conditional
fixing vertex,
and $i_\gamma^t \notin I_0^t$,
and $\gamma^{-1}(i_\gamma^t) \notin I_1^t$,
and $\gamma(i_\gamma^t) > i_\gamma^t$,
and $\gamma^{-1}(i_\gamma^t) > i_\gamma^t$.
\label{def:sufficientcompletenessconditions:3}
\end{enumerate*}
\end{corollary}%
\begin{proof}
Follows from Proposition~\ref{lem:terminationcondition}
and the encoding of $\infconj{\gamma}{t}$ and $\eqconj{\gamma}{t}$
by~$\tree_\gamma^t$.
\end{proof}%

Moreover, when these sufficient conditions are used by the algorithm,
it is guaranteed that the algorithm terminates.
Namely, one of the sufficient conditions is $i_\gamma^t > n$,
so one can increase the index of a single permutation at most~$n$ times.
Also, at most~$n$ fixings are possible, so the inner loop can only be
called a finite number of times, as well.

\paragraph{Index increasing event}
At Line~\ref{alg:propagateallindividualcons:indexincr},
the index $i_\gamma$ of a permutation $\gamma \in \Pi$ is increased by one,
along with the timestamp~$t$,
denoted here by $i_\gamma^{t + 1} = i_\gamma^t + 1$.
This does not affect the fixing sets or
the conjunction sets of other permutations,
so also the encoding implication trees remain the same.
In order for a conjunction $C$ to be a $i_\gamma^{t + 1}$-eq-conjunction,
it specifically has to be a $i_\gamma^{t}$-eq-conjunction,
and therefore the latter conjunctions are the starting point for
determining $i_\gamma^{t + 1}$-eq-conjunctions.
If an inclusionwise minimal $i_\gamma^{t}$-eq-conjunction $C$
is an $i_\gamma^{t + 1}$-inf-conjunction, then no $i_\gamma^{t +
1}$-eq-conjunctions are derived from~$C$.
Otherwise, this conjunction is extended (if needed)
with fixings of $i_\gamma^t$ to zero
or $\gamma^{-1}(i_\gamma^t)$ to one,
preventing that $x \succp{i_\gamma^t + 1} \gamma(x)$
for some $x \in F(I_0^t, I_1^t)$ that satisfies the conjunction.
Note that, this way, a single conjunction of the previous timestamp could
induce two other eq-conjunctions.

\begin{restatable}[Updating eq-conjunctions for index increasing event]
{proposition}{indexeqconj}
\label{prop:indexeqconj}
Consider an index increasing event for permutation $\gamma \in \Pi$ at
time~$t$.
Then, for $\delta \in \Pi \setminus \{ \gamma \}$, we have
$\eqconj{\delta}{t + 1} = \eqconj{\delta}{t}$,
and $\eqconj{\gamma}{t + 1} = Y \cup Z$, where
\begin{align*}
Y &=
\left\{%
\vphantom{\begin{aligned}1\\2\\3\end{aligned}}
\right.%
C :
&&
\begin{aligned}
&C \in \eqconj{\gamma}{t},
\text{and}
\\
&\text{for all}\
x \in \xymretope\gamma^{(i_\gamma^t)} \cap F(I_0^t, I_1^t) \cap V(C)\
\text{holds}\
x_{i_\gamma^t} \leq \gamma(x)_{i_\gamma^t},\
\text{and}
\\
&\text{there is}\
x \in \xymretope\gamma^{(i_\gamma^t)} \cap F(I_0^t, I_1^t) \cap V(C)\
\text{with}\
x_{i_\gamma^t} = \gamma(x)_{i_\gamma^t}
\end{aligned}
\left.%
\vphantom{\begin{aligned}1\\2\\3\end{aligned}}
\right\},\
\text{and}
\\
Z &=
\left\{\rule{0cm}{2.5em}\right.%
C \cup S :
&&
\begin{aligned}
&
C \in \eqconj{\gamma}{t},
S \in \{ \{ (i_\gamma^t, 0) \},
\{ (\gamma^{-1}(i_\gamma^t), 1) \} \},\
\\
&
\xymretope\gamma^{(i_\gamma^{t+1})} \cap F(I_0^{t+1}, I_1^{t+1})
\cap V(C \cup S) \neq \emptyset,\
\text{and}
\\
&
\text{there is}\
x \in \xymretope\gamma^{(i_\gamma^t)} \cap F(I_0^t, I_1^t) \cap V(C)\
\text{with}\
x_{i_\gamma^t} > \gamma(x)_{i_\gamma^t}
\end{aligned}
\left.\rule{0cm}{2.5em}\right\}%
.
\end{align*}
\end{restatable}

Regarding the $i_\gamma^{t + 1}$-inf-conjunctions,
all $i_\gamma^t$-inf-conjunctions are again valid inf\hyph{}conjunctions.
Added to this, new inf-conjunctions have equality up to
$i_\gamma^t$, and $x_{i_\gamma^t} < \gamma(x)_{i_\gamma^t}$.
The latter are constructed from the $i_\gamma^t$-eq-conjunctions,
which are (if needed) extended with fixings of $i_\gamma^t$ to zero and
$\gamma^{-1}(i_\gamma^t)$ to one
to enforce infeasibility at entry $i_\gamma^t$.
Again, in both cases one is interested in inclusionwise minimal
conjunctions, so non-minimal conjunctions are removed.

\begin{restatable}[Updating inf-conjunctions for index increasing event]%
{proposition}{indexinfconj}
\label{prop:indexinfconj}
Consider an index increasing event for permutation $\gamma \in \Pi$ at
time~$t$.
Then, for $\delta \in \Pi \setminus \{ \gamma \}$, we have
$\infconj{\delta}{t + 1} = \infconj{\delta}{t}$,
and
$
\infconj{\gamma}{t + 1} =
\infconj{\gamma, \mathrm{eq}}{t + 1}
\cup
\infconj{\gamma, \mathrm{inf}}{t + 1}
$
with
\[
\infconj{\gamma, \mathrm{eq}}{t + 1}
=
\left\{
C \cup S :
\begin{aligned}
&
C \in \eqconj{\gamma}{t},\
S \subseteq \{ (i_\gamma^t, 0), (\gamma^{-1}(i_\gamma^t), 1) \},\
\gamma(i_\gamma^t) \neq i_\gamma^t,\
\text{and}
\\
&
\text{either}\ (i_\gamma^t, 0) \in S\
\text{or}\
x_{i_\gamma^t} = 0\
\text{for all}\
x \in \xymretope\gamma^{(i_\gamma^t)} \cap F^t \cap V(C),\
\text{and}
\\
&
\text{either}\ (\gamma^{-1}(i_\gamma^t), 1) \in S\
\text{or}\
\gamma(x)_{i_\gamma^t} = 1\
\text{for all}\
x \in \xymretope\gamma^{(i_\gamma^t)} \cap F^t \cap V(C)
\end{aligned}
\right\}%
,
\]
with $F^t \define F(I_0^t, I_1^t)$,
and~%
$
\infconj{\gamma, \mathrm{inf}}{t + 1}
=
\left\{
C \in \infconj{\gamma}{t}
:
\text{for all}\
C' \in \infconj{\gamma, \mathrm{eq}}{t}\
\text{holds}\
C \not\supsetneq C'
\right\}%
$%
.
\end{restatable}

These propositions allow us to efficiently update the implication trees in
case of an index increasing event.
The main idea of the implication tree update stems from the observation
that all newly introduced conjunctions are based on
some~$i_\gamma^t$-eq-conjunction that gets extended with zero, one or two
fixings to find the new conjunctions.
These~$i_\gamma^t$-eq-conjunctions correspond to the loose end vertices
in~$\tree_\gamma^t$, by the induction hypothesis.
We construct~$\tree_\gamma^{t + 1}$ by replacing those loose end vertices
with a subtree that encodes the new $i_\gamma^{t + 1}$-inf-conjunctions
and $i_\gamma^{t + 1}$-eq-conjunctions.
It is also possible that the tree structure of a subtree containing this
loose end vertex changes, to account for the
$i_\gamma^{t}$-inf-conjunctions that are setwise dominated by new
$i_\gamma^{t+1}$-inf-conjunctions.
In the special case that~$i_\gamma^t$ is a fixed point of~$\gamma$,
then these updates correspond
to~$\infconj{\gamma}{t + 1} = \infconj{\gamma}{t}$
and~$\eqconj{\gamma}{t + 1} = \eqconj{\gamma}{t}$.
Hence, the implication tree structure remains unchanged in this case.
In the following,
we therefore assume that $i_\gamma^t$ is no fixed point of~$\gamma$.

Consider a loose end vertex $v_\ell \in V_{\mathrm{loose}, \gamma}^{t}$.
Its associated minimal $i_\gamma^t$-eq-conjunction
corresponds to the union of the fixings of the conditional ancestors
of~$v_\ell$ (Property~\ref{prop:eqconj}).
To ensure that constraint $x \succeqp{i_\gamma^t} \gamma(x)$ is satisfied
if these fixings are applied, also all fixings of
the necessary fixing vertices on the rooted path to $v_\ell$
must be applied. That is, because otherwise infeasibility is yielded by
the~$i_\gamma^t$-inf-conjunctions associated with the necessary fixing
vertices on the rooted path to~$v_{\ell}$ (Property~\ref{prop:infconj}).
These fixings are also sufficient to ensure feasibility.
Namely, consider a necessary fixing vertex~$v_n$
that does not lie on the rooted path to~$v_\ell$,
and let~$v_c$ be the first common ancestor of~$v_n$ and~$v_\ell$.
Also, let $u_{c, n}$ be the last non-common ancestor of~$v_c$,
and~$u_{c, \ell}$ the last non-common ancestor of~$v_\ell$.
Then~$v_c$ corresponds to~$\hat v$ in
Property~\ref{prop:implicationtreeoutdegreetwo},
and without loss of generality~$u_{c, n}$ corresponds to~$\hat u_1$,
and~$u_{c, \ell}$ to~$\hat u_2$.
Since all fixings on the rooted path to $v_\ell$ are applied,
Property~\ref{prop:implicationtreeoutdegreetwo} yields
that~$\bar f_{\hat u_1}$ is applied, since $f_{\hat u_1}$ is in the
$i_\gamma^t$-inf-conjunction associated with~$v_n$.
Hence, no infeasibility will be induced by any necessary fixing vertex
that does not lie on the rooted path to $v_\ell$.

This result allows for phrasing the updates of
Proposition~\ref{prop:indexeqconj} and~\ref{prop:indexinfconj}
in terms of the implication tree.
Recall that
$
C_v
\define
\{
f_u :
u \in \ancestors(v)
\cap
(
V_{\mathrm{necc}, \gamma}^t
\cup V_{\mathrm{cond}, \gamma}^t
)
\}
$
is the conjunction of all fixings found on the path from the
root~$r_\gamma$ to~$v$.
The previous paragraph yields:
\begin{observation}
\label{obs:indexincr}
For $v \in V_{\mathrm{loose}, \gamma}^t$,
with associated $i_\gamma^t$-eq-conjunction~$C \in \eqconj{\gamma}{t}$,
we have
$\xymretope\gamma^{(i_\gamma^t)} \cap F(I_0^t, I_1^t) \cap V(C)
= F(I_0^t, I_1^t) \cap V(C_v)$.
\end{observation}
Let~%
$h^t\colon [n] \times V_{\mathrm{loose}, \gamma}^t \to \binaryblank$
be the value of the fixing of~$i$ at time~$t$
when respecting the fixings given by~%
$I_0^t, I_1^t$ and by~$C_v$. If entry $i$ is not fixed, this is denoted
by~$\blank$:
\begin{equation}
\label{eq:func_utility_h}
h^t(i, v) \define
\begin{cases}
0, & \text{if}\ i \in I_0^t\ \text{or}\ (i, 0) \in C_v, \\
1, & \text{if}\ i \in I_1^t\ \text{or}\ (i, 1) \in C_v, \\
\blank, & \text{if}\ i \notin I_0^t \cup I_1^t\ \text{and}\ i \notin
\ent(C_v).
\end{cases}
\end{equation}%

Note that is well-defined due to
Property~\ref{prop:rootedpathnotwoentries}.
The implication tree $\tree_\gamma^{t + 1}$
is found by applying an update to $\tree_\gamma^t$ for each loose end
vertex $v \in V_{\mathrm{loose}, \gamma}^t$,
and let $C \in \eqconj{\gamma}{t}$ be the
associated~$i_\gamma^t$-eq-conjunction.
By Property~\ref{prop:looseend},~$v$ is always a leaf of the tree.
For compactness, denote~%
$(i, j) \define (i_\gamma^t, \gamma^{-1}(i_\gamma^t))$,
and let~$(\alpha, \beta) \define (h^t(i, v), h^t(j, v))$.

We discuss the updates for different tuples~$(\alpha, \beta)$ in turn.
To this end, we need to show
Properties~\ref{prop:eqconj}--\ref{prop:rootedpathnotwoentries}.
We start with
Properties~\ref{prop:looseend}--\ref{prop:rootedpathnotwoentries}.
On the one hand, if $(\alpha, \beta) = (0, 1)$,
then Observation~\ref{obs:indexincr} yields
$x_{i_\gamma^t} = 0$ and $\gamma(x)_{i_\gamma^t} = 1$
for all~%
$x \in \xymretope\gamma^{(i_\gamma^t)} \cap F(I_0^t, I_1^t) \cap V(C)$.
The following update is applied.
Let $u$ be the first ancestor of $v$ that is a conditional fixing vertex.
If~$u$ does not exist,
then~$\emptyset \in \infconj{\gamma}{t + 1}$
and we mark the tree as infeasible, so that this situation is handled by
Line~\ref{alg:propagateallindividualcons:inf} next.
Otherwise,
all necessary fixing vertices~$w$ in the subtree of~$u$
with associated $i_\gamma^t$-inf-conjunctions $C'$
have $C' = C \cup \{ \bar f_w \} \supsetneq C$, so these must be removed.
Also loose end vertex $v$ must be removed, since it is no
$i_\gamma^{t+1}$-eq-conjunction because of $(\alpha, \beta)$.
Hence, replace the subtree rooted at $u$ with a single necessary fixing
vertex $u_{\mathrm{new}}$ with $f_{u_{\mathrm{new}}} = \bar f_{u}$.
A \emph{merging step} is applied if $u$ has a sibling $w$.
That is,
the child $x$ of $w$ is removed and $w$ becomes the parent of the children
of $x$, and the parent of $w$ is changed to $v$.
Figure~\ref{fig:mergingstep} depicts the merging step.

These operations remove all vertices from the subtree of $u$.
That is consistent, since $u_{\mathrm{new}}$ has $C$ as
associated~$i_\gamma^{t + 1}$-inf-conjunction,
which dominates the former $i_\gamma^t$-inf-conjunctions associated to the
necessary fixing vertices in the subtree of $u$.
In the merging step, the removal of $x$ is needed since
Property~\ref{prop:implicationtreeoutdegreetwo}
yields~$f_{u_{\mathrm{new}}} = \bar f_u = f_x$,
meaning that the inf-conjunction that $x$ encodes
is~$C \cup \{ f_w \} \supsetneq C$.
Moreover, no $i_\gamma^t$-inf-conjunctions $C'$ with $C' \supsetneq C$ are
missed, since by Property~\ref{prop:implicationtreeoutdegreetwo}
and~\ref{prop:rootedpathnotwoentries} the necessary fixing vertices
associated to $C'$ must be in the rooted path to $u_{\mathrm{new}}$, or in
the subtree of~$w$. %
Last, Properties~\ref{prop:looseend}, \ref{prop:entries},
\ref{prop:implicationtreeoutdegreetwo},
and~\ref{prop:rootedpathnotwoentries}
are maintained, because they are not influenced by a subtree removal or by
the merging step.

\begin{figure}[!tbp]
\centering
\begin{subfigure}{\textwidth/3}
\begin{minipage}[c][15mm][c]{\textwidth}
\centering
\begin{tikzpicture}[x=8mm, y=3mm, font=\footnotesize]
\node[] (start) at (0, 0) {};
\node[cond, label=above:{$w$}] (w) at (1, 1) {};
\node[necc, label=below:{$u_{\mathrm{new}}$}] (u) at (1, -1) {};
\node[necc, label=above:{$x$}] (x) at (2, 1) {};
\node[] (dots) at (3, 1) {$\cdots$};

\draw[edge] (start) to (w);
\draw[edge] (start) to (u);
\draw[edge] (w) to (x);
\draw[edge, dotted] (x) to (dots);
\end{tikzpicture}
\end{minipage}
\caption{Before merging step}
\end{subfigure}%
\begin{subfigure}{\textwidth/3}
\begin{minipage}[c][15mm][c]{\textwidth}
\centering
{changes to}
\end{minipage}
\end{subfigure}%
\begin{subfigure}{\textwidth/3}
\begin{minipage}[c][15mm][c]{\textwidth}
\centering
\begin{tikzpicture}[x=8mm, y=3mm, font=\footnotesize]
\node[] (start) at (0, 0) {};
\node[cond, label=above:{$w$}] (w) at (2, 0) {};
\node[necc, label=above:{$u_{\mathrm{new}}$}] (u) at (1, 0) {};
\node[] (dots) at (3, 0) {$\cdots$};

\draw[edge] (u) to (w);
\draw[edge] (start) to (u);
\draw[edge, dotted] (w) to (dots);
\end{tikzpicture}
\end{minipage}
\caption{After merging step}
\end{subfigure}%
\caption{Depiction of merging step.}
\label{fig:mergingstep}
\end{figure}
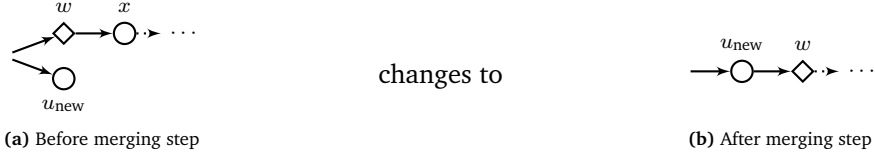%

On the other hand, if $(\alpha, \beta) \neq (0, 1)$, then the loose end
vertex~$v$ is replaced by
a subtree depending on the value of $(\alpha, \beta)$.
The possible new subtrees are listed in
Figure~\ref{fig:update:subtree_options}.
These subtrees replace vertex $v$, so they are connected to its parent.

\begin{figure}[!tbp]
\begin{subfigure}{\textwidth/4}
\begin{minipage}[c][15mm][c]{\textwidth}
\centering
\begin{tikzpicture}[x=8mm, y=3mm, font=\footnotesize]
\node[] (start) at (0, 0) {};
\node[cond, label=above:{$(i, 0)$}] (a0) at (1, 1) {};
\node[cond, label=below:{$(j, 1)$}] (b1) at (1, -1) {};
\node[necc, label=above:{$(j, 0)$}] (b0) at (2, 1) {};
\node[necc, label=below:{$(i, 1)$}] (a1) at (2, -1) {};
\node[loos] (l0) at (3, 1) {};
\node[loos] (l1) at (3, -1) {};

\draw[edge] (start) to (a0);
\draw[edge] (a0) to (b0);
\draw[edge] (b0) to (l0);
\draw[edge] (start) to (b1);
\draw[edge] (b1) to (a1);
\draw[edge] (a1) to (l1);
\end{tikzpicture}
\end{minipage}
\caption{$(\alpha, \beta) = (\blank, \blank)$}
\label{fig:update:subtree_options:__}
\end{subfigure}%
\begin{subfigure}{\textwidth/4}
\begin{minipage}[c][15mm][c]{\textwidth}
\centering
\begin{tikzpicture}[x=8mm, y=3mm, font=\footnotesize]
\node[] (start) at (0, 0) {};
\node[necc, label=below:{$(j, 0)$}] (b0) at (1, 0) {};
\node[loos] (l0) at (2, 0) {};

\draw[edge] (start) to (b0);
\draw[edge] (b0) to (l0);
\end{tikzpicture}
\end{minipage}
\caption{$(\alpha, \beta) = (0, \blank)$}
\end{subfigure}%
\begin{subfigure}{\textwidth/4}
\begin{minipage}[c][15mm][c]{\textwidth}
\centering
\begin{tikzpicture}[x=8mm, y=3mm, font=\footnotesize]
\node[] (start) at (0, 0) {};
\node[cond, label=below:{$(j, 1)$}] (b0) at (1, 0) {};
\node[loos] (l0) at (2, 0) {};

\draw[edge] (start) to (b0);
\draw[edge] (b0) to (l0);
\end{tikzpicture}
\end{minipage}
\caption{$(\alpha, \beta) = (1, \blank)$}
\end{subfigure}%
\begin{subfigure}{\textwidth/4}
\begin{minipage}[c][15mm][c]{\textwidth}
\centering
\begin{tikzpicture}[x=8mm, y=3mm, font=\footnotesize]
\node[] (start) at (0, 0) {};
\node[cond, label=below:{$(i, 0)$}] (b0) at (1, 0) {};
\node[loos] (l0) at (2, 0) {};

\draw[edge] (start) to (b0);
\draw[edge] (b0) to (l0);
\end{tikzpicture}
\end{minipage}
\caption{$(\alpha, \beta) = (\blank, 0)$}
\end{subfigure}%
\newline
\begin{subfigure}{\textwidth/4}
\begin{minipage}[c][15mm][c]{\textwidth}
\centering
\begin{tikzpicture}[x=8mm, y=3mm, font=\footnotesize]
\node[] (start) at (0, 0) {};
\node[loos] (l0) at (1, 0) {};

\draw[edge] (start) to (l0);
\end{tikzpicture}
\end{minipage}
\caption{$(\alpha, \beta) = (0, 0)$}
\end{subfigure}%
\begin{subfigure}{\textwidth/4}
\begin{minipage}[c][15mm][c]{\textwidth}
\centering
Empty graph.
\end{minipage}
\caption{$(\alpha, \beta) = (1, 0)$}
\end{subfigure}%
\begin{subfigure}{\textwidth/4}
\begin{minipage}[c][15mm][c]{\textwidth}
\centering
\begin{tikzpicture}[x=8mm, y=3mm, font=\footnotesize]
\node[] (start) at (0, 0) {};
\node[necc, label=below:{$(i, 1)$}] (b0) at (1, 0) {};
\node[loos] (l0) at (2, 0) {};

\draw[edge] (start) to (b0);
\draw[edge] (b0) to (l0);
\end{tikzpicture}
\end{minipage}
\caption{$(\alpha, \beta) = (\blank, 1)$}
\end{subfigure}%
\begin{subfigure}{\textwidth/4}
\begin{minipage}[c][15mm][c]{\textwidth}
\centering
\begin{tikzpicture}[x=8mm, y=3mm, font=\footnotesize]
\node[] (start) at (0, 0) {};
\node[loos] (l0) at (1, 0) {};

\draw[edge] (start) to (l0);
\end{tikzpicture}
\end{minipage}
\caption{$(\alpha, \beta) = (1, 1)$}
\end{subfigure}%
\caption{Possible subtrees to replace a loose end vertex $v$ for
$(\alpha, \beta)$.}
\label{fig:update:subtree_options}
\end{figure}
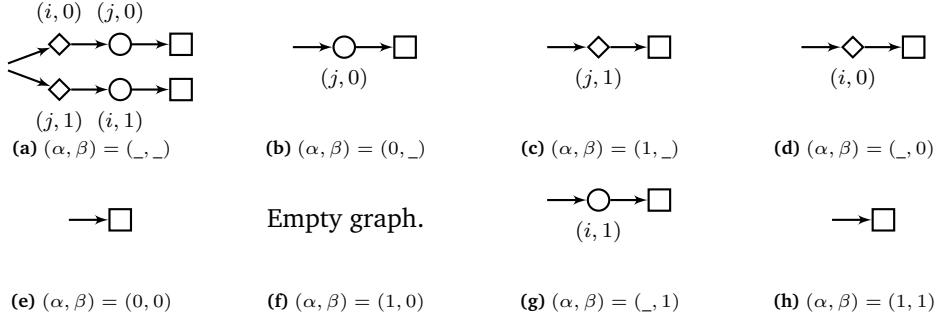

Observe that loose end vertices are always leaves, so the
update maintains Property~\ref{prop:looseend}.
Due to Property~\ref{prop:entries},
any loose end vertex $u \in V_{\mathrm{loose}, \gamma}^t$
has $\ent(C_u) = \ent(C_v)$, and $h^t(i', u) = \blank$
for $i' \in \{ i, j\}$
if and only if $i' \notin \ent(C_v) \cup (I_0^t \cup I_1^t)$.
Note that all subtrees in Figure~\ref{fig:update:subtree_options} contain
exactly these fixing vertices with $h^t(i', v) = \blank$ for $i' \in \{ i,
j\}$, which shows that Property~\ref{prop:entries} is maintained.
Property~\ref{prop:implicationtreeoutdegreetwo} holds after the
update, since only the tree of
Figure~\ref{fig:update:subtree_options:__} that
introduces a vertex with outdegree larger than one satisfies this property.
Also, Property~\ref{prop:rootedpathnotwoentries} is maintained,
since only fixing vertices with fixing~$(k, b)$ are introduced
with~$h^t(k, v) = \blank$, this means that~$k \notin I_0^t \cup I_1^t$,
and~$k$ is not an entry of any fixing on~$C_v$.

Concluding, in both cases $(\alpha, \beta) = (0, 1)$
and~$(\alpha, \beta) \neq (0, 1)$,
Properties~\ref{prop:looseend}--\ref{prop:rootedpathnotwoentries}
are maintained.
Last we show that the Properties~\ref{prop:eqconj} and~\ref{prop:infconj}
are also maintained, by showing that the tree updates correspond to the
elements that~$C$
yields in Proposition~\ref{prop:indexeqconj} and~\ref{prop:indexinfconj}.
Consider~$(\alpha, \beta) = (\blank, \blank)$.
Substituting the expression of Observation~\ref{obs:indexincr} in
Proposition~\ref{prop:indexeqconj} yields that
the elements~$\eqconj{\gamma}{t + 1}$
derived from~$C$ are~$C \cup \{ (i, 0) \}$
and~$C \cup \{ (j, 1) \}$, and these correspond to the conjunctions
derived from the new loose end vertices such as in
Figure~\ref{fig:update:subtree_options:__}.
Likewise, substituting in Proposition~\ref{prop:indexinfconj}
yields that the only element derived from~$C$ in~$\infconj{\gamma}{t + 1}$
is~$C \cup \{ (i, 0), (j, 1) \}$,
and this is represented by both necessary fixing vertices in
Figure~\ref{fig:update:subtree_options:__}.
The new $i_\gamma^{t+1}$-inf-conjunctions cannot be setwise contained in
an old $i_\gamma^{t}$-conjunction, because no $i_\gamma^t$-inf-conjunction
containing entries~$i$ or~$j$ exists.
This shows that the update is consistent
for~$(\alpha, \beta) = (\blank, \blank)$,
and the other cases with $(\alpha, \beta) \neq (0, 1)$ are analogous.

In the same way, for $(\alpha, \beta) = (0, 1)$,
the conjunction $C$ does not derive any elements in $\eqconj{\gamma}{t+1}$.
Instead, $C$ is added as element to $\infconj{\gamma}{t+1}$.
This is consistent with the update, since the update for $(\alpha, \beta)
= (0, 1)$ introduces a necessary fixing vertex whose conjunction
corresponds to~$C$.
Since all conjunctions in $\infconj{\gamma}{t+1}$ and
$\eqconj{\gamma}{t+1}$ are represented by~$\tree_\gamma^{t+1}$,
Lemma~\ref{lem:noinclusionwisecontainmentsinf}
and~\ref{lem:noinclusionwisecontainmentseq} show that no non-minimal
conjunctions are contained, so Property~\ref{prop:eqconj}
and~\ref{prop:infconj} are maintained.

Figure~\ref{fig:ex:exampletree} shows the implication tree
of Example~\ref{ex:algexample} with $i_{\gamma_1} = 6$.
We execute an index increasing event.
There is a single loose end vertex, and this
has~$(\alpha, \beta) = (0, 1)$.
Thus, if fixing $(7, 1)$ is applied, then an infeasible situation is
found. Hence, this conditional fixing vertex turns in a necessary fixing
vertex with fixing $(7, 0)$. We also apply a merging step to find the
implication tree as in Figure~\ref{fig:ex:exampletree:indexincreasingstep}.

\begin{figure}[!tbp]
\centering
\begin{tikzpicture}[x=12mm, y=3mm, font=\footnotesize]
\node[] (root) at (0, 0) {$r_{\gamma_1}$};

\node[necc, label=above:{$(7, 0)$}] (70) at (1, 0) {};

\node[cond, label=above:{$(2, 0)$}] (20) at (2, 0) {};
\node[necc, label=above:{$(8, 0)$}] (80) at (3, 0) {};

\draw[edge] (root) to (70);
\draw[edge] (70) to (20);
\draw[edge] (20) to (80);
\end{tikzpicture}
\caption{Implication tree of permutation $\gamma_1$ in
Example~\ref{ex:algexample} with $i_{\gamma_1} = 7$ and $x_1 = 1$.}
\label{fig:ex:exampletree:indexincreasingstep}
\end{figure}
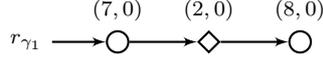

\paragraph{Selecting a fixing}
Line~\ref{alg:propagateallindividualcons:whilefix}
searches for a permutation $\gamma \in \Pi$
that has $\{ (i, b) \}$ as a $i_\gamma$-inf-conjunction,
and $i \notin I_{1-b}^t$.
This is a minimal inf-conjunction,
as otherwise infeasibility was detected at
Line~\ref{alg:propagateallindividualcons:inf},
so it corresponds, by definition,
with $\{ (i, b) \} \in \infconj{\gamma}{t}$
and $i \notin I_{1-b}^t$.
Recall that by Property~\ref{prop:infconj}
such sets correspond to necessary fixing vertices in the implication trees
that lie on a path from the root without any conditional fixing vertices
on it, and that this correspondence works both ways.
Because loose end vertices only occur as leaves
(Property~\ref{prop:looseend}),
such paths only contain the root vertex and other necessary
fixing vertices.
Hence, as a result of the choice of our implication tree data structure,
if the root vertex has a necessary fixing vertex as child,
then the fixing associated with that fixing vertex has to be applied for
completeness.
Moreover, if no such vertex exists, also no corresponding conjunction
in~$D_\gamma^t$ exists.

\paragraph{Variable fixing event}
Line~\ref{alg:propagateallindividualcons:varfixing}
applies a fixing $f = (i, b)$.
In conjunctions containing this fixing, the fixing can be removed.
Likewise, conjunctions containing the converse fixing can impossibly hold,
and need to be removed. Last, to ensure inclusionwise minimality,
non-minimal elements need to be removed, as well.
The following propositions describe the update for the conjunction-sets.

\begin{restatable}[Updating inf-conjunctions for variable fixing event]%
{proposition}{varfixinfconj}
\label{prop:varfixinfconj}
Consider a variable fixing event for fixing
$f = (i, b) \in [n] \times \binary$ at time~$t$.
For every $\gamma \in \Pi$ holds
\begin{equation}
\label{eq:varfixinfconj}
\infconj{\gamma}{t + 1} =
\left\{
C \setminus \{ f \}
:
\begin{aligned}[c]
&C \in \infconj{\gamma}{t}, \text{and} \\
&\text{for all}\
C' \in \infconj{\gamma}{t}\
\text{holds}\
C' \setminus \{ f \} \not\subsetneq C \setminus \{ f \}
\end{aligned}
\right\}%
.
\end{equation}
\end{restatable}

\begin{restatable}[Updating eq-conjunctions for variable fixing event]%
{proposition}{varfixeqconj}
\label{prop:varfixeqconj}
Consider a variable fixing event for fixing
$f = (i, b) \in [n] \times \binary$ at time~$t$.
For every $\gamma \in \Pi$ holds%
\begin{equation}
\label{eq:varfixeqconj}
\eqconj{\gamma}{t + 1} =
\left\{
C \setminus \{ f \}
:
\begin{aligned}[c]
&C \in \eqconj{\gamma}{t},
\text{and}\ \\
&\text{for all}\
C' \in \eqconj{\gamma}{t}\
\text{holds}\
C' \setminus \{ f \} \not\subsetneq C \setminus \{ f \},\
\text{and}
\\
&\text{for all}\
C' \in \infconj{\gamma}{t}\
\text{holds}\
C' \setminus \{ f \} \not\subseteq C \setminus \{ f \}
\end{aligned}
\right\}%
.
\end{equation}
\end{restatable}

In terms of the implication trees,
for each $\gamma \in \Pi$,
the tree $\tree_\gamma^{t + 1}$ is created from $\tree_\gamma^t$
by a sequence of updates.
For all fixing vertices
$v \in V_{\mathrm{cond}, \gamma} \cup V_{\mathrm{necc}, \gamma}$
with $\ent(f_v) = i$, apply one of the following updates:
\begin{enumerate*}[label=\textbf{\emph{\small(\roman*)}},
ref=\emph{\small(\roman*)}]
\item \label{update:varfix:1}
If $f = f_v$, remove the vertex and connect its parent to its
children.
If, additionally, $v$ has a sibling, then also remove the subtree rooted
at the sibling.
\item \label{update:varfix:2}
Or, if $f = \bar f_v$ and $v$ is a conditional fixing vertex, remove the
vertex and its descendants.
\item \label{update:varfix:3}
Otherwise,
if $f = \bar f_v$ and $v$ is a necessary fixing vertex,
then get the first ancestor $u$ of
$v$ that is a conditional fixing vertex.
If $u$ does not exist, then the tree is marked as infeasible,
such that Line~\ref{alg:propagateallindividualcons:inf} returns
infeasibility directly after this update.
If $u$ does exist, replace $u$ by a
necessary fixing
vertex with associated fixing $\bar f_u$, and remove all its proper
descendants.
A \emph{merging step} is applied if $u$ has a sibling $w$.
This merging step is identical to the merging step at the index increasing
event.
\end{enumerate*}%

For example, after the index increasing event yielding
Figure~\ref{fig:ex:exampletree:indexincreasingstep},
an index increasing event is called for fixing $(7, 0)$.
Entry $7$ is thus added to $I_0^t$,
and the tree of the figure is updated such that the necessary fixing node
of fixing $(7, 0)$ is removed, and the root is connected to the
conditional fixing vertex with fixing $(2, 0)$.

We argue that the tree properties are maintained.
The updates remove whole subtrees, remove all fixing vertices with
entry~$i$, and apply merging steps.
If a leaf vertex~$u$ is not removed at a merging step,
then the only difference of the set $C_u$ before and after any merging
step is the removal of all fixings with entry $i$.
The loose end vertices are either removed, or remain leaves of the tree,
so Property~\ref{prop:looseend} is maintained.
All fixing vertices with entry $i$ are removed,
and~$I_0^{t + 1} \cup I_1^{t + 1} = I_0^t \cup I_1^t \cup \{ i \}$,
so Properties~\ref{prop:entries} and~\ref{prop:rootedpathnotwoentries} are
maintained.

To show Property~\ref{prop:implicationtreeoutdegreetwo},
note that no tree update can increase the degree of a vertex, so it
suffices to check that the property is maintained for each vertex with
outdegree two.
Let~$v$ be a vertex with outdegree larger than one before an update,
and let~$\hat u_1, \hat u_2, \hat w_1, \hat w_2$ be vertices as in
Figure~\ref{fig:prop:implicationtreeoutdegreetwo}, with $\hat v$ as $v$.
If $f_{\hat u_1} = f$, then either
update~\ref{update:varfix:1} removes the subtree rooted in $\hat u_2$,
or
update~\ref{update:varfix:3} is executed for due to $\hat w_2$.
Which of the two updates applies depends on the order in which the
vertices are processed, but the same result is yielded.
If $\bar f_{\hat u_1} = f$, then
update~\ref{update:varfix:2} removes the subtree rooted in $u_1$.
The case for~$f_{\hat u_2} = f$ and~$\bar f_{\hat u_2} = f$ is analogous.
In these cases, after the update the degree of $v$ is one, so the property
holds.
If both~$f_{\hat u_1}$ and $f_{\hat u_2}$ are neither~$f$ nor $\bar f$,
then either the substructure
(of~$v$ and its predecessors at maximal distance 2)
of Figure~\ref{fig:prop:implicationtreeoutdegreetwo} is maintained,
or~\ref{update:varfix:3} yields that a merging step is applied that
changes the degree of $v$ to one.

Last we argue that Properties~\ref{prop:eqconj} and~\ref{prop:infconj} are
maintained after applying all updates.

Let $v \in V_{\mathrm{cond}, \gamma}^t$ be a conditional fixing vertex to
which an update is applied.
On the one hand, if $f_v = f$, then all necessary fixing vertices in its
subtree induce a conjunction~$C$ with $f \in C$.
By removing vertex $v$ and reconnecting its parent to its
children, these conjunctions are updated to $C \setminus \{ f \}$,
consistent with Proposition~\ref{prop:indexinfconj}.
The same holds for conjunctions represented by loose end vertices in the
subtree.
Additionally, if $v$ has a sibling $v'$,
then by Property~\ref{prop:implicationtreeoutdegreetwo}
vertex~$v'$ is a conditional fixing vertex, and
the first child~$u$ of~$v$ is a necessary fixing vertex
with~$f_u = \bar f_{v'}$.
Hence, if all conditional fixings up to $v$ are applied, then $f_u$ must
be applied as well, so it is not possible to apply the fixing associated
with $v'$. Hence, the subtree rooted in $v'$ can be removed,
since it does not contain any minimal conjunction for time~$t+1$.
On the other hand, if~$f = \bar f_v$, then all conjunctions~$C$ induced by
necessary fixing vertices~$u$ have~$\bar f_v \in C$.
However, $C' = \{ f, \bar f \} \in \infconj{\gamma}{t}$,
so $C' \setminus \{ f \} \subsetneq C \setminus \{ f \}$.
Likewise, for loose end vertices in the subtree of $v$ that induce
conjunction~$C$ has $C' \setminus \{ f \} \subseteq C \setminus \{ f \}$.
In both cases, vertices in this subtree induce no inclusionwise minimal
conjunctions, so this subtree can be removed.

Let $v \in V_{\mathrm{necc}, \gamma}^t$ be a necessary fixing vertex to
which an update is applied, denote its associated conjunction by $C$.
On the one hand, if $f = f_v$, then $\bar f \in C$ and
update~\ref{update:varfix:1} is applied.
Because fixing $f$ is applied, conjunction $C$ can never be violated any
more, so $v$ can be removed. Vertex $v$ cannot have siblings due to
Property~\ref{prop:implicationtreeoutdegreetwo},
so this is all the update does.
On the other hand, if $f = \bar f_v$, then $f \in C$ and
update~\ref{update:varfix:3} is applied.
Let $u$ be the first ancestor of $v$ that is either a conditional
fixing vertex, or the root vertex.
If~$u$ is the root vertex, then $\{ f \}$ is an inf-conjunction,
so applying $f$ yields infeasibility for the permutation.
This is correctly marked.
Otherwise, if $u$ is a conditional fixing vertex,
then any necessary fixing vertex in the subtree of $u$ with associated
conjunction $C'$
has~$C \setminus \{ f \} \subsetneq C' \setminus \{ f \} = C'$,
and any loose end vertex in the subtree of $u$ with associated
conjunction $C'$
has~$C \setminus \{ f \} \subseteq C' \setminus \{ f \} = C'$.
Since $C \setminus\{f\}$ is a $i_\gamma^{t+1}$-inf-conjunction,
all these vertices can be removed.
In terms of the tree update, this is phrased as: one may not apply all
conditional fixings on the path to $u$, and~$f_u$ itself.
Thus, we can
replace the subtree of $u$ by a necessary fixing vertex $u_{\mathrm{new}}$
with $f_{u_{\mathrm{new}}} = \bar f_u$,
such that the conjunction associated with this necessary fixing vertex
corresponds to~$C \setminus \{f\}$.
If $u$ has a sibling $u'$, then a merging step is applied
that removes the child $w'$ of $u'$.
By Proposition~\ref{prop:implicationtreeoutdegreetwo}~$w'$ is a necessary
fixing vertex with~$f_{w'} = \bar f_{u} = f_{u_{\mathrm{new}}}$ as
associated fixing.
Since $C \setminus \{ f \}$ is the conjunction associated
with~$u_{\mathrm{new}}$, and because the conjunction associated with~$w'$
is inclusionwise larger than in $C \setminus \{f\}$,
this vertex~$w'$ is correctly removed.

To summarize, all loose end vertices or necessary fixing vertices that
induce a conjunction $C$ represent $C \setminus \{ f \}$.
These are all correct $i_\gamma^{t+1}$-inf-conjunctions or
$i_\gamma^{t+1}$-eq-conjunctions.
Only non-minimal conjunctions of its type are removed, and all
conjunctions in $\infconj{\gamma}{t+1}$ and $\eqconj{\gamma}{t+1}$ are
generated from conjunctions in $\infconj{\gamma}{t}$ and
$\eqconj{\gamma}{t}$, respectively.
Together with Lemma~\ref{lem:noinclusionwisecontainmentsinf}
and~\ref{lem:noinclusionwisecontainmentseq}, no other conjunctions can be
represented by $\tree_\gamma^{t+1}$, therewith proving
Property~\ref{prop:eqconj} and~\ref{prop:infconj}.

\paragraph{Running time analysis}
A timestamp~$t$ is used in the algorithm description,
but this is only required for the analysis, and not for the implementation.
Hence, it suffices to maintain the last state only.
The index sets of the fixings
$(I_0^t, I_1^t)$ are implemented by two Boolean arrays of size $n$
specifying if an entry $i$ is included in the index set or not.
For each permutation $\gamma \in \Pi$, we maintain the tree data structure
and the index $i_\gamma$.

In the following, we work with the implication tree representation and
suppose that the tree updates are applied as described in the previous
paragraphs, and that the sufficient conditions of completeness from
Corollary~\ref{cor:terminationcondition}
are used as selection condition.

Consider~$\gamma \in \Pi$ and time~$t$, and suppose there exist vertices
with outdegree larger than one.
Then Property~\ref{prop:implicationtreeoutdegreetwo} yields that any
rooted path to a leaf contains at least a conditional fixing vertex. The
only way that a new vertex with minimal outdegree two is introduced by the
algorithm is during an index increasing event for $\gamma$ where
$i_\gamma^t$ is no fixed point and $(\alpha, \beta) = (\blank, \blank)$,
which by Property~\ref{prop:entries} yields $i_\gamma^t,
\gamma^{-1}(i_\gamma^t) \notin I_0^t \cup I_1^t$
and~$i_\gamma^t, \gamma^{-1}(i_\gamma^t) \notin
\bigcup_{i' \in [i_\gamma^t-1]} \{i', \gamma^{-1}(i')\}
\setminus (I_0^t \cup I_1^t)$.
In particular, this shows that for all $i' < i_\gamma^t$ we
have~$\gamma^{-1}(i') \neq i_\gamma^t$,
i.e.~$i' \neq \gamma(i_\gamma^t)$,
implying that $\gamma(i_\gamma^t) \geq i_\gamma^t$.
Since $i_\gamma^t$ is no fixed point of $\gamma$,
$\gamma(i_\gamma^t) > i_\gamma^t$ follows.
Also, $\gamma^{-1}(i_\gamma^t) > i_\gamma^t$ follows
since $i_\gamma^t$ is not a fixed point,
and $\gamma^{-1}(i_\gamma^t) \neq i'$ for any $i' < i_\gamma^t$.
Concluding, all conditions of \ref{cor:terminationcondition:C3} of
Corollary~\ref{cor:terminationcondition} hold, meaning that
sufficient conditions
for completeness of $\gamma$ are met. Hence, the algorithm will never
introduce a second vertex with minimal outdegree two.
This means that the implication tree can be encoded by two paths.
Moreover, because of Property~\ref{prop:rootedpathnotwoentries} we can
encode these two paths as two separate arrays of length $n$.

As a result, one only needs to store the root vertex,
and for each of the two possible branches one vertex per entry for the
fixing vertices, and a loose end vertex.
To allow for constant-time lookups and insertions of the fixing vertices
where the entry has a certain fixing, two arrays of size $n$ are
maintained, in which at entry $i$ (necessary or conditional) fixing
vertices are stored whose fixings have entry $i$ in either path.

Furthermore,
the algorithm maintains $\queue_F$,
which consists of a subset of fixings that must be applied as a result of
Line~\ref{alg:propagateallindividualcons:whilefix} of the algorithm.
Specifically, this set is empty only if the condition of
Line~\ref{alg:propagateallindividualcons:whilefix} is not met,
which means that valid fixings can never be missed.
Each time an implication tree for a permutation is updated,
it is checked if the root has a necessary fixing vertex as child.
If that is the case, then this fixing is added to $\queue_F$, if not
already there.
This is a constant-time check, since the root has at most two children due
to Property~\ref{prop:implicationtreeoutdegreetwo}.
By implementing $\queue_F$
as a stack and a Boolean direct-lookup array, pushing
specific elements to the set, popping the lastly added element from
the set,
and checking if an element is contained are constant-time operations.
The while-loop of
Line~\ref{alg:propagateallindividualcons:whilefix}
corresponds to iteratively applying virtual fixings popped from $\queue_F$
until it is empty.

Likewise, the set $\queue_\Pi$ is the set of permutations for
which the sufficient conditions of completeness of
Corollary~\ref{cor:terminationcondition} do not hold.
These conditions are checked for a permutation $\gamma \in \Pi$
if the index $i_\gamma$ increases,
or if a fixing is applied.
The conditions can be checked in constant time:
For~\ref{cor:terminationcondition:C1} there are at most two leaves since
there is at most one vertex with outdegree two and loose end vertices are
always a leaf of the tree;
and for \ref{cor:terminationcondition:C3} this corresponds with checking
if the root vertex only has conditional fixing vertices as child,
of which at most two could exist due to
Property~\ref{prop:implicationtreeoutdegreetwo}.

At initialization,
consistent with the initial state of the algorithm,
$\queue_F$ is empty and~$\queue_\Pi$ consists of all permutations.

Since loose end vertices are leaves, and the tree only has at most a
single vertex with outdegree 2 (and the rest having smaller outdegree),
implication trees have at most two loose end vertices.
At an index increasing event the loose end vertices are replaced by
subtrees having at most six vertices. Hence, a constant number of vertices
is introduced at an index increasing event.
Introduced vertices can be removed at most once,
and (at a merging step) changed from a conditional fixing vertex type to a
necessary fixing type. Transforming the vertex type back is no option, so
this can happen at most once, as well.
This means that tree updates at the index increasing or variable fixing
events take amortized constant time.

As one of the sufficient conditions of completeness is that
$i_\gamma^t > n$, and $i_\gamma^t$ increases by one for every index
increasing event and never decreases, the index increasing events can only
be called $n$ times per permutation,
so that is at most $n |\Pi|$ times in total.
Likewise, at most $n$ different fixings can be applied before
infeasibility is found,
so at most~$n + 1$ variable fixing events can occur at most.

Also, the bottleneck of initialization is the preparation of the constant
number of arrays of size $n$ for each permutation.
Hence, collecting all observations above,
the running time of this algorithm is $\bigo(n |\Pi|)$.
Specifically, with this algorithm specification the bound is tight,
as this corresponds to the space requirement at initialization.

\section{Propagating Lexicographic Orders for Cyclic Groups}
\label{sec:lexleadersgroup}
The propagation algorithm of the previous section ensures that solutions
are lexicographically not smaller than their image w.r.t.\ the
permutations in a set~${\Pi \subseteq \symmetricgroup{n}}$.
This is enforced by a propagation loop, which
determines valid variable
fixings for each individual constraint~$x \succeq \gamma(x)$, for
each~$\gamma \in \Pi$.
The derived set of fixings, however, is not necessarily complete,
as illustrated next:%
\begin{example}
\label{ex:indivpropnotcomplete}
Let $\tilde\gamma = (1, 2, 3, 4, 5)$ and $\Pi = \gen{\tilde\gamma}$.
Consider the set of fixings~$I_0 = \{2, 5\}$ and~$I_1 = \emptyset$.
The set of vectors that are feasible for all permutations are~%
$\xymretope\Pi \cap F(I_0, I_1) =
\{ (0, 0, 0, 0, 0), (1, 0, 0, 0, 0), (1, 0, 1, 0, 0) \}$,
which means that the fourth entry is a zero-fixing
in~$\xymretope\Pi \cap F(I_0, I_1)$.
However, this entry cannot be detected by the propagation loop,
as~$(1, 0, 1, 1, 0) \in \xymretope\delta$
for~$\delta \in \{ \gamma^1, \gamma^2, \gamma^4 \}$,
and~$(1, 0, 0, 1, 0) \in \xymretope{\gamma^3}$.
This shows that~$I_0$ and~$I_1$ cannot be extended if we consider the
permutations of~$\Pi$ independently.
\end{example}

Consequently, there are variable fixings that cannot be detected by
treating the elements in a set of permutations~$\Pi$ individually.
In particular, this is also not the case if~$\Pi$ forms a cyclic group.
The aim of this section is to close this gap for cyclic groups, i.e.,
devising algorithms finding the complete set of fixings for different
types of cyclic groups.

If all entries are fixed initially, a complete propagation
algorithm for lexicographic leaders in a group reduces to determine
if the single possible vector adhering to the fixings is the lexicographic
leader in its $\Gamma$-orbit.
Babai and Luks~\cite{babai1983canonical} show that this is
\coNP-complete.
This means one cannot expect to find a polynomial-time algorithm to find
the complete set of fixings of lexicographic leaders in the
$\Gamma$-orbits, for general symmetry groups $\Gamma \leq
\symmetricgroup{n}$, unless $\P = \coNP$.
Luks and Roy~\cite{luksroy2002symbreakformulas} also show that, although
the problem of testing elements for being lexicographically maximal in
their group orbits is \coNP-complete, for Abelian groups there always
exists an ordering of the group support such that this problem becomes
polynomial. Roy~\cite{roy2007symmetry} gives a generalization of this
result for a wider class of groups.

In the following, we provide two propagation algorithms.
First, we consider one for groups generated by a single cycle
that has a \emph{monotone representation}, or subgroups of such groups.
After that, groups are considered that have a generator with a monotone
and \emph{ordered} representation.
What monotone and ordered representation embraces is made precise below.
Although these requirements do not allow to handle all cyclic groups, we
will discuss in Section~\ref{sec:computational} how our algorithms can
still be used to handle symmetries in arbitrary binary programs.
That is, we describe how the technical assumptions on the groups can be
guaranteed for every binary program.

\subsection{Cyclic Group Generated by Monotone Cycle or Subgroup Hereof}
\label{sec:lexleadersgroup:circular}
Let $\Gamma \leq \symmetricgroup{n}$
and two disjoint sets $I_0, I_1 \subseteq [n]$ defining initial fixings.
To find the complete set of fixings for lexicographic leaders in the
$\Gamma$-orbit, one needs to determine
the maximal sets~$J_0, J_1 \subseteq [n]$ such that~%
$\xymretope\Gamma \cap F(I_0, I_1) = \xymretope\Gamma \cap F(J_0, J_1)$.
The simple Algorithm~\ref{alg:completesinglemonotonecycle}
shows how to find the complete set of fixings for lexicographic leaders in
the $\Gamma$-orbit.
Starting with the initial set of fixings $(I_0, I_1)$,
the algorithm extends this set with all necessary fixings for completeness
for the constraint~$x \succeq \gamma(x)$ for all~$\gamma \in \Gamma$,
and then checks for each remaining non-fixed entry if any
lexicographically maximal vector exists in its $\Gamma$-orbit
where the non-fixed entry is fixed to zero or to one.
To simplify the analysis of
Algorithm~\ref{alg:completesinglemonotonecycle}, we maintain a
timestamp~$t$, starting at 0,
and
the set of fixings at time~$t$ is denoted by~$I_0^t$ and~$I_1^t$.

Fixings valid for each individual constraint $x \succeq \gamma(x)$
for $\gamma \in \Gamma$ are also valid fixings
in~$\xymretope\Gamma \cap F(I_0, I_1)$,
so the set $(I_0^t, I_1^t)$ after
Line~\ref{alg:completesinglemonotonecycle:determine}
is a valid set of fixings for $\xymretope\Gamma \cap F(I_0, I_1)$.
Also, since~%
$
\xymretope\Gamma \cap F(I_0^t, I_1^t)
=
\left(\xymretope\Gamma \cap F(I_0^t \cup \{ i \}, I_1^t) \right)
\dot\cup
\left(\xymretope\Gamma \cap F(I_0^t, I_1^t \cup \{ i \}) \right)
$,
the updates of Line~\ref{alg:completesinglemonotonecycle:peek0}
and~\ref{alg:completesinglemonotonecycle:peek1}
ensure that
$
\xymretope\Gamma \cap F(I_0^t, I_1^t)
=
\xymretope\Gamma \cap F(I_0, I_1)
$
for all $t \geq 0$.
Hence, the algorithm never applies incorrect fixings.
If no fixing is applied in an iteration for index~$i$,
then~%
$\xymretope\Gamma \cap F(I_0^t \cup \{ i \}, I_1^t) \neq \emptyset$
and~%
$\xymretope\Gamma \cap F(I_0^t, I_1^t \cup \{ i \}) \neq \emptyset$
certify the existence of two vectors~%
$x^0, x^1 \in \xymretope\Gamma \cap F(I_0, I_1)$
with $x^0_i = 0$ and $x^1_i = 1$,
respectively.
The algorithm iterates over all indices $i \in [n]$
with $i \notin I_0^0 \cup I_1^0$,
so no valid fixing could be missed.
This shows that the algorithm is correct.

\begin{algorithm}[!tbp]
\caption{Find the complete set of fixings of
$\xymretope\Gamma \cap F(I_0, I_1)$
with $\Gamma \leq \symmetricgroup{n}$}
\label{alg:completesinglemonotonecycle}
\SetKwInOut{Input}{input}
\SetKwInOut{Output}{output}
\Input{%
group~$\Gamma \leq \symmetricgroup{n}$,
sets~$I_0, I_1 \subseteq [n]$}
\Output{%
the message \Infeasible, or \Feasible and two subsets of~$[n]$
}
\lIf{$\xymretope\Gamma \cap F(I_0, I_1) = \emptyset$}{\Return \Infeasible}
\label{alg:completesinglemonotonecycle:inf}
$t \gets 0$;\
$
(I_0^t, I_1^t) \gets \text{the complete set of fixings for each individual
constraint}\
x \succeq \gamma(x)\
\text{for all}\
\gamma \in \Gamma
$\;
\label{alg:completesinglemonotonecycle:determine}
\ForEach{$i \in [n] \setminus (I_0^0 \cup I_1^0)$}
{
\lIf{$\xymretope\Gamma \cap F(I_0^t \cup \{ i \}, I_1^t) = \emptyset$}{
$(I_0^{t + 1}, I_1^{t + 1}) \gets (I_0^{t}, I_1^{t} \cup \{ i \})$;
$t \gets t + 1$%
}
\label{alg:completesinglemonotonecycle:peek0}
\lIf{$\xymretope\Gamma \cap F(I_0^t, I_1^t \cup \{ i \}) = \emptyset$}{
$(I_0^{t + 1}, I_1^{t + 1}) \gets (I_0^{t} \cup \{ i \}, I_1^{t})$;
$t \gets t + 1$%
}
\label{alg:completesinglemonotonecycle:peek1}
}
\Return \Feasible, $(I_0^t, I_1^t)$\;
\end{algorithm}

In the Lines~\ref{alg:completesinglemonotonecycle:inf},%
~\ref{alg:completesinglemonotonecycle:peek0}
and~\ref{alg:completesinglemonotonecycle:peek1}
it is checked whether a vector exists that is lexicographically maximal in
the $\Gamma$-orbit, and satisfies certain fixings.
Recall the result by Babai and Luks~\cite{babai1983canonical},
which shows that a restriction of this problem is \coNP-complete.
Therefore, one cannot expect a polynomial-time realization of this
algorithm for general groups, unless~$\P = \coNP$.
In the remainder of this section,
we restrict ourselves to groups generated by a
\emph{monotone} cycle (or subgroups hereof),
for which we show a polynomial-time realization of
Algorithm~\ref{alg:completesinglemonotonecycle}.
A cycle $\zeta$ is called \emph{monotone} if the cycle has exactly one
entry $i$ with~$\zeta(i) < i$.
Let~$\Gamma \leq \gen{(1, \dots, n)}$.
With the following proposition, we argue how the checks of
Line~\ref{alg:completesinglemonotonecycle:inf},%
~\ref{alg:completesinglemonotonecycle:peek0}
and~\ref{alg:completesinglemonotonecycle:peek1}
can be applied efficiently.

\begin{proposition}
\label{prop:monotonecyclicgroupfeasibility}
Let~$I_0, I_1 \subseteq [n]$ be disjoint index sets defining a complete
set of fixings for the constraint~$x \succeq \gamma(x)$
for each permutation~$\gamma \in \Gamma \leq \gen{(1, \dots, n)}$.
The set~$\xymretope\Gamma \cap F(I_0, I_1)$ is non-empty if and only if
the sets~$\xymretope\gamma \cap F(I_0, I_1)$ are non-empty for every
permutation~$\gamma \in \Gamma$.%
\end{proposition}

The checks of
Line~\ref{alg:completesinglemonotonecycle:inf},%
~\ref{alg:completesinglemonotonecycle:peek0}
and~\ref{alg:completesinglemonotonecycle:peek1}
can be executed using
Proposition~\ref{prop:monotonecyclicgroupfeasibility},
by computing the set of fixings that is complete for
each individual constraint $x \succeq \gamma(x)$,
for each~$\gamma \in \Gamma$.
This is possible, for instance, with
Algorithm~\ref{alg:propagateallindividualcons} of
Section~\ref{sec:completepermset} in $\bigo(n \ord(\Gamma))$ time.
As a benefit, the result of that algorithm gives $(I_0^0, I_1^0)$,
which satisfies
$
\xymretope\Gamma \cap F(I_0, I_1)
=
\xymretope\Gamma \cap F(I_0^0, I_1^0)
$,
as needed in Line~\ref{alg:completesinglemonotonecycle:determine}.
There are at most $n$ unfixed entries, so the subroutine that determines
the complete set of fixings for each permutation in $\Gamma$ is called at
most $1 + 2n$ times, taking~$\bigo(n \ord(\Gamma))$ time
when using Algorithm~\ref{alg:propagateallindividualcons}.
Hence, the total running time is $\bigo(n^2 \ord(\Gamma))$.
Recall that $\Gamma \leq \gen{(1, \dots, n)}$,
such that $\ord(\Gamma) \leq n$,
so that the running time is polynomially bounded.

\begin{remark}
Fewer calls to the subroutine are required if a
vector~$\tilde x \in \xymretope\Gamma \cap F(I_0, I_1)$ can be determined.
If~$\tilde x_i = 0$,
we can avoid checking
if~$\xymretope\Gamma \cap F(I_0^t \cup \{ i \}, I_1^t)$ is empty or not.
The same argument can be used for~$\tilde x_i = 1$.
\end{remark}

Last, for the missing proof of
Proposition~\ref{prop:monotonecyclicgroupfeasibility},
we make repeated use of the following argument:
\begin{remark}
\label{remark:completenessimplication}
Let $I_0, I_1 \subseteq [n]$ be disjoint index sets defining a complete
set of fixings for the constraint $x \succeq \gamma(x)$
for each permutation~$\gamma \in \Gamma$.
Let~$\gamma \in \Gamma$
with~$\xymretope\gamma \cap F(I_0, I_1) \neq \emptyset$,
and $k \in [n]$.
Suppose that for all~$i < k$
we either have $i, \gamma^{-1}(i) \in I_0$ or $i, \gamma^{-1}(i) \in I_1$.
This means that for all $x \in F(I_0, I_1)$ and $i < k$
we have $x_i = \gamma(x)_i$, i.e.\@~$x \eqp{k} \gamma(x)$.
Since the set of fixings is complete
with respect to~$x \succeqp \gamma(x)$ for~$\gamma$,
and for any $x \in F(I_0, I_1)$ the vectors~$x$ and~$\gamma(x)$
are identical up to entry $k$:
If~$k \in I_0$, then~$\gamma^{-1}(k) \in I_0$.
Likewise, if~$\gamma^{-1}(k) \in I_1$, then~$k \in I_1$.
\end{remark}

\begin{proof}[%
Proof of Proposition~\ref{prop:monotonecyclicgroupfeasibility}]
Let~$I_0, I_1 \subseteq [n]$ be disjoint sets that define
for each~$\gamma \in \Gamma$
a complete set of fixings for~$\xymretope\gamma \cap F(I_0, I_1)$.
For brevity, we denote~$F \define F(I_0, I_1)$.
By definition,~$\xymretope\Gamma = \bigcap_{\gamma \in \Gamma}
\xymretope\gamma$.
Hence, if there exist~$\gamma \in \Gamma$
with~$\xymretope\gamma \cap F = \emptyset$, then~$\xymretope\Gamma \cap F
= \emptyset$ as
well, proving the first implication.
In the remainder of the proof, we show the converse implication.
Suppose that~$\xymretope\gamma \cap F$ is non-empty for all
permutations~$\gamma \in \Gamma$. We now show that there exists a
vector~$\tilde x \in \xymretope\Gamma \cap F$.
First, if~$I_0 \cup I_1 = [n]$, there are no unfixed
entries. Thus the set~$F$ consists of a single element~$\tilde x$.
By assumption,~$\xymretope\gamma \cap F \neq \emptyset$ for
every~$\gamma \in \Gamma$, so~$\tilde x \in \xymretope\gamma$,
meaning that~$\xymretope\Gamma \cap F
= \bigcap_{\gamma \in \Gamma} \xymretope\gamma \cap F
= \{ \tilde x \} \neq \emptyset$.

Second, suppose~$I_0 \cup I_1 \subsetneq [n]$.
Let~%
$\ihat
\define
\min\{i : i \in [n] \setminus (I_0 \cup I_1 )\}$
be the first unfixed entry of the vectors in~$F$,
and for any~$\gamma \in \Gamma$
let~$\icheck_\gamma
\define
\min\{ \gamma^{-1}(i) : i \in [n] \setminus (I_0 \cup I_1) \}$
be the first unfixed entry of a vector~$\gamma(x)$ for vectors~$x \in F$.
If the permutation is clear
from the context, we drop the subscript of~$\icheck_\gamma$.
Let~$\tilde x \in F$
with~$\tilde x_i = 1$ if~$i \in I_1 \cup \{ \ihat \}$,
and~$\tilde x_i = 0$ otherwise.
We claim that~$\tilde x \in \xymretope\Gamma \cap F$,
which completes the proof.

For the sake of contradiction, assume there is a $\gamma \in \Gamma$
with~$\tilde x \prec \gamma(\tilde x)$.
Then, there is~$i \in [n]$ such that~$\tilde x \eqp{i} \gamma(\tilde x)$
and~$\tilde{x}_i = 0 < \perm(\tilde{x})_i = 1$.
Since~$\xymresack{\perm} \cap F \neq \emptyset$, we thus have~$i \geq
\min\{\ihat, \icheck\}$.
That is, $k, \inv{\perm}(k) \in I_0 \cup I_1$ for all~$k < m
\define \min \{\ihat, \icheck\}$,
and for all $x \in F$ holds that $x \eqp{m} \gamma(x)$.
Note that therefore~$\tilde{x} \succ \perm(\tilde{x})$ if~$\tilde{x}_m =
1$ and~$\perm(\tilde{x})_m = 0$.
We will use this observation to find a contradiction in the following.

If $\ihat \leq \icheck$,
then $\gamma^{-1}(\ihat) \notin I_1$, because otherwise the complete
propagation algorithm for~$\xymretope\gamma \cap F$ had fixed~$\ihat$
to~$1$: see Remark~\ref{remark:completenessimplication}.
Since $\gamma$ has no fixed points because it is a cyclic shift,
$\gamma^{-1}(\ihat) \neq \ihat$, so $\gamma(\tilde x)_{\ihat} = 0$.
This implies $\tilde x_{\ihat} > \gamma(\tilde x)_{\ihat}$,
contradicting $\tilde x \prec \gamma(\tilde x)$.
Hence,~$\ihat > \icheck$,
and thus~$\icheck \in I_0 \cup I_1$.
If~$\icheck \in I_0$, then $\gamma^{-1}(\icheck)$ is necessarily contained
in $I_0$, as otherwise any $x \in F$
with~$\gamma(x)_{\icheck} = 1$
satisfies~$x \eqp{\icheck} \gamma(x)$
and~$x_{\icheck} < \gamma(x)_{\icheck}$,
which contradicts completeness of the fixings with respect to $\gamma$.
Thus, in the remainder we have $\ihat > \icheck$ and $\icheck \in I_1$.%

If~$\gamma(\tilde x)_{\icheck} = 0$, then~$\tilde x \succ \gamma(\tilde x)$
which contradicts the assumption
that~$\tilde x \prec \gamma(\tilde x)$.
Therefore~$\gamma(\tilde x)_{\icheck} = 1$, which means that~%
$\gamma^{-1}(\icheck) \in I_1 \cup \{ \ihat \}$, and by the definition
of~$\icheck$ only~$\gamma^{-1}(\icheck) = \ihat$ remains.
Let~$j \define \min\{ i \in [n] : \tilde x_i \neq \gamma(\tilde x)_i \}$
be the
first entry where the values of~$\tilde x$ and~$\gamma(\tilde x)$ differ.
Then $\tilde x_j = 0$ by $\tilde x \prec \gamma(\tilde x)$.
It is not possible that~$j = \ihat$,
as~$\tilde x_j = 0 \neq 1 = \tilde x_{\ihat}$.
Using the following claim,
we show that~$j \neq \ihat$ also yields a contradiction.
\begin{claim*}
If~$\ihat > \icheck$,
any~$x \in \xymretope\gamma \cap F$ with~$x_{\ihat} = 1$
has~$x_i = \tilde x_i$ for all~%
$i \leq \min\{ n, j + \ihat - \icheck - 1\}$.
\end{claim*}
\begin{claimproof}[Proof of claim]
Let~$x \in \xymretope\gamma \cap F$ with~$x_{\ihat} = 1$.
To prove the claim,
we show that~$x_i = \tilde x_i$ for all~$i \leq b$,
where~$b \leq \min\{ n, j + \ihat - \icheck - 1 \}$.
We proceed by induction on $b$.
Because~$x, \tilde x \in F$, by~$x_{\ihat} = \tilde x_{\ihat} = 1$
and the definition of~$\ihat$ we have~$x_i = \tilde x_i$
for all~$i \leq \ihat$,
so the statement holds for all $b \leq \ihat$.

Consider~$b \leq \min\{ n, j + \ihat - \icheck - 1 \}$
and~$b > \ihat$,
and suppose that for all~$i \leq b-1$
we have~$x_i = \tilde x_i$~{\small\bf(IH)}.
By the previous paragraph, this holds for $b = \ihat + 1$.
We show that~$x_b = \tilde x_b$,
so that the claim follows by induction.
If $b \in I_0 \cup I_1$, then $x_b = \tilde x_b$ holds trivially
because entry $b$ is fixed,
so in the remainder we assume~$b \notin I_0 \cup I_1$.
The following chain of equations holds
due to the arguments presented earlier:
\begin{equation}
\label{prop:monotonecyclicgroupfeasibility:proof:claim:eq}
x_{b - \ihat + \icheck}
\overset{\text{\bf(IH)}}=
\tilde x_{b - \ihat + \icheck}
\overset{{(b - \ihat + \icheck < j)}}=
\gamma(\tilde x)_{b - \ihat + \icheck}
= \tilde x_{\gamma^{-1}(b - \ihat + \icheck)}
\overset{{\bf(*)}}=
\tilde x_b.
\end{equation}
In this, ${\bf(*)}$ uses that~$\gamma$ is a cyclic shift
with~$\gamma^{-1}(\icheck) = \ihat$,
such that~$\gamma^{-1}(i) = i + \ihat - \icheck$
for~$0 < i \leq n - \ihat + \icheck$.
Observe that~$b - \ihat + \icheck \leq n - \ihat + \icheck$
due to~$b \leq n$,
and~$b - \ihat + \icheck > \icheck > 0$ due to~$b > \ihat$.

If~$x_{b - \ihat + \icheck} = 1$,
then Equation~\eqref{prop:monotonecyclicgroupfeasibility:proof:claim:eq}
yields~$\tilde x_b = 1$.
Thus,~$b \in I_1 \cup \{ \ihat \}$ by the definition of~$\tilde x$.
As~$b > \ihat$,~$b \in I_1$.
Consequently,~$x_b = \tilde x_b = 1$.

Otherwise, if~$x_{b - \ihat + \icheck} = 0$,
the induction hypothesis yields~$x \eqp{b - \ihat + \icheck} \tilde x$,
as~$\ihat > \icheck$.
Also~$b < j + \ihat - \icheck$, so~$b - \ihat + \icheck < j$.
Because~$j$ is the first entry where~$\tilde x$ and~$\gamma(\tilde x)$
differ,~$\tilde x \eqp{b - \ihat + \icheck} \gamma(\tilde x)$.
Since~$\icheck$ is the first entry where the entries of the vectors
after permuting with~$\gamma$ are not fixed,~$x, \tilde x \in F$
imply~$\gamma(x)_i = \gamma(\tilde x)_i$ for all~$i < \icheck$.
Last, if~$\icheck \leq i < b - \ihat + \icheck \leq n - \ihat + \icheck$,
the induction hypothesis yields~%
$\gamma(x)_i
= x_{i + \ihat - \icheck}
= \tilde x_{i + \ihat - \icheck}
= \gamma(\tilde x)_i$.
Combining these results yields~$x \eqp{b - \ihat + \icheck} \gamma(x)$.
Since~$x \in \xymretope\gamma$, this means that~$x \succeq \gamma(x)$,
and thus from~$x \eqp{b - \ihat + \icheck} \gamma(x)$
and~$x_{b - \ihat + \icheck} = 0$
follows~$\gamma(x)_{b - \ihat + \icheck} = x_b = 0$.
Concluding,
Equation~\eqref{prop:monotonecyclicgroupfeasibility:proof:claim:eq}
yields~$0
= x_b
= \gamma(x)_{b - \ihat + \icheck}
= x_{b - \ihat + \icheck}
\stackrel{\eqref{prop:monotonecyclicgroupfeasibility:proof:claim:eq}}=
\tilde x_b$.
In each case~$x_b = \tilde x_b$ is found,
so by induction the claim follows.
\end{claimproof}

Recall that $j \neq \ihat > \ihat$,
and that $\gamma^{-1}(\icheck) = \ihat$.
As completeness of the fixing for permutation~$\gamma$ is assumed,
and~$\ihat \notin I_0 \cup I_1$,
there always exists a vector~$x \in \xymretope\gamma \cap F$
with~$x_{\ihat} = 1$.
If the minimum in the claim evaluates to~$n$, then~$x = \tilde x$,
so~$\tilde x \in \xymretope\gamma \cap F$,
contradicting the assumption that~$\tilde x \prec \gamma(\tilde x)$.
Otherwise, if the minimum evaluates to~$j + \ihat - \icheck - 1$,
then~$x \eqp{j + \ihat - \icheck} \tilde x$,
so especially $x_j = \tilde x_j = 0$, as a result of $\ihat > \icheck$.
Moreover,
because~$\gamma$ is a cyclic shift with~$\gamma^{-1}(\icheck) = \ihat$,~%
$\gamma(x) \eqp{j} \gamma(\tilde x)$.
Also,~$\gamma(\tilde x)_j = 1$,
so~$\gamma^{-1}(j) \in I_1 \cup \{ \ihat \}$.
From~$\tilde x_j = 0 \neq \tilde x_{\ihat} = 1$
follows~$\gamma^{-1}(j) \neq \ihat$, so~$\gamma^{-1}(j) \in I_1$.
It follows
from~$\tilde x \eqp{j} \gamma(\tilde x)$ and~$\ihat > \icheck$
that~%
$x \eqp{j} \tilde x \eqp{j} \gamma(\tilde x) \eqp{j} \gamma(x)$,
$x_j = 0$ and $\gamma(x)_j = 1$,
such that~$x \prec \gamma(x)$. This contradicts~$x \in \xymretope\gamma$.
So every case with $\tilde x \prec \gamma(\tilde x)$ yields a
contradiction, proving that~$\tilde x \in \xymretope\gamma$.
\end{proof}

\subsection{Cyclic Group Generated by Ordered and Monotone Subcycles}
In this
section, a complete propagation algorithm for lexicographic leaders of the
$\Gamma$-orbit is presented, where $\Gamma$ is generated by a composition
$\tilde\gamma = \zeta_1 \circ \dots \circ \zeta_m$
of \emph{monotone} and \emph{ordered} subcycles,
possibly of different lengths.
That is, for each~$i \in [m]$, the subcycle~$\zeta_i$ is monotone, so there is
only a single entry~$k$ with~$\zeta_i(k) < k$.
Also, the subcycles are \emph{ordered}, which means that
for each~$i, j \in [m]$ with $i < j$, all entries in the support of
$\zeta_i$ are smaller than the entries in the support of $\zeta_j$.
If $\tilde\gamma$ has no fixed points, one can write
$\tilde\gamma = (1,\dots,z_1)(z_1+1,\dots,z_2)\cdots(z_{m-1}+1,\dots,z_m)$
for some $1 < z_1 < \cdots < z_m = n$.
Since this case allows to deal with subcycles of different length, this
case can handle more types of cyclic groups than the variant of the
previous section.
To derive an efficient algorithm for such groups, we need a stronger
version of Proposition~\ref{prop:monotonecyclicgroupfeasibility},
which we discuss next.
Afterwards, we proceed with the discussion of the algorithm.

\begin{proposition}
\label{prop:monotonecyclicgroupfeasibilitystrict}
Let $I_0, I_1 \subsetneq [n]$
with $I_0 \cup I_1 \subsetneq [n]$
be disjoint index sets defining a complete
set of fixings for every permutation in
group~$\Gamma \leq \gen{(1, \dots, n)}$.
If $\xymretope\Gamma \cap F(I_0, I_1)$ is non-empty,
then there exists $x \in F(I_0, I_1)$
such that $x \succ \gamma(x)$ for all
$\gamma \in \Gamma \setminus \{ \id \}$.
\end{proposition}
\begin{proof}

Let $\ihat \define \min \{i \in [n] : i \notin I_0 \cup I_1 \}$ be the
smallest unfixed entry.
We provide two constructions of a vector
$\tilde x \in \xymretope\Gamma \cap F(I_0, I_1)$
with $\tilde x \succ \gamma(\tilde x)$:
one for $\ihat \leq \frac12 n$, and one for~$\ihat > \frac12 n$.

\begin{claim}
If~$\ihat > \frac12 n$,
then~$\tilde x \in F(I_0, I_1)$
with~$\tilde x_i = 1$ if~$i \in I_1$ and~$\tilde x_i = 0$ if~$i \notin I_1$
satisfies~$\tilde x \succ \gamma(\tilde x)$
for all~$\gamma \in \Gamma \setminus \{ \id \}$.
\end{claim}
\begin{claimproof}
We prove the claim by contradiction in two parts.
First, suppose~$\tilde x = \gamma(\tilde x)$ for some
$\gamma \in \Gamma \setminus \{ \id \}$.
Then, $\tilde x = \delta(\tilde x)$ for all~$\delta \in \gen\gamma$, and
since~$\gamma \neq \id$ is a cyclic shift, there exists~$\delta \in
\gen\gamma$ with~$\delta(\ihat) \leq \frac12 n < \ihat$.
For this reason, we may assume w.l.o.g.\ that
$\gamma(\ihat) \leq \frac12 n$.
From this and
$\tilde x_{\gamma(\ihat)}
= \gamma(\tilde x)_{\gamma(\ihat)}
= \tilde x_{\ihat} = 0$,
we conclude~$\gamma(\ihat) \in I_0$.
As permutation~$\gamma$ is a cyclic shift,
$\gamma^{-1}(i) = i + \ihat - \gamma(\ihat) \pmod{n}$,
wherein the modular residual classes identify with~$\{1, \dots, n\}$.
From~$1 \leq \gamma(\ihat) \leq \frac12 n < \ihat \leq n$,
we have that~$1 \leq \gamma^{-1}(i) < \ihat$
for all~$1 \leq i < \gamma(\ihat)$,
showing that~$i, \gamma^{-1}(i) \in I_0 \cup I_1$
for all~$i < \gamma(\ihat)$.
By the assertion~$\tilde x = \gamma(\tilde x)$,
all vectors~$x \in F(I_0, I_1)$
thus have~$x \eqp{\gamma(\ihat)} \gamma(x)$,
and~$x_{\gamma(\ihat)} = 0$ by~$\gamma(\ihat) \in I_0$.
But~$\ihat \notin I_0 \cup I_1$,
and any~$x \in F(I_0, I_1)$ with~$x \succeq \gamma(x)$
must thus have~
$0 = x_{\gamma(\ihat)} \geq \gamma(x)_{\gamma(\ihat)} = x_{\ihat}$,
such that $x_{\ihat} = 0$, which violates
completeness for permutation~$\gamma$.
Hence~$\tilde x \neq \gamma(\tilde x)$.

Second, suppose~$\tilde x \prec \gamma(\tilde x)$
for some~$\gamma \in \Gamma \setminus \{ \id \}$.
Let~\mbox{$j \define \min\{ i \in [n] : \tilde x_i \neq \gamma(\tilde x)_i \}$}
and let
$\icheck \define \min\{ i \in [n] : \gamma^{-1}(i) \notin I_0 \cup I_1 \}$
be the first non-fixed entry in a vector permuted by~$\gamma$.
We derive some properties of such vectors~$x \in F(I_0, I_1)$.
Every vector~$x \in F(I_0, I_1)$ satisfies~$x_i = \tilde x_i$ for~$i < \ihat$
and~$\gamma(x)_i = \gamma(\tilde x)_i$ for~$i < \icheck$.
Moreover, $0 = \tilde x_j < \gamma(\tilde x)_j = 1$,
showing~$j \notin I_1$ and~$\gamma^{-1}(j) \in I_1$.
As~$\xymretope\Gamma \cap F(I_0, I_1)$ is non-empty
and~$\tilde x \prec \gamma(\tilde x)$,~%
$j \geq \min\{ \ihat, \icheck \}$.
Thus, $x \eqp{\min\{ \ihat, \icheck \}} \gamma(x)$
for all~$x \in F(I_0, I_1)$.
We distinguish three cases.

On the one hand, suppose~$\icheck < \min\{\ihat, j\}$.
Then there is~$x \in F(I_0, I_1)$
with~$\gamma(x)_{\icheck} = 1$ and~$x \succeq \gamma(x)$.
But~$\icheck < \ihat$ implies~$x_{\icheck}
= \tilde x_{\icheck}
= \gamma(\tilde x)_{\icheck} = 0$
and thus $\icheck \in I_0$.
So all vectors $x \in F(I_0, I_1)$ have $x \eqp{\icheck} \gamma(x)$,
and $x_{\icheck} = 0$,
but if $\gamma(x)_{\icheck} = 1$, then $x \prec \gamma(x)$.
To ensure $x \succeq \gamma(x)$, entry $\gamma^{-1}(\icheck)$ must
be fixed to zero as well,
which is in contradiction with completeness of the fixings
and~$\gamma^{-1}(\icheck) \notin I_0 \cup I_1$.
Hence~$\icheck \geq \min\{ \ihat, j \}$,
and we conclude from~$j \geq \min\{ \ihat, \icheck \}$
that~$j \geq \ihat$ and~$\icheck \geq \ihat$.

On the other hand, if~$j = \ihat$, then~$\gamma(\tilde x)_{\ihat} = 1$,
i.e., $\gamma^{-1}(\ihat) \in I_1$,
Then, completeness yields~$\ihat \in I_0$, which is a contradiction.
Consequently, $j > \ihat$.
Since~$\ihat > \frac12 n$, there is exactly one sequence of~$\ihat - 1$
consecutive numbers modulo~$n$ in~$I_0 \cup I_1$ followed by a
non-contained number. This sequence is~$1, \dots, \ihat$.
In the vector permuted by~$\gamma$, this sequence is mapped to~$1 +
(\icheck - \ihat), \dots, \ihat + (\icheck - \ihat)$, because~$\gamma$ is a
cyclic shift, $\icheck$ is the first non-fixed entry
in the vector permuted by $\gamma$, and~$\icheck \geq \ihat > \frac12n$.
Hence,~$\gamma^{-1}(\icheck) = \ihat$,
and thus~$\gamma^{-1}(i) = i - \icheck + \ihat \pmod{n}$.
Then~$\icheck > \ihat$ follows from~$\gamma \neq \id$,
so~$\icheck > \ihat$ and~$j > \ihat$.

Because~$j > \ihat$ and~$\icheck > \ihat$,
any vector~$x \in F(I_0, I_1)$
has~$x \eqp{\ihat} \gamma(x)$.
From this we conclude~$\gamma^{-1}(\ihat) \notin I_1$, because otherwise,
completeness of~$I_1$ implied~$\ihat \in I_1$ to ensure~$x
\succeqp{\ihat+1} \gamma(x)$, contradicting the definition of~$\ihat$.
Thus, $\gamma^{-1}(\ihat) \in I_0$ because~$\icheck > \ihat$.
However, this leads to a contradiction with completeness
of permutation~$\gamma^{-1} \in \Gamma$, as we show in the remainder of the
proof.

The previously derived explicit formula for~$\gamma^{-1}(i)$ implies
$\gamma(i) = i + \icheck - \ihat \pmod{n}$.
As~$\frac12 n < \ihat < \icheck < n$,
the first unfixed entry in the vector permuted by~$\gamma^{-1}$
is~$\gamma^{-1}(\ihat) = 2\ihat - \icheck$.
Applying permutation~$\gamma^{-1}$ to any vector,
permutes the entries~$1 + (\icheck - \ihat), \dots, \ihat + (\icheck - \ihat)$
to~$1, \dots, \ihat$.
Because for any vector~$x \in F(I_0, I_1)$
we have~$x \eqp{\ihat} \gamma(x)$, we can apply~$\gamma^{-1}$
on either side
to find~$\gamma^{-1}(x) \eqp{\gamma^{-1}(\ihat)} \gamma^{-1}(\gamma(x))$,
so~$x \eqp{\gamma^{-1}(\ihat)} \gamma^{-1}(x)$.
By the previous paragraph,~$\gamma^{-1}(\ihat) \in I_0$,
so any~$x \in F(I_0, I_1)$
has~$x \eqp{\gamma^{-1}(\ihat)} \gamma^{-1}(x)$
and~$x_{\gamma^{-1}(\ihat)} = 0$.
As the set of fixings is complete for~$x \succeq \gamma^{-1}(x)$,
all~$x \in F(I_0, I_1)$ have~%
$\gamma^{-1}(x)_{\gamma^{-1}(\ihat)} = x_{\ihat} = 0$,
so~$\ihat \in I_0$. This contradicts that~$\ihat \notin I_0 \cup I_1$.

A contradiction is found for~$\tilde x \preceq \gamma(\tilde x)$
for all~$\gamma \in \Gamma \setminus \{ \id \}$,
so~$\tilde x \succ \gamma(\tilde x)$.
\end{claimproof}

\begin{claim}
If~$\ihat \leq \frac12 n$,
then~$\tilde x \in F(I_0, I_1)$
with~$\tilde x_i = 1$ if~$i \in I_1 \cup \{ \ihat \}$
and~$\tilde x_i = 0$ if~$i \notin I_1 \cup \{ \ihat \}$
satisfies~$\tilde x \succ \gamma(\tilde x)$
for all~$\gamma \in \Gamma \setminus \{ \id \}$.
\end{claim}

\begin{claimproof}
The proof of Proposition~\ref{prop:monotonecyclicgroupfeasibility}
shows that~$\tilde x \succeq \gamma(\tilde x)$ for all~$\gamma \in \Gamma$.
To show~$\tilde x \succ \gamma(\tilde x)$ if~$\ihat < \frac12n$,
for the sake of contradiction
suppose that~$\tilde x = \gamma(\tilde x)$
for some~$\gamma \in \Gamma \setminus \{ \id \}$.
Then especially~$\tilde x = \delta(\tilde x)$
for all~$\delta \in \gen\gamma$.
Since~$\tilde x_{\ihat} = 1$,
for all~$\delta \in \gen\gamma$ we have~$\delta(\tilde x)_{\ihat} = 1$,
such that
$\orbit{\ihat}{\gamma} \subseteq I_1 \cup \{ \ihat \}$.

Let~$\delta \in \gen\gamma \setminus \{ \id \}$ and~%
$\icheck_\delta
\define \min\{ i \in [n] : \delta^{-1}(i) \notin I_0 \cup I_1 \}$.
If~$\ihat \leq \icheck_\delta$,
then~$x \eqp{\ihat} \delta(x)$
for all~$x \in F(I_0, I_1)$.
Moreover,~$\delta(x)_{\ihat} = 1$
follows from
$\orbit{\ihat}{\gamma} \subseteq I_1 \cup \{ \ihat \}$
and that~$\delta \neq \id$ is cyclic,
such that~$\delta^{-1}(\ihat) \neq \ihat$.
But then completeness of the fixings for permutation~$\delta$
implies that~$\ihat \in I_1$,
which violates the definition of~$\ihat$.
Hence,~$\ihat > \icheck_\delta$
for all~$\delta \in \gen\gamma \setminus \{ \id \}$.

As~$\icheck_\delta < \ihat$, we have~$\icheck_\delta \in I_0 \cup I_1$,
and~$x \eqp{\icheck_\delta} \delta(x)$ for all~$x \in F(I_0, I_1)$.
If~$\icheck_\delta \in I_0$,
then completeness of the fixings for permutation~$\delta$
yields~$\delta^{-1}(\icheck_\delta) \in I_0$,
violating the definition of~$\icheck_\delta$.
Thus,~$\icheck_\delta \in I_1$,
and as such from
$1 = {\tilde x}_{\icheck_\delta}
= \delta(\tilde x)_{\icheck_\delta}$
follows~$\delta^{-1}(\icheck_\delta) \in I_1 \cup \{ \ihat \}$.
Again, by the definition of
$\icheck_\delta$,~$\delta^{-1}({\icheck}_\delta) \notin I_0 \cup I_1$,
so~$\delta^{-1}(\icheck_\delta) = \ihat$.
Hence,~$\icheck_\delta = \delta(\ihat)$.
As~$\gamma \neq \id$ is a cyclic shift and~$\ihat \leq \frac12 n$,
there exists a~$\delta \in \gen\gamma \setminus \{ \id \}$
with~$\icheck_\delta = \delta(\ihat) > \frac12n \geq \ihat$.
This violates that~$\ihat > \icheck_\delta$ for all
$\delta \in \gen\gamma \setminus \{ \id \}$,
so a contradiction follows.
\end{claimproof}
\vspace{-1em}
\qedhere
\end{proof}

In the remainder of this section,
let~$\Gamma \leq \gen{\tilde\gamma}$,
where $\tilde\gamma$ is a composition of $m$ disjoint, ordered and
monotone cycles $\zeta_1, \dots, \zeta_m$,
with $N_i \define \supp(\zeta_i)
\define \{ k \in [n] : \zeta_i(k) \neq k \}$
for all $i \in [m]$,
and with~%
$\supp(\tilde\gamma) = [n] = N_1 \dot\cup \cdots \dot\cup N_m$.

\begin{algorithm}[!tbp]
\caption{\texttt{propagate}($\Gamma$, $I_0$, $I_1$, $k$, computeFixings)}
\label{alg:completemonotoneorderedcyclic}
\SetKwInOut{Input}{input}
\SetKwInOut{Output}{output}
\Input{%
group~$\Gamma \leq \gen{\tilde\gamma}$ for monotone and
ordered~$\tilde{\gamma} = \zeta_1 \circ \dots \circ \zeta_m \in \sym{n}$,
sets~$I_0, I_1 \subseteq [n]$, $k \in [m]$, Boolean variable
``computeFixings''
}
\Output{%
the message \Infeasible or \Feasible, and, if ``computeFixings'' is true,
also two subsets of~$[n]$
}

$\Delta_k \gets \Gamma$, $(J_0, J_1) \gets (I_0, I_1)$\;

\For{$c \gets k, \dots, m$\label{alg:propMainForLoop}}
{
  $(J_0, J_1) \gets$ complete set of fixings
  of~$\xymresack{\gamma} \cap F(J_0, J_1)$
  for every~$\gamma \in \restr(\Delta_c, N_c)$\;

  \If{$\xymretope{\restr(\Delta_c, N_c)} \cap F(J_0 \cap N_c, J_1 \cap
  N_c) = \emptyset$}
  {
    \Return \Infeasible\;\label{alg:propInf}
  }
  \If{computeFixings\label{alg:beginpeek}}
  {
    $(J_0', J_1') \gets (J_0, J_1)$\;
    \For{$i \in N_c \setminus
      (J_0' \cup J_1')$}
    {
      \If{\texttt{propagate}($\Delta_c$,
        $J_0 \cup \{i\}$, $J_1$, $c$, false)
        is \Infeasible}
      {
        $(J_0, J_1) \gets (J_0, J_1 \cup \{ i \})$\;
      }
      \ElseIf{\texttt{propagate}($\Delta_c$,
        $J_0$, $J_1 \cup \{i\}$, $c$, false) is \Infeasible}
      {
        $(J_0, J_1) \gets (J_0 \cup \{ i \}, J_1)$\label{alg:endpeek}\;
      }
    }
  }
  \If{$N_c \setminus (J_0 \cup J_1) = \emptyset$}
  {
    $\Delta_{c + 1} \gets \stab{J_0 \cap N_c}{\Delta_c}
    \cap \stab{J_1 \cap N_c}{\Delta_c}$\;
  }
  \Else
  {
    $\Delta_{c + 1} \gets \STAB{N_c}{\Delta_c}$\;
  }
}
\Return \Feasible, and~$(J_0, J_1)$ if computeFixings is true\;
\end{algorithm}

Algorithm~\ref{alg:completemonotoneorderedcyclic} shows a complete
propagation algorithm for the lexicographic leaders of the~$\Gamma$-orbit
with a given set of initial fixings.
It describes function~%
\texttt{propagate}($\Gamma$, $I_0$, $I_1$, $k$, $\text{computeFixings}$).
When choosing $k = 1$, this determines
whether~$\xymretope\Gamma \cap F(I_0, I_1)$ is empty (i.e., \Infeasible)
or not (i.e., \Feasible).
If the Boolean parameter $\text{computeFixings}$ is \True,
then it additionally computes the complete set of fixings of~%
$\xymretope\Gamma \cap F(I_0, I_1)$.

When running the algorithm with $\text{computeFixings}$ set to \False,
the worst-case running time is determined by determining the
complete set of fixings of $\xymretope\gamma \cap F(I_0^t, I_1^t)$
for each~%
$\gamma \in \hat\Delta_c \define \restr(\Delta_c, N_c) \leq \gen{\zeta_c}$.
As $\zeta_c$ is a monotone cycle,
Algorithm~\ref{alg:propagateallindividualcons}
of Section~\ref{sec:completepermset}
can be used.
Hence, the running time of this variant of the algorithm is~%
$
f(\Gamma) \define
\bigo\left(
\sum_{c \in [m]} |N_c| \ord(\hat\Delta_c)
\right)
\subseteq
\bigo\left(
n \max\{ |N_c| : c \in [m] \}
\right)
\subseteq
\bigo(n^2)
$,
as $\ord(\hat\Delta_c) \leq \ord(\gen{\zeta_c}) = |N_c|$.
In the case where the complete set of fixings is to be determined,
one additionally peeks for feasibility by calling the function recursively,
with $\text{computeFixings}$ set to \False.
That happens at most as many times as the number of unfixed entries in the
cycle, that is $|N_c|$ times for all $c \in [m]$.
The total running time is therefore
$
\bigo\left(
\sum_{c \in [m]} \left(
|N_c| \ord(\hat\Delta_c) + |N_c| f(\Gamma)
\right)
\right)
\subseteq
\bigo\left(
n^2 \max\{ |N_c| : c \in [m] \}
\right)
\subseteq \bigo(n^3)
$.

In this remainder of this section, we prove that
Algorithm~\ref{alg:completemonotoneorderedcyclic} correctly detects
infeasibility, or finds the complete set of fixings.
\begin{proposition}
  Let~$\tilde{\gamma} = \zeta_1 \circ \dots \circ \zeta_m \in \sym{n}$ be a
  monotone and ordered permutation, let~$\Gamma \leq \gen{\tilde{\gamma}}$,
  and let~$I_0, I_1 \subseteq [n]$.
  Then, \texttt{propagate}($\Gamma$, $I_0$, $I_1$, 1, true) returns
  \Feasible if and only if~$\xymretope\Gamma \cap F(I_0, I_1) \neq
  \emptyset$.
  Moreover, the returned sets~$J_0$ and~$J_1$ define a complete set of
  fixings.
\end{proposition}

\begin{proof}
  We claim that, at the beginning of iteration~$c$ of the for-loop in
  Line~\ref{alg:propMainForLoop}, $(J_0, J_1)$ defines a complete set of
  fixings on the first~$c-1$ cycles of~$\tilde{\gamma}$.
  Moreover, we claim that~$\Delta_c$ is the intersection of all setwise
  stabilizers~$\stab{J_0 \cap N_k}{\Gamma} \cap \stab{J_1 \cap N_k}{\Gamma}$ for
  cycles~$k \in [c-1]$ with~$N_k \subseteq J_0 \cup J_1$ and the pointwise
  stabilizers of all remaining cycles~$k \in [c-1]$.
  The claim holds true for~$c = 1$, so assume we have shown the claim for
  the first~$c \in [m]$ iterations.
  We show that at the end of iteration~$c$, $(J_0, J_1)$ is complete on the
  first~$c$ cycles of~$\tilde{\gamma}$ and~$\Delta_{c + 1}$
  has the described properties.
  The assertion follows then by induction.

  First, we show that we cannot derive further variable fixings from
  any~$\gamma \in \Gamma \setminus \Delta_c$.
  For~$x \in \R^n$, let~$x^c$ be the restriction of~$x$
  onto~$N_c$.
  Let~$\gamma \in \Gamma \setminus \Delta_c$
  and let~$N_{k'}$, $k' \in [c-1]$,
  be the first cycle violating a stabilizer condition of~$\Delta_c$.
  On the one hand, if~$N_{k'} \subseteq J_0 \cup J_1$, then all variables
  in~$N_{k'}$ are fixed.
  Since all the previous cycles are stabilized, this means that, for any~$x
  \in \xymretope\Gamma \cap F(J_0, J_1)$, $x$ and~$\gamma(x)$ coincide on
  the first~$k' - 1$ cycles
  and~$x^{k'} \neq \gamma(x^{k'})$.
  As~$x \succeq \gamma(x)$, the latter implies~$x^{k'} \succ
  \gamma(x^{k'})$ and thus~$x \succ \gamma(x)$.
  Hence, no variable fixings on~$N_c$ can be derived.
  On the other hand, if~$N_{k'} \setminus (J_0 \cup J_1) \neq \emptyset$,
  Proposition~\ref{prop:monotonecyclicgroupfeasibilitystrict} implies that
  there is~$
  x \in \xymretope{\restr(\Delta_{k'}, N_{k'})}
  \cap F(J_0 \cap N_{k'}, J_1 \cap N_{k'})
  $
  such that~$x^k \succ \gamma(x^k)$ for any~$\gamma \in \Delta_{k'}$ not
  stabilizing cycle~$N_{k'}$ completely.
  Consequently, we again cannot deduce variable fixings on cycle~$N_k$ from
  such~$\gamma$,
  and it is sufficient to check only the permutations in $\Delta_c$
  to derive further variable fixings.

  Second, we show that the algorithm finds all variable fixings on
  cycle~$N_c$.
  By the above discussion, $\Delta_c$ stabilizes every~$x \in
  \xymretope\Gamma \cap
  F(J_0, J_1)$ on the first~$c-1$ cycles.
  Hence, the algorithm correctly terminates in Line~\ref{alg:propInf}
  if~$\xymretope{\Delta_c} \cap F(I_0^c, I_1^c) = \emptyset$ as any~$x \in
  F(J_0, J_1)$ then satisfies~$x^c \prec \gamma(x^c)$
  for some~$\gamma \in \Delta_c$.
  Next, the algorithm uses~\texttt{propagate} to check,
  for each~$i \in N_c \setminus (I_0^c \cup I_1^c)$
  whether~$\xymretope{\Delta_c} \cap F(J_0 \cup \{i\}, J_1)$ is empty.
  In this case, $J_1$ can be extended by~$i$.
  The same reasoning is used to find variables that can be fixed to~$0$.
  Since this step is carried out for all not yet fixed variables on
  cycle~$N_k$, this step finds all variable fixings on~$N_c$ provided that
  \texttt{propagate} is able to correctly determine
  whether~$\xymretope{\Delta} \cap F(J_0, J_1)$ is empty or not.
  Then, after computing all fixings on~$N_c$, the algorithm checks whether
  all entries on~$N_c$ are already fixed and updates the
  stabilizer~$\Delta_{c + 1}$ according to the rules described in the
  first paragraph.

  To conclude the proof, we need to show that \texttt{propagate} indeed is
  able to determine whether~$\xymretope{\Delta_c} \cap F(J_0, J_1)$ is
  empty or not.
  So, assume~$\Delta_c$ stabilizes the first~$c-1$ cycles as described
  above and that we call~%
  \texttt{propagate}($\Delta$, $J_0$, $J_1$, $k$, false).
  Then, the algorithm performs the following steps: it checks whether the
  If-statement in  Line~\ref{alg:propInf} shows infeasibility for
  cycle~$c$, updates the stabilizer according to fixings on cycle~$N_c$,
  and proceeds in the same way on cycles~$c+1, \dots, m$.
  As explained above, the algorithm correctly terminates if it reaches
  Line~\ref{alg:propInf} in one of these iterations, because~$\Delta_c$
  stabilizes the previous cycles depending on whether all of their entries
  are fixed or not.
  If the algorithm does not reach Line~\ref{alg:propInf}, we thus need to
  show that~%
  $\xymretope{\Delta_c} \cap F(J_0, J_1) \neq \emptyset$.

To this end, we construct~$x \in \xymretope{\Gamma} \cap F(J_0, J_1)$ via
subvectors~$x^c \in \B{N_c}$ specifying each cycle~$c \in [m]$.
If~$N_c \subseteq J_0 \cup J_1$, all entries on~$N_c$ are fixed and we
define~$x^c$ accordingly;
otherwise, let~$x^c \in F(J_0 \cap N_c, J_1 \cap N_c)$ be a vector that
satisfies~$x^c \succ \gamma(x^c)$
for all nontrivial~$\gamma \in \restr(\Delta_c, N_c)$,
which exists by
Proposition~\ref{prop:monotonecyclicgroupfeasibilitystrict}.
We claim~$x \succeq \gamma(x)$ for all~$\gamma \in \Gamma$.
If the claim was wrong, there would exist~$c \in [m]$
and~$\gamma \in \Gamma$
such that~$x$ and~$\perm(x)$ coincide on the first~$c-1$
cycles and~$x^c \prec \perm(x^c)$.
Since~$\perm$ stabilizes the first~$c-1$ cycles, $\perm \in \Delta_c$.
Consequently, $N_c \subseteq J_0 \cup J_1$,
because otherwise~$x^c \succ \gamma(x^c)$ for all~$\gamma \in \Delta_c$
that do not stabilize cycle $c$ pointwise
(i.e.,~$\restr(\gamma, N_c) \neq \id$),
by construction.
That is the case for $\gamma$ that we consider,
as otherwise~$x^c \prec \gamma(x^c)$ is violated.
This, however, means that~$\xymretope{\restr(\Delta_c, N_c)} \cap F(J_0,
J_1) = \emptyset$ and Algorithm~\ref{alg:completemonotoneorderedcyclic} steps
into Line~\ref{alg:propInf}.
A contradiction to our assumption.
\end{proof}

\section{Computational Results}
\label{sec:computational}

This section's aim is to investigate the practical performance of the
algorithms derived in the preceding sections.
We are particularly interested in the following questions:
\begin{enumerate}[label=\textbf{Q\arabic*}, ref=Q\arabic*]
\item\label{Q1} Given a cyclic group~$\group$,
  by how much does the running time of
  Algorithm~\ref{alg:propagateallindividualcons} improve on the time
  for propagating the constraints~$x \succeq \delta(x)$, $\delta \in
  \group$, individually?

\item\label{Q2} Algorithm~\ref{alg:completemonotoneorderedcyclic} can be
called with the Boolean parameter ``computeFixings''.
  If it is \True, all possible fixings can be found for ordered and
  monotone permutations, but this requires additional computational effort.
  Is this additional running time compensated by strong fixings that can be
  detected?
\end{enumerate}
Note that these questions are only \emph{meaningful} for cyclic groups of order
greater than~2 as otherwise there is no difference between
Algorithm~\ref{alg:propagateallindividualcons} and individual propagation.

To answer~\ref{Q1} and~\ref{Q2}, we have implemented our methods in the mixed-integer
programming framework \scip~\cite{BestuzhevaEtal2021OO}.
\scip already provides methods to compute symmetries of a mixed-integer
program (MIP),
and it can enforce~$x \succeq \perm(x)$ for a
symmetry~$\perm$ of a MIP via propagation and separation methods for
so-called symresack and orbisack constraints, see~\cite{HojnyPfetsch2019}.
If the symmetry group of a MIP is a product group and one of its factors
defines an action associated with a certain symmetric group, \scip applies
orbitopal fixing, see Kaibel and Pfetsch~\cite{KaibelEtAl2011} as well
as Bendotti et al.~\cite{BendottiEtAl2021}, to handle the action of the
entire factor.
We have extended this code by a plugin that implements the enforcement
of~$x \succeq \delta(x)$ for all~$\delta \in \gen{\gamma}$.
This plugin uses Algorithm~\ref{alg:completemonotoneorderedcyclic} if the
representation of~$\gamma$ is monotone and ordered.
Otherwise we use Algorithm~\ref{alg:propagateallindividualcons},
where~$\Pi$ consists of all non-identity group elements of~$\gen\gamma$.

\paragraph{Computational Setup}
All experiments use a developer's version of \scip~8.0.0.1 (git hash
430ca7b) as mixed-integer programming framework and
\solver{SoPlex}~5.0.1.3 as LP solver.
\scip detects symmetries of a MIP by building an auxiliary
graph~\cite{salvagnin2005dominance,margot2010symmetry,pfetsch2019computational},
and computing its automorphism group~$\group$ using \solver{bliss}~0.73~\cite{JunttilaKaski2015bliss}.
This way, we find the permutations~$\perm_1,\dots,\perm_k$ that
generate~$\group$.
\scip checks whether~$\group$ is
a product group~${\group_1 \otimes \dots \otimes \group_\ell}$,
since the symmetries of each factor~$\group_i$ can be handled
independently, c.f.~\cite{HojnyPfetsch2019}.
Based on the generators, \scip heuristically decides whether a
factor~$\group_i$ can be completely handled by orbitopal
fixing, c.f.~\cite{HojnyPfetsch2019}.
Since this check evaluates positively only if no generator of~$\group_i$
has a cycle of length at least~3, we also make use of orbitopal fixing in
our experiments since no permutation that is \emph{meaningful} for~\ref{Q1}
and~\ref{Q2} is already handled by orbitopal fixing.

The generators of the remaining factors are handled using different
strategies that will allow us to answer~\ref{Q1} and~\ref{Q2}.
More precisely, we compare the following settings for the
generators~$\perm_i$ that are not handled via orbitopal fixing:
\begin{description}[itemsep=0pt]
\item[nosym]
  No symmetry handling is applied, in particular also no orbitopal fixing;
\item[gen]
  We only propagate generator constraints $x \succeq \gamma_i(x)$;
\item[group]
  We propagate the constraints $x \succeq \delta(x)$ for all~$\delta \in
  \gen{\gamma_i}\setminus \{ \id \}$ individually;
\item[nopeek]
  We propagate the constraints $x \succeq \delta(x)$ for all~$\delta \in
  \gen{\gamma_i}\setminus \{ \id \}$ using
  Algorithm~\ref{alg:completemonotoneorderedcyclic} if~$\perm_i$ is
  monotone and ordered with parameter \emph{completeFixings} set to \False,
  and otherwise using Algorithm~\ref{alg:propagateallindividualcons};
\item[peek]
  Same as \texttt{nopeek}, but with \emph{completeFixings} set to \True.
The constraints $x \succeq \delta(x)$ are handled
for all~$i \in [k]$
and~$\delta \in \gen{\gamma_i}\setminus \{ \id \}$
by
Algorithm~\ref{alg:completemonotoneorderedcyclic} with
\emph{completeFixings} set to \True if $\gamma_i$ is monotone and ordered
(guaranteeing completeness).
Otherwise, Algorithm~\ref{alg:propagateallindividualcons} is used and we
check whether further variable fixings can be found using
Observation~\ref{obs:fixing}.
This check is carried out for all unfixed variables whose fixing values
were looked up in the previous calls to
Algorithm~\ref{alg:propagateallindividualcons}.
\end{description}
Only in the \texttt{peek} setting with an ordered and monotone generator,
our methods are guaranteed to find the complete set of fixings for the group generated
by this generator.
Since none of the generators~$\gamma_1,\dots,\gamma_k$ is guaranteed to be
in monotone and ordered representation, we also implemented a heuristic
that relabels the variable array to guarantee that at least one generator
per factor is monotone and ordered.
It iterates over the generators, sorted descending with respect to the
largest subcycle size and then descending on the group order.
If the variable indices of the generator are not yet relabeled, then
it relabels these such that the generator becomes monotone and ordered.
For this, we provide three options:
\begin{enumerate*}[label=\textbf{\emph{\small(\roman*)}},
	ref=\emph{\small(\roman*)}]
\item sorting the subcycles in decreasing length ({\tt max}),
\item sorting the subcycles in increasing length ({\tt min}), and
\item sorting the subcycles based on the variable index that is minimal in
the original variable ordering ({\tt respect}).
\end{enumerate*}
For example, consider three generators $\gamma_1 = (1, 8, 7, 3)$,
$\gamma_2 = (3, 4, 5, 8)$, and $\gamma_3 = (2, 5, 6, 9, 4)$.
The heuristic relabels~$\gamma_1$,
does not relabel~$\gamma_2$ since it also contains variable index 3,
and relabels~$\gamma_3$.
Using hats to distinguish the relabeled space, the heuristic yields~%
$
[1, 2, 3, 4, 5, 6, 7, 8, 9] \mapsto
[\hat 1, \hat 5, \hat 4, \hat 9, \hat 6, \hat 7, \hat 3, \hat 2, \hat 8]
$,
such that~$\gamma_1 = (\hat 1, \hat 2, \hat 3, \hat 4)$,
$\gamma_2 = (\hat 4, \hat 9, \hat 6, \hat 2)$,
and~$\gamma_3 = (\hat 5, \hat 6, \hat 7, \hat 8, \hat 9)$.
Note that~$\gamma_1$ and~$\gamma_3$ are monotone (and ordered) in this
relabeled space.
In particular, with this heuristic we make sure that at least the first
processed generator has a monotone and ordered representation.

Note that a cyclic group~$\gen{\perm}$ acting on~$n$ elements might have
order~$\Omega(2^n)$.
To prevent memory overflow in
Algorithm~\ref{alg:propagateallindividualcons} and in the individual
propagation ({\tt group}), we added a safeguard that might only propagate
a restricted set of permutations in these algorithms.
Let~$s$ be the cardinality of the support of the generator~$\gamma$
and~$\Pi \define \gen\gamma \setminus \{ \id \}$.
If $|\Pi| > \SI{e4}{}$ or $s |\Pi| > 5 \cdot 10^6$, we only enforce
the symmetry constraints for~$\Pi' \define \{ \gamma^i : i \in [k] \}$,
where~$k \in \mathds{Z}_+$ is maximal such that~$|\Pi'| \leq 10^4$
and~$s |\Pi'| \leq 5 \cdot 10^6$.

In the following subsections, we report on our experiments for two
different classes of problems, including benchmarking instances.
To reduce the impact of performance variability, all instances have been
run with~5 different random seeds.
The running time~$t_i$ per instance~$i$ is reported in shifted
geometric mean~$\prod_{i = 1}^n (t_i + s)^{\frac{1}{n}} - s$ with a shift
of~$s = \SI{10}{\second}$, as well as the number of instances solved (S)
within the time limit of~\SI{2}{\hour} per instance.
If the time limit is hit, then we report a running time of~\SI{2}{\hour}
for that run.
In the article's main part, we present aggregated results, showing the shifted
geometric means of all instances with the given settings, as well as the
number of instances solved.
Detailed tables can be found in Appendix~\ref{app:supplements}.
These tables contain information per instance, as well as the total time
spent on the instances, and the total time and percentage of time spent on
symmetry handling.
All computations were run on a Linux cluster with Intel Xeon E5
\SI{3.5}{\GHz} quad core processors and \SI{32}{\giga\byte} memory.  The
code was executed using a single thread.

\paragraph{Edge-coloring flower snarks}
Flower snark graphs~\cite{fiorini1977edge,isaacs1975infinite}, described
in Figure~\ref{fig:flowersnark}, are a family of undirected graphs with
chromatic index~4, i.e., there is no coloring of the edges with three
colors such that incident edges are colored differently.
\begin{figure}[b]
	\begin{minipage}{.3\textwidth}
		\centering
		\begin{tikzpicture}[x=.1\textwidth, y=.1\textwidth]
		\tikzstyle{every node}=[font=\scriptsize]
		\def\n{5}
		\foreach \i in {0,...,\n}%
		{
			\ifnum \i=\n
			\else
			\pgfmathsetmacro{\ii}{int(\i+1)}
			\node[] (A\i) at (
			{-(0 * cos(\i*360/\n) - 2 * sin(\i*360/\n))},
			{0 * sin(\i*360/\n) + 2 * cos(\i*360/\n)}
			) {$a_{\ii}$};
			\node[] (B\i) at (
			{-(0 * cos(\i*360/\n) - 1 * sin(\i*360/\n))},
			{0 * sin(\i*360/\n) + 1 * cos(\i*360/\n)}
			) {$b_{\ii}$};

			\pgfmathparse{int(mod(\i,2))}
			\ifnum \pgfmathresult=0
			\node[] (C\i) at (
			{-(sqrt(2)/2 * cos(\i*360/\n) - 3 * sin(\i*360/\n))},
			{sqrt(2)/2 * sin(\i*360/\n) + 3 * cos(\i*360/\n)}
			) {$c_{\ii}$};
			\node[] (D\i) at (
			{-(-sqrt(2)/2 * cos(\i*360/\n) - 3 * sin(\i*360/\n))},
			{-sqrt(2)/2 * sin(\i*360/\n) + 3 * cos(\i*360/\n)}
			) {$d_{\ii}$};
			\else
			\node[] (C\i) at (
			{-(-sqrt(2)/2 * cos(\i*360/\n) - 3 * sin(\i*360/\n))},
			{-sqrt(2)/2 * sin(\i*360/\n) + 3 * cos(\i*360/\n)}
			) {$c_{\ii}$};
			\node[] (D\i) at (
			{-(sqrt(2)/2 * cos(\i*360/\n) - 3 * sin(\i*360/\n))},
			{sqrt(2)/2 * sin(\i*360/\n) + 3 * cos(\i*360/\n)}
			) {$d_{\ii}$};
			\fi
			\fi
		}

		\foreach \i in {0,...,\n}%
		{
			\ifnum \i=\n
			\else
			\draw (A\i) to (B\i);
			\draw (A\i) to (C\i);
			\draw (A\i) to (D\i);

			\pgfmathparse{int(mod(\i+1,\n))}
			\draw (B\i) to (B\pgfmathresult);

			\ifnum \pgfmathresult=0
			\draw (C\i) to [bend left=80] (D\pgfmathresult);
			\draw (D\i) to [bend left=80] (C\pgfmathresult);
			\else
			\draw (C\i) to [bend left=80] (C\pgfmathresult);
			\draw (D\i) to [bend left=80] (D\pgfmathresult);
			\fi
			\fi
		}
		\end{tikzpicture}
	\end{minipage}%
	\begin{minipage}{.7\textwidth}
		\textbf{Construction of flower snark graph}\quad
		Graph $J_n$ is defined for odd $n \geq 3$ and has $4n$ vertices
		labeled by~$a_i, b_i, c_i, d_i$ for~$i \in [n]$.
		For $i \in [n]$, connect $a_i$ to $b_i, c_i, d_i$,
		make cycle $(b_1, b_2, \dots, b_n)$,
		and for $i < n$ connect $c_i$ to $d_{i+1}$,
		and $d_i$ to $c_{i+1}$. Last, connect $c_n$ to $c_1$ and $d_n$ to
		$d_1$.
		The automorphism group generators are~%
		$(a_1, \dots, a_n)(b_1, \dots, b_n)(c_1, \dots, c_n,d_1, \dots,
		d_n)$,
		and
		$
		(c_1, d_1)
		\cdot
		\prod_{i=2}^{\frac{n+1}{2}}
		(a_i, a_{n + 2 -i})
		(b_i, b_{n + 2 -i})
		(c_i, c_{n + 2 -i})
		(d_i, d_{n + 2 -i})%
		.
		$
	\end{minipage}
	\caption{Flower snark graph $J_5$ and construction of $J_n$.}
	\label{fig:flowersnark}
\end{figure}%
Deciding whether an edge coloring with~3 colors exists for an undirected
graph~$\mathcal G = (V, E)$ is equivalent to decide whether
\[
S_{\mathcal G} =
\left\{
x \in \binary^{E \times [3]} :
\begin{aligned}
x_{e, k} + x_{e', k} \leq 1\
&
\text{for all}\ k \in [3]\
\text{and}\ e, e' \in E\
\text{with}\ |e \cap e'| = 1, \\
\sum_{k \in [3]} x_{e, k} = 1\
&
\text{for all}\ e \in E.
\end{aligned}
\right\}%
\]%
is empty, which can be done by binary programming techniques.
Margot~\cite{margot2007symmetric} studied this binary program for the
flower snarks~$J_{13}$, $J_{15}$ and $J_{21}$, and we also use these snarks
in our experiments as they admit cyclic symmetries.
The flower snarks are defined for odd~$n \geq 3$ and have an automorphism
group of order~$4n$.
The group is generated by a cycle of order~$2n$ and a reflection.
Symmetries in the problem are therefore given by these graph
automorphisms, and by interchanging the edge colors.
These symmetries are automatically identified by \scip.

\begin{table}[t]
\caption{3-edge coloring on flower snark graphs $J_n$
for~$n \in \{3, \dots, 43\}$.}
\label{tab:flowersnark}
\scriptsize
\centering
\begin{tabular}{l*{10}{rr}}
\toprule
& \multicolumn{2}{c}{\tt nosym}& \multicolumn{2}{c}{\tt gen}&
\multicolumn{2}{c}{\tt group}& \multicolumn{2}{c}{\tt nopeek}&
\multicolumn{2}{c}{\tt peek}\\
\cmidrule(l{1pt}r{1pt}){2-3}
\cmidrule(l{1pt}r{1pt}){4-5}
\cmidrule(l{1pt}r{1pt}){6-7}
\cmidrule(l{1pt}r{1pt}){8-9}
\cmidrule(l{1pt}r{1pt}){10-11}
relabeling& time(s) & S & time(s) & S & time(s) & S & time(s) & S &
time(s) & S \\
\midrule
original            &   730.78 & 54 &   172.35 & 88 &   187.56 & 87 &
169.79 & 93 &   153.23 & 97 \\
max                 & --       &       &   407.02 & 65 &   312.93 & 70 &
278.00 & 77 &   270.19 & 78 \\
min                 & --       &       &    97.99 & 102 &   131.58 & 95
&   127.48 & 92 &   119.67 & 95 \\
respect             & --       &       &   184.44 & 88 &   173.41 & 86 &
174.32 & 88 &   178.46 & 87 \\
\cmidrule{2-11}
aggregated          &   730.78 & 54 &   189.90 & 343 &   191.75 & 338 &
180.33 & 350 &   172.84 & 357 \\
\bottomrule
\end{tabular}
\end{table}
Our testset consists of all these instances with all odd parameters $n \in
\{3, \dots, 49\}$. In Table~\ref{tab:flowersnark} we present the
aggregated results for these instances, and we refer to
Section~\ref{app:supplements} for the results per instance.
Comparing the shifted geometric means of our proposed methods (\texttt{nopeek} and
\texttt{peek}) with naively propagating individual constraints~$x \succeq \delta(x)$ (\texttt{group}),
without relabeling, our proposed methods gain \SI{9.5}{\percent} and
\SI{18.3}{\percent}, respectively.
In particular, in both the \texttt{peek} and \texttt{nopeek} variant at least one run of all
instances solved the problem, while \texttt{group} could not solve a single
instance for~$n \geq 45$.
These results are even more pronounced if we consider the subset of
instances that could not be solved without symmetry handling within the
time limit, which is for parameter~$n \geq 27$.
\begin{table}[t]
\caption{3-edge coloring on flower snark graphs $J_n$
for~$n \in \{27, \dots, 43\}$.}
\label{tab:flowersnark:restrict}
\scriptsize
\centering
\begin{tabular}{l*{10}{rr}}
\toprule
& \multicolumn{2}{c}{\tt nosym}& \multicolumn{2}{c}{\tt gen}&
\multicolumn{2}{c}{\tt group}& \multicolumn{2}{c}{\tt nopeek}&
\multicolumn{2}{c}{\tt peek}\\
\cmidrule(l{1pt}r{1pt}){2-3}
\cmidrule(l{1pt}r{1pt}){4-5}
\cmidrule(l{1pt}r{1pt}){6-7}
\cmidrule(l{1pt}r{1pt}){8-9}
\cmidrule(l{1pt}r{1pt}){10-11}
relabeling& time(s) & S & time(s) & S & time(s) & S & time(s) & S &
time(s) & S \\
\midrule
original            &  7200.00 &  0 &  1266.33 & 33 &  1526.99 & 32 &
1269.40 & 38 &  1056.42 & 42 \\
max                 & --       &       &  4752.76 & 10 &  3501.50 & 15 &
2884.53 & 22 &  2734.31 & 23 \\
min                 & --       &       &   531.09 & 47 &   912.18 & 40 &
889.73 & 37 &   797.14 & 40 \\
respect             & --       &       &  1480.93 & 33 &  1444.56 & 31 &
1452.88 & 33 &  1509.23 & 32 \\
\cmidrule{2-11}
aggregated          &  7200.00 &  0 &  1478.12 & 123 &  1630.32 & 118 &
1475.85 & 130 &  1366.37 & 137 \\
\bottomrule
\end{tabular}
\end{table}
These results are shown in Table~\ref{tab:flowersnark:restrict}:
we gain \SI{30.8}{\percent} and \SI{16.9}{\percent}
when comparing \texttt{nopeek} and \texttt{peek} with  in the original labeling,
respectively.

Since the instances only have one generator with order larger than two,
our algorithm guarantees completeness with respect to this cyclic subgroup
in the automatically relabeled \texttt{peek}-variants.
As the tables show, relabeling has a big impact on the running times and
on the number of instances that we can solve.
We observe a significant improvement when comparing the \texttt{nopeek} and \texttt{peek}
variants with \texttt{group} in the {\tt max}-relabeling and
{\tt min}-relabeling. In the
{\tt respect}-relabeling a worse result is found.
Surprisingly, the weaker symmetry handling method \texttt{gen}
is on average the fastest in the {\tt min}-relabeling.

The \texttt{peek}-variant requires significant additional computational time.
This is also reflected in the total time spent on handling symmetries,
which is \SI{16.7}{\percent}, \SI{18.0}{\percent}, \SI{6.9}{\percent} and
\SI{5.9}{\percent} for the original, max, min and respect relabelings,
respectively. As a comparison, this is at most \SI{3.7}{\percent} for the
\texttt{nopeek}-variant.
However, this investment pays off in terms of the average time to solve
the instances, since the \texttt{peek}-variant solves the instances on average
faster than the \texttt{nopeek}-variant, when comparing the aggregated results.
We hypothesize that the additional fixings found by peeking outweighs the
additional time spent on handling symmetries, and that this effect will be
even more pronounced for more difficult instances with a larger cyclic
symmetry subgroup order.
This is consistent with the numbers in the tables of the supplements:
For instances with a higher value of the parameter $n$,
the running times of {\texttt{nopeek}} and {\texttt{peek}} relative to
{\texttt{nosym}}, {\texttt{gen}} and {\texttt{symregroup}} decrease.
For the flower snark instances, we thus can answer \ref{Q1} and \ref{Q2}
affirmatively as Algorithm~\ref{alg:completemonotoneorderedcyclic} and (in
the not relabeled case) Algorithm~\ref{alg:propagateallindividualcons}
clearly outperform the \texttt{group} setting.

\paragraph{MIPLIB testset}
MIPLIB 2010~\cite{KochEtAl2011MIPLIB}
and MIPLIB 2017~\cite{Gleixner2021MIPLIB} are data bases of benchmark MIP
instances consisting of 1426 instances.
We extracted all instances for which \scip could identify at least one
generator that defines a cyclic group of order at least~3, and that \scip
can presolve within three hours.
This results in~38 instances relevant for our experiments.
We restricted this test set further by removing all instances that cannot
be solved by any of our settings within three hours or that can be solved
within presolving.
This results in~10 instances, which we use for the experiments below.
Note that this is a very small testset and that the characteristics
between the instances can be very different.
For this reason, we carefully analyze aggregated results and refer to the
appendix for detailed per-instance results when necessary.

\begin{table}[t]
\caption{MIPLIB 2010 and MIPLIB 2017 benchmark instances with group
generators larger than two.}
\label{tab:miplib}
\scriptsize
\centering
\begin{tabular}{l*{10}{rr}}
\toprule
& \multicolumn{2}{c}{\tt nosym}& \multicolumn{2}{c}{\tt gen}&
\multicolumn{2}{c}{\tt group}& \multicolumn{2}{c}{\tt nopeek}&
\multicolumn{2}{c}{\tt peek}\\
\cmidrule(l{1pt}r{1pt}){2-3}
\cmidrule(l{1pt}r{1pt}){4-5}
\cmidrule(l{1pt}r{1pt}){6-7}
\cmidrule(l{1pt}r{1pt}){8-9}
\cmidrule(l{1pt}r{1pt}){10-11}
relabeling& time(s) & S & time(s) & S & time(s) & S & time(s) & S &
time(s) & S \\
\midrule
original            &  1853.78 & 42 &   491.69 & 50 &   407.76 & 50 &
449.34 & 48 &   506.23 & 48 \\
max                 & --       &       &  1061.29 & 47 &   812.31 & 48 &
771.56 & 48 &   796.69 & 49 \\
min                 & --       &       &   901.65 & 48 &   612.10 & 50 &
837.48 & 47 &   657.29 & 49 \\
respect             & --       &       &   970.19 & 47 &   743.01 & 49 &
639.20 & 50 &   723.93 & 49 \\
\cmidrule{2-11}
aggregated          &  1853.78 & 42 &   822.47 & 192 &   623.37 & 197 &
656.66 & 193 &   662.01 & 195 \\
\bottomrule
\end{tabular}
\end{table}
The results of the performance tests on the~10 instances are shown in
Table~\ref{tab:miplib}.
It is clear that the results are very sensitive to the chosen relabeling
variant.
On average, the instances could be solved significantly faster when the
variable ordering of the original problem is used, rather than any
relabeled variant.
We hypothesize that this result can be attributed to internal behavior of
\scip, where operations (such as branching) on variables with a low
internal index are preferred.
Moreover, in contrast to the flower snark instances, many benchmark
instances find multiple generators of order larger than two, the order of
these generators may be low, and we may not always be able to relabel
all these group generators such that they are all monotone and ordered.
Together, for such diverse instances, this may explain why relabeling has
a negative impact on the running times.
This means that even in the heuristically relabeled and \texttt{peek}-cases, since
after relabeling often not all generators are monotone and ordered, we
often need to fallback to the simpler variant
(Algorithm~\ref{alg:propagateallindividualcons}), which does not guarantee
completeness for the generating subgroup and comes at a higher asymptotic
computational cost.
This might explain why our more sophisticated algorithms perform on average
by roughly~\SI{5}{\percent} and~\SI{10}{\percent} worse than the
\texttt{group} setting.
Due to the small size of the test set, however, one needs to be careful
interpreting the aggregated numbers as they can be highly biased because of
outliers.

\begin{table}[t]
\caption{Relevant MIPLIB 2010 and MIPLIB 2017 benchmark instances,
without automatic relabeling.}
\label{tab:miplib:norelabel}
\scriptsize
\centering
\begin{tabular}{l*{10}{rr}}
\toprule
& \multicolumn{2}{c}{\tt nosym}& \multicolumn{2}{c}{\tt gen}&
\multicolumn{2}{c}{\tt group}& \multicolumn{2}{c}{\tt nopeek}&
\multicolumn{2}{c}{\tt peek}\\
\cmidrule(l{1pt}r{1pt}){2-3}
\cmidrule(l{1pt}r{1pt}){4-5}
\cmidrule(l{1pt}r{1pt}){6-7}
\cmidrule(l{1pt}r{1pt}){8-9}
\cmidrule(l{1pt}r{1pt}){10-11}
instance & time(s) & S & time(s) & S & time(s) & S & time(s) & S & time(s)
& S \\
\midrule
cod105               &  7200.04 &  0 &    65.94 &  5 &    61.84 &  5 &
59.14 &  5 &    57.45 &  5 \\
cov1075              &  4761.35 &  5 &   117.17 &  5 &    47.48 &  5 &
32.71 &  5 &    33.15 &  5 \\
fastxgemm-n2r6s0t2   &  1628.36 &  5 &   286.01 &  5 &   209.40 &  5 &
141.51 &  5 &   145.82 &  5 \\
fastxgemm-n2r7s4t1   &  6394.82 &  2 &   971.15 &  5 &   812.15 &  5 &
768.81 &  5 &   776.56 &  5 \\
neos-1324574         &  6251.90 &  5 &  2398.62 &  5 &  2080.58 &  5 &
2064.22 &  5 &  2060.96 &  5 \\
neos-3004026-krka    &   129.58 &  5 &   262.07 &  5 &   144.61 &  5 &
685.34 &  4 &  1286.86 &  4 \\
neos-953928          &  2332.51 &  5 &  1471.52 &  5 &  1166.06 &  5 &
975.68 &  5 &   975.07 &  5 \\
neos-960392          &   850.24 &  5 &  1954.62 &  5 &  2278.05 &  5 &
2451.05 &  4 &  2454.52 &  4 \\
supportcase29        &   282.05 &  5 &   196.09 &  5 &   253.30 &  5 &
388.96 &  5 &   662.54 &  5 \\
wachplan             &  2713.70 &  5 &   906.09 &  5 &   939.21 &  5 &
849.65 &  5 &   849.25 &  5 \\
\cmidrule{2-11}
All instances combined                      &  1853.78 & 42 &   491.69 &
50 &   407.76 & 50 &   449.34 & 48 &   506.23 & 48 \\
Total time                                  & \multicolumn{2}{r}{
45:46:14     } & \multicolumn{2}{r}{ 13:24:05     } & \multicolumn{2}{r}{
12:12:07     } & \multicolumn{2}{r}{ 15:57:33     } & \multicolumn{2}{r}{
18:43:22     } \\
Symmetry time                               & \multicolumn{2}{r}{
0:00:00     } & \multicolumn{2}{r}{  0:03:44     } & \multicolumn{2}{r}{
0:03:47     } & \multicolumn{2}{r}{  0:06:56     } & \multicolumn{2}{r}{
1:56:05     } \\
Percentage time                             & \multicolumn{2}{r}{
0.0\%} & \multicolumn{2}{r}{         0.5\%} & \multicolumn{2}{r}{
0.5\%} & \multicolumn{2}{r}{         0.7\%} & \multicolumn{2}{r}{
10.3\%} \\
\bottomrule
\end{tabular}
\end{table}

Looking into the results on a per-instance basis confirms this conjecture.
Table~\ref{tab:miplib:norelabel} provides these results without relabeling.
We can observe two classes of instances.
The instances for which no symmetry handling (\texttt{nosym}) or only
handling generators (\texttt{gen}) perform best---\emph{neos-3004026-krka}, \emph{neos-960392},
and \emph{supportcase29}---and the remaining instances which benefit from
more aggressive symmetry handling.
On the three mentioned instances, our methods are significantly slower than
\texttt{group}.
However, since (almost) no symmetry handling performs best for these
instances, we conjecture that symmetry handling might sometimes not be a
favorable strategy.
For example, neos-3004026-krka and neos-960392 have already an optimal dual bound after
processing the root node, i.e., once a matching primal solution is found, \scip
immediately terminates.
Based on the results it thus seems that the symmetry-based variable fixings
hinder \scip's heuristics to find a good solution.
In this case, of course, less symmetry handling is beneficial.

On the remaining seven instances that benefit from symmetry handling, however,
our methods \texttt{nopeek} and \texttt{peek} clearly dominate the
\texttt{group} setting.
As Table~\ref{tab:miplibwo} shows, we can improve the running time of
\texttt{group} by \SI{5.9}{\percent}
and \SI{4.6}{\percent}, respectively, if considering the mean running for
all relabeling strategies.
Without relabeling, the performance improves even by \SI{14.2}{\percent}
and \SI{13.9}{\percent}, respectively.
Thus, these numbers also indicate a positive answer to Question~\ref{Q1}:
If symmetry handling is important to solve an instance efficiently, then
exploiting the structure of cyclic groups via Algorithms~\ref{alg:propagateallindividualcons}
and~\ref{alg:completemonotoneorderedcyclic} outperforms the approach that
neglects the group structure.
For Question~\ref{Q2}, the answer is less clear, because the performance of
all symmetry handling methods is highly sensitive to changing the variable
labeling.

\begin{table}[t]
\caption{Relevant MIPLIB 2010 and MIPLIB 2017 benchmark instances, without
\emph{neos-3004026-krka}, \emph{neos-920392}, and \emph{supportcase29}.}
\label{tab:miplibwo}
\scriptsize
\centering
\begin{tabular}{l*{10}{rr}}
\toprule
& \multicolumn{2}{c}{\tt nosym}& \multicolumn{2}{c}{\tt gen}&
\multicolumn{2}{c}{\tt group}& \multicolumn{2}{c}{\tt nopeek}&
\multicolumn{2}{c}{\tt peek}\\
\cmidrule(l{1pt}r{1pt}){2-3}
\cmidrule(l{1pt}r{1pt}){4-5}
\cmidrule(l{1pt}r{1pt}){6-7}
\cmidrule(l{1pt}r{1pt}){8-9}
\cmidrule(l{1pt}r{1pt}){10-11}
relabeling& time(s) & S & time(s) & S & time(s) & S & time(s) & S &
time(s) & S \\
\midrule
original            &  3917.80 & 27 &   501.57 & 35 &   393.39 & 35 &
337.55 & 35 &   338.57 & 35 \\
max                 & --       &       &  1193.75 & 33 &   668.57 & 35 &
694.06 & 35 &   698.20 & 35 \\
min                 & --       &       &  1062.99 & 33 &   719.87 & 35 &
710.16 & 35 &   706.24 & 35 \\
respect             & --       &       &  1226.05 & 33 &   761.70 & 35 &
678.57 & 35 &   715.09 & 35 \\
\cmidrule{2-11}
aggregated          &  3917.80 & 27 &   940.64 & 134 &   616.62 & 140 &
580.20 & 140 &   588.38 & 140 \\
\bottomrule
\end{tabular}
\end{table}

\paragraph{Conclusion}
The algorithms developed in this article allow to efficiently handle
symmetries of cyclic groups.
In particular, if the generator of the cyclic group is ordered and
monotone, Algorithm~\ref{alg:completemonotoneorderedcyclic} is guaranteed
to find all symmetry based variable fixings that can be derived from a set
of given fixings.
As illustrated by our numerical experiments for the flower snark instances,
this algorithm clearly dominates the propagation of individual constraints
(\texttt{group}) and the weaker (but faster) variant of
Algorithm~\ref{alg:completemonotoneorderedcyclic} that does not come with a
guarantee of completeness.
On general benchmark instances, the situation is less clear.
On the one hand, it seems that some instances do not benefit from symmetry
handling, e.g., because of interferences with heuristics.
On the other hand, while flower snark instances only had one cyclic
generator, the benchmark instances might admit several cyclic generators
that cannot be made ordered and monotone simultaneously.
Here, the best results for instances that benefit from symmetry handling
could be achieved without relabeling.
In this case, the completeness of
Algorithm~\ref{alg:completemonotoneorderedcyclic} is not guaranteed, but
the weaker \texttt{nopeek} version still dominates the individual treatment
of permutations.
We thus conclude that the dedicated methods that exploit the structure of
cyclic groups developed in this article clearly help to improve the
performance of symmetry handling in \scip.

Nevertheless, the numerical results also show interesting possibilities for
future research.
On the one hand, it seems that it is important to identify which instances
benefit from (aggressive) symmetry handling.
On the other hand, since completeness of
Algorithm~\ref{alg:completemonotoneorderedcyclic} is only guaranteed for
monotone and ordered permutations~$\perm$, it would be helpful to derive methods
that achieve completeness even if one of the assumptions on~$\perm$ are
dropped.
This reduces the impact of the labeling, which has be shown to highly
influence the performance of \scip.
Finally, if no alternative algorithms can be developed, also more
sophisticated methods for finding a good relabeling could be beneficial.
This, however, is out of scope of this article.

\paragraph{Acknowledgments} We are thankful to Rudi Pendavingh, Marc E.\
Pfetsch, and Frits Spieksma for carefully reading a preliminary version of
the article and their suggestions to improve the presentation.

{
\scriptsize
\bibliographystyle{plain}

}

\appendix
\section{Proofs of Section~\ref{sec:completepermset}}
\label{sec:proofs}
In Section~\ref{sec:completepermset}, we presented various propositions,
without giving their proofs.
In this appendix, the missing proofs are provided.

\terminatecondition*
\begin{proof}
Let~$\gamma \in \Pi$,~$t$ be some time index,~%
$\emptyset \notin \infconj{\gamma}{t}$,
and that~$\{ (i, b) \} \in \infconj{\gamma}{t}$ implies~$i \in I_{1-b}^t$.
For brevity, we denote~$F^t \define F(I_0^t, I_1^t)$,
and~$\xymretope{\gamma}^t \define \xymretope{\gamma}^{(i_\gamma^t)}$.
We prove correctness for each of
\ref{prop:terminationcondition:1},
\ref{prop:terminationcondition:2} and
\ref{prop:terminationcondition:3}
by contradiction,
by assuming that the set of fixings~$(I_0^t, I_1^t)$
is not complete for the constraint~$x \succeq \gamma(x)$
with~$\gamma \in \Pi$.
We denote a missing valid fixing by~$f = (i, b)$,
and show that a contradiction follows.
This means that~$\xymretope\gamma \cap F^t \cap V(\bar f) = \emptyset$.
Also,~$i \notin I_0^t \cup I_1^t$
implies~$\{ \bar f \} \notin \infconj{\gamma}{t}$,
which means~$\xymretope{\gamma}^t \cap F^t \cap V(\bar f) \neq \emptyset$.

First suppose~\ref{prop:terminationcondition:1},
that~$\eqconj{\gamma}{t} = \emptyset$.
Then no vector~$x \in F^t$
exists with~$x \eqp{i_\gamma^t} \gamma(x)$.
In this case, $x \prec \gamma(x)$ for all~$x \in F^t$ with~$x_i = 1-b$,
since~$\xymretope\gamma \cap F^t \cap V(\bar f) = \emptyset$.
Hence, $x \precp{i_\gamma^t} \gamma(x)$.
A contradiction follows,
as this observation and~$\emptyset \notin \infconj{\gamma}{t}$
shows that~$\{ \bar f \}$ is a minimal~$i_\gamma^t$-inf-conjunction,
contradicting~$\{ \bar f \} \notin \infconj{\gamma}{t}$.

Second suppose~\ref{prop:terminationcondition:2}.
If~$i_\gamma^t > n$,
then~$\xymretope\gamma \cap F^t = \xymretope\gamma^t \cap F^t$
violates~%
$\xymretope\gamma \cap F^t \cap V(\bar f)
= \emptyset \neq \xymretope\gamma^t \cap F^t \cap V(\bar f)$.

Last suppose~\ref{prop:terminationcondition:3}.
For any vector~$x \in \binary^n$,
with~$x_{i_\gamma^t} = 1$
and~$\gamma(x)_{i_\gamma^t} = x_{\gamma^{-1}(i_\gamma^t)} = 0$,
we have that~$x \succeqp{i_\gamma^t} \gamma(x)$
implies~$x \succeq \gamma(x)$,
so~%
$
\xymretope\gamma \supseteq
\xymretope\gamma^t \cap
V(\{ (i_\gamma^t, 1), (\gamma^{-1}(i_\gamma^t), 0) \})
$.
Hence,~$\xymretope\gamma \cap F^t \cap V(\bar f) = \emptyset$
yields~%
$
\xymretope\gamma^t \cap F^t \cap
V(\{ \bar f, (i_\gamma^t, 1), (\gamma^{-1}(i_\gamma^t), 0) \})
= \emptyset
$,
meaning that~%
$C = \{ \bar f, (i_\gamma^t, 1), (\gamma^{-1}(i_\gamma^t), 0) \}$
is an~$i_\gamma^t$-inf-conjunction.
As~$\gamma(i_\gamma^t) > i_\gamma^t$,
entry~$i_\gamma^t$ does not occur
in the first~$i_\gamma^t - 1$ entries of a vector~$x$ and~$\gamma(x)$,
so~$(i_\gamma^t, 1)$ cannot be element of a
minimal~$i_\gamma^t$-inf-conjunction.
Likewise, because~$\gamma^{-1}(i_\gamma^t) > i_\gamma^t$
and~$\gamma(\gamma^{-1}(i_\gamma^t)) = i_\gamma^t$,
entry~$\gamma^{-1}(i_\gamma^t)$ does also not occur in the
first~$i_\gamma^t - 1$ entries of any vector~$x$ and~$\gamma(x)$,
so~$(\gamma^{-1}(i_\gamma^t), 0)$ cannot be element of any
minimal~$i_\gamma^t$-inf-conjunction, too.
By removing these, we have that~$\{ \bar f \}$
is an~$i_\gamma^t$-inf-conjunction,
contradicting~$\xymretope{\gamma}^t \cap F^t \cap V(\bar f) \neq
\emptyset$.
\end{proof}

\indexeqconj*
\begin{proof}
In this proof, for brevity we denote~$F^t \define F(I_0^t, I_1^t)$,
and~$\xymretope{\gamma}^t \define \xymretope{\gamma}^{(i_\gamma^t)}$.
In particular, because no fixing is applied at time~$t$,
we have~$F^{t + 1} = F^t$.

An index increasing event for permutation~$\gamma \in \Pi$ does not impact
any permutation but~$\gamma$,
so the update is correct for all~$\delta \in \Pi \setminus \{ \gamma \}$.
For~$\gamma$,
we prove correctness by showing~$\eqconj{\gamma}{t+1} \subseteq Y \cup Z$
and~$\eqconj{\gamma}{t+1} \supseteq Y \cup Z$,
respectively.
Consistent with the index increasing event, we have~%
$i_\gamma^{t + 1} = i_\gamma^t + 1$,~$I_0^{t + 1} = I_0^t$,
and~$I_1^{t + 1} = I_1^t$.

To prove the first inclusion, let~$\bar C \in \eqconj{\gamma}{t + 1}$
be a minimal~$i_\gamma^{t + 1}$-eq-conjunction.
Then~$\bar C$ is especially an~$i_\gamma^t$-eq-conjunction
(but not necessarily minimal).

On the one hand,
suppose that~$\bar C$ is a \emph{minimal}~$i_\gamma^t$-eq-conjunction,
that is~$\bar C \in \eqconj{\gamma}{t}$.
We show that~$\bar C \in Y$.
Since~$\bar C$ is an~$i_\gamma^{t + 1}$-eq-conjunction,
all~$x \in \hat F(C) \define F(I_0^t, I_1^t) \cap V(\bar C)$
satisfy~$x \preceqp{i_\gamma^{t + 1}} \gamma(x)$
and there is~$\tilde x \in \hat F(C)$
with~$\tilde x \eqp{i_\gamma^{t + 1}} \gamma(\tilde x)$.
As all~$x \in \xymretope\gamma^{t} \cap F^t \cap V(\bar C)$
have~$x \eqp{i_\gamma^t} \gamma(x)$,
all conditions of the set~$Y$
for~$\bar C \in \eqconj{\gamma}{t}$ are satisfied,
so~$\bar C \in Y$.

On the other hand, suppose that~$\bar C$ is
no minimal~$i_\gamma^t$-eq-conjunction.
We show that~$\bar C \in Z$.
There is~$C \in \eqconj{\gamma}{t}$ with~$C \subsetneq \bar C$,
such that all~$x \in \xymretope\gamma^{t} \cap F^t \cap V(C)$
have~$x \eqp{i_\gamma^t} \gamma(x)$.
Because~$\bar C \in \eqconj{\gamma}{t + 1}$
and~$C \subsetneq \bar C$,
there exists~$x \in \xymretope\gamma^{t + 1} \cap F^{t+1} \cap V(C)
\subseteq \xymretope\gamma^t \cap F^t \cap V(C)$
with~$x \succp{i_\gamma^{t + 1}} \gamma(x)$,
meaning that this vector~$x$
satisfies~$x_{i_\gamma^t} > \gamma(x)_{i_\gamma^t}$.
Hence,~$C \in \eqconj{\gamma}{t}$ satisfies the last condition of set~$Z$.
We show the remaining conditions with the following claim:

\begin{claim}
Let~$\bar C \in \eqconj{\gamma}{t + 1}$, and~$C \in \eqconj{\gamma}{t}$
with~$C \subsetneq \bar C$.
Suppose there exists a vector~$x \in \xymretope\gamma^t \cap F^t \cap V(C)$
with~$x_{i_\gamma^t} > \gamma(x)_{i_\gamma^t}$.
\begin{enumerate*}[label=\textbf{\emph{\small(\roman*)}},
ref=\emph{\small(\roman*.)}]
\item%
\label{claim:indexeqconj:1}%
If $x_{i_\gamma^t} = \gamma(x)_{i_\gamma}^t = 0$
for all~$x \in \xymretope\gamma^{t + 1} \cap F^{t+1} \cap V(\bar C)$,
then~$(i_\gamma^t, 0) \in \bar C$;
and
\item%
\label{claim:indexeqconj:2}%
if $x_{i_\gamma^t} = \gamma(x)_{i_\gamma}^t = 1$
for all~$x \in \xymretope\gamma^{t + 1} \cap F^{t+1} \cap V(\bar C)$,
then~$(\gamma^{-1}(i_\gamma^t), 1) \in \bar C$.
\end{enumerate*}
\end{claim}
\begin{claimproof}[Proof of claim]
The proof for statements~\ref{claim:indexeqconj:1}
and~\ref{claim:indexeqconj:2} are analogous, so we only prove it
for~\ref{claim:indexeqconj:1}.
Suppose that the precondition of \ref{claim:indexeqconj:1} holds,
and let~$\hat x \in \xymretope\gamma^{t+1} \cap F^{t+1} \cap V(\bar C)$,
such that~$\hat x \eqp{i_\gamma^t + 1} \gamma(\hat x)$
and~$\hat x_{i_\gamma^t} = 0$.
In particular,~$(i_\gamma^t, 1) \notin \bar C$.
For sake of contradiction, suppose
that~$(i_\gamma^t, 0) \notin \bar C$.
Then~$(i_\gamma^t, 0) \notin C$, as well.
We find a contradiction
for the three cases~$i_\gamma^t = \gamma(i_\gamma^t)$,
$i_\gamma^t < \gamma(i_\gamma^t)$
and~$i_\gamma^t > \gamma(i_\gamma^t)$,
together implying \ref{claim:indexeqconj:1}.
Since there exists a vector~$x \in \xymretope\gamma^t \cap F^t \cap V(C)$
with~$x_{i_\gamma^t} > \gamma(x)_{i_\gamma^t}$,
$i_\gamma^t$ is no fixed point of $\gamma$,
so~$i_\gamma^t \neq \gamma(i_\gamma^t)$.
If~$i_\gamma^t < \gamma(i_\gamma^t)$,
let~$z \in F^{t} \cap V(\bar C)$
with~$z_i = \hat x_i$ for~$i \neq i_\gamma^t$
and~$z_{i_\gamma^t} = 1$.
Then~$z
\succp{i_\gamma^{t + 1}}
\hat x
\eqp{i_\gamma^{t + 1}}
\gamma(\hat x)
\eqp{i_\gamma^{t + 1}}
\gamma(z)
$,
where the last equality is a consequence of that the vector~$\hat x$
and~$z$ only differ in entry~$i_\gamma^t$,
which is entry~$\gamma(i_\gamma^t) > i_\gamma^t$
of the permuted vectors~$\gamma(\hat x)$ and~$\gamma(z)$.
This is not possible, as~$\bar C$ is a~$i_\gamma^{t + 1}$-eq-conjunction.
Last, if~$i_\gamma^t > \gamma(i_\gamma^t)$,
let ~$x \in \xymretope\gamma^t \cap F^t \cap V(C)$
with~$x_{i_\gamma^t} = 1$,
and let~$z \in \xymretope\gamma^t \cap F^t \cap V(C)$
with~$z_i = x_i$ for all~$i \neq i_\gamma^t$,
and~$z_{i_\gamma^t} = 0$.
Then~$x \eqp{\gamma(i_\gamma^t)} z$
and~$\gamma(x) \eqp{\gamma(i_\gamma^t)} \gamma(z)$,
and~$x_{\gamma(i_\gamma^t)}
= \gamma(x)_{\gamma(i_\gamma^t)}
= x_{i_\gamma^t} = 1$,
while~$1 = x_{\gamma(i_\gamma^t)} = z_{\gamma(i_\gamma^t)}$
and~$\gamma(z_{\gamma(i_\gamma^t)}) = z_{i_\gamma^t} = 0$.
Hence,~$z \succp{\gamma(i_\gamma^t) + 1} \gamma(z)$,
which contradicts that~$C$ is a valid~$i_\gamma^t$-eq-conjunction.
A contradiction follows in all cases, so~$(i_\gamma^t, 0) \in \bar C$.
\end{claimproof}
By the claim, we thus have that~$(i_\gamma^t, 0) \in \bar C$
or~$(\gamma^{-1}(i_\gamma^t), 1) \in \bar C$.
The two cases are analogous, so suppose without loss of generality
that~$(i_\gamma^t, 0) \in \bar C$.
As~$C$ is an~$i_\gamma^t$-eq-conjunction,~%
$\hat C \define C \cup \{ (i_\gamma^t, 0) \} \subseteq \bar C$
is also an~$i_\gamma^t$-eq-conjunction.
In particular, it is also a~$i_\gamma^{t + 1}$-eq-conjunction,
since all~$x \in F(I_0^{t + 1}, I_1^{t + 1}) \cap V(\hat C)$
have~$x \preceqp{i_\gamma^t} \gamma(x)$
and~$0 = x_{i_\gamma^t} \leq \gamma(x)_{i_\gamma^t}$,
meaning that~$x \preceqp{i_\gamma^{t + 1}} \gamma(x)$.
As~$\hat C \subseteq \bar C$, there must also be one vector~%
$x \in F(I_0^{t + 1}, I_1^{t + 1}) \cap V(\hat C)$
with~$x \eqp{i_\gamma^{t + 1}} \gamma(x)$,
that is~%
$\xymretope\gamma^{t + 1} \cap F^{t+1} \cap V(\hat C) \neq \emptyset$.
Since~$\bar C \in \eqconj{\gamma}{t + 1}$
and~$\hat C \subseteq \bar C$ is an~$i_\gamma^{t + 1}$-eq-conjunction,
it must hold that~$\bar C = \hat C = C \cup S$,
where~$S = \{ (i_\gamma^t, 0) \}$.
All remaining conditions hold, so~$\bar C \in Z$.
Together with the analogous case, this proves
that~$\eqconj{\gamma}{t + 1} \subseteq Y \cup Z$.
In the remainder we prove the remaining inclusion.

Let~$\bar C \in Y$.
Then~$\bar C \in \eqconj{\gamma}{t}$,
so it is a valid~$i_\gamma^t$-eq-conjunction.
Added to this, by the second condition,
all~$x \in \xymretope\gamma^t \cap F^t \cap V(\bar C)$
have~$x \preceqp{i_\gamma^{t + 1}} \gamma(x)$,
and by the third condition,
there is an~$x \in \xymretope\gamma^t \cap F^t \cap V(\bar C)$
with~$x \eqp{i_\gamma^{t + 1}} \gamma(x)$.
Thus, $\bar C$ is an $(i_\gamma^{t+1})$-eq-conjunction, and it remains to
show minimality. This, however, follows immediately,
because for any~$C' \subsetneq \bar C$
there exists~$x \in \xymretope\gamma^t \cap F^t \cap V(C')$
with~$x \succp{i_\gamma^t} \gamma(x)$,
so in particular~$x \succp{i_\gamma^{t+1}} \gamma(x)$.
Hence,~$\bar C \in \eqconj{\gamma}{t + 1}$,
showing that~$Y \subseteq \eqconj{\gamma}{t + 1}$.

Let~$\bar C \in Z$, with representation~$C \cup S$
for~$C \in \eqconj{\gamma}{t}$,
and~$S = \{(i_\gamma^t, 0)\}$ or~$S = \{(\gamma^{-1}(i_\gamma^t), 1)\}$.
Denote $S = \{ (j, b) \}$.
Since there is~$x \in \xymretope\gamma^t \cap F^t \cap V(C)$
with~$x_{i_\gamma^t} > \gamma(x)_{i_\gamma^t}$,
$C$ is no~$i_\gamma^{t + 1}$-eq-conjunction.
By additionally applying $(j, b)$,
one enforces~$x_{i_\gamma^t} \leq \gamma(x)_{i_\gamma^t}$.
Since~$C$ is an~$i_\gamma^t$-eq-conjunction,
all~$x \in \xymretope\gamma^t \cap F^t \cap V(C \cup S)$
have~$x \eqp{i_\gamma^t} \gamma(x)$
and~$x_{i_\gamma^t} \leq \gamma(x)_{i_\gamma^t}$,
so~$x \preceqp{i_\gamma^{t + 1}} \gamma(x)$.
If~$\xymretope\gamma^t \cap F^t \cap V(C \cup S) \neq \emptyset$,
then there exists
a vector~$x \in \xymretope\gamma^t \cap F^t \cap V(C \cup S)$
with~$x \succeqp{i_\gamma^{t + 1}} \gamma(x)$
and~$x \preceqp{i_\gamma^{t + 1}} \gamma(x)$,
so~$x \eqp{i_\gamma^{t + 1}} \gamma(x)$.
This shows that~$C \cup S$ is an~$i_\gamma^{t + 1}$-eq-conjunction.
This is a minimal conjunction, by the following two arguments:
First, by the last condition there is
a vector~$x \in V(C)$
with~$x \succp{i_\gamma^{t + 1}} \gamma(x)$,
so the fixing from~$S$ cannot be removed.
Second,~$C$ is a minimal~$i_\gamma^t$-eq-conjunction,
so for any~$C' \subsetneq C$ there is
a vector~$x' \in \xymretope\gamma^t \cap F^t \cap V(C')$
with~$x' \succp{i_\gamma^t} \gamma(x')$.
Let~$x \in \binary^n$
with $x_i = x'_i$ for~$i \neq j$ and~$x_j = b$.
By the definition of $Z$ there is
a~$y \in \xymretope\gamma^t \cap F^t \cap V(C)$
with $y_{i_\gamma^t} > \gamma(y)_{i_\gamma^t}$,
implying~$(j, 1-b) \notin C'$,
so that~$x \in F(I_0^t, I_1^t) \cap V(C' \cup S)$.
Since the only difference between~$x$ and~$x'$ is entry~$j$,
a difference in the first $i_\gamma^t - 1$ entries between~$x$ and~$x'$
is only possible in the case~$(j, b) = (\gamma^{-1}(i_\gamma^t), 1)$
and~$j < i_\gamma^t$,
which makes $x_j \geq x'_j$.
Case $(j, b) = (i_\gamma^t, 0)$ does not affect the first $i_\gamma^t -1$
entries of~$x$ and~$x'$. Hence,~$x \succeqp{i_\gamma^t} x'$.
Analogously,~$\gamma(x) \preceqp{i_\gamma^t} \gamma(x')$.
Concluding, we have~$
x
\succeqp{i_\gamma^t}
x'
\succeqp{i_\gamma^t}
\gamma(x')
\succeqp{i_\gamma^t}
\gamma(x)
$,
showing the existence
of~$x \in \xymretope\gamma^{t + 1} \cap F^{t+1} \cap V(C' \cup S)$
with $x \succp{i_\gamma^t} \gamma(x)$.
Hence, no fixing from~$\bar C = C \cup S$ can be removed from the set
to find a smaller~$i_\gamma^{t + 1}$-eq-conjunction,
so~$\bar C = C \cup S \in \eqconj{\gamma}{t + 1}$,
showing that~$Z \subseteq \eqconj{\gamma}{t + 1}$.
Concluding with the previous
paragraphs~$\eqconj{\gamma}{t + 1} = Y \cup Z$ follows.
\end{proof}

\indexinfconj*
\begin{proof}
Again, for brevity we denote $\xymretope\gamma^t \define
\xymretope\gamma^{(i_\gamma^t)}$ and $F^t \define F(I_0^t, I_1^t)$.
As the event at time~$t$ is an index increasing event,
in particular $F^{t+1} = F^t$.

An index increasing event for permutation~$\gamma \in \Pi$ does not impact
any permutation but~$\gamma$,
so the update is correct for all~$\delta \in \Pi \setminus \{ \gamma \}$.
For~$\gamma$,
we prove correctness in four parts:
First,
\begin{enumerate*}[label=\textbf{\emph{\small(\roman*)}},
ref=\emph{\small(\roman*.)}]
\item
\label{prop:varfixeqconj:proof:1}
all~$\bar C \in \infconj{\gamma}{t + 1}$
that are no~$i_\gamma^t$-inf-conjunction
have~$\bar C \in \infconj{\gamma, \mathrm{eq}}{t + 1}$,
then
\item%
\label{prop:varfixeqconj:proof:2}%
$
\infconj{\gamma, \mathrm{eq}}{t + 1}
\subseteq
\infconj{\gamma}{t + 1}
$,
after which we show that
\item%
\label{prop:varfixeqconj:proof:3}%
all~$\bar C \in \infconj{\gamma}{t + 1}$
that are~$i_\gamma^t$-inf-conjunction
have~$\bar C \in \infconj{\gamma, \mathrm{inf}}{t + 1}$,
and last
\item%
\label{prop:varfixeqconj:proof:4}%
$
\infconj{\gamma, \mathrm{inf}}{t + 1}
\subseteq
\infconj{\gamma}{t + 1}
$.
\end{enumerate*}
Part~\ref{prop:varfixeqconj:proof:1} and~\ref{prop:varfixeqconj:proof:3}
show~%
$
\infconj{\gamma}{t + 1}
\subseteq \infconj{\gamma, \mathrm{eq}}{t + 1} \cup
\infconj{\gamma, \mathrm{inf}}{t + 1}
$,
and part~\ref{prop:varfixeqconj:proof:2}
and~\ref{prop:varfixeqconj:proof:4}
show the converse inclusion,
therewith proving the proposition.
Consistent with the index increasing event, we have~%
$i_\gamma^{t + 1} = i_\gamma^t + 1$,~$I_0^{t + 1} = I_0^t$,
and~$I_1^{t + 1} = I_1^t$.

To show \ref{prop:varfixeqconj:proof:1},
let~$\bar C \in \infconj{\gamma}{t + 1}$
be a minimal~$i_\gamma^{t + 1}$-inf-conjunction,
and
suppose that~$\bar C$ is no~$i_\gamma^t$-inf-conjunction.
Let~$x \in \xymretope{\gamma}^t \cap F^t \cap V(\bar C)$.
Then~$x \precp{i_\gamma^{t + 1}} \gamma(x)$
as~$\bar C$ is an~$i_\gamma^{t + 1}$-inf-conjunction,
and~$x \succeqp{i_\gamma^t} \gamma(x)$
as~$\bar C$ is no~$i_\gamma^{t}$-inf-conjunction.
Together,~$x \eqp{i_\gamma^t} \gamma(x)$ must hold,
and~$x_{i_\gamma^t} < \gamma(x)_{i_\gamma^t}$.
Hence,~$\bar C$ is an~$i_\gamma^t$-eq-conjunction.

If~$\bar C$ is a minimal~$i_\gamma^t$-eq-conjunction,
i.e.,~$\bar C \in \eqconj{\gamma}{t}$,
then for any~$x \in \xymretope\gamma^t \cap F^t \cap V(\bar C)$
holds~$0 = x_{i_\gamma^t} < \gamma(x)_{i_\gamma^t} = 1$.
Thus, the conditions of~$\infconj{\gamma, \mathrm{eq}}{t + 1}$
hold for~$C = \bar C \in \eqconj{\gamma}{t}$
and~$S = \emptyset$
(as~$x_{i_\gamma^t} = 0$, $\gamma(x)_{i_\gamma^t} = 1$
for all~$x \in \xymretope\gamma^t \cap F^t \cap V(\bar C)$),
that is~$\bar C = C \cup S \in \infconj{\gamma, \mathrm{eq}}{t + 1}$.

Otherwise, if~$\bar C$ is a non-minimal~$i_\gamma^t$-eq-conjunction,
let~$C \subsetneq \bar C$ with~$C \in \eqconj{\gamma}{t}$.
By minimality of~$\bar C \in \infconj{\gamma}{t + 1}$,
there exists an~$x \in \xymretope\gamma^{t + 1} \cap F^{t+1} \cap V(C)$
with~$x \succeqp{i_\gamma^{t + 1}} \gamma(x)$.
Since~$C \in \eqconj{\gamma}{t}$
also~$x \eqp{i_\gamma^t} \gamma(x)$.
Thus,~$x_{i_\gamma^t} \geq \gamma(x)_{i_\gamma^t}$.

Suppose that~$(i_\gamma^t, 0) \in \bar C \setminus C$.
Since~$\bar C \in \infconj{\gamma}{t + 1}$,
all~$x \in \xymretope\gamma^{t + 1} \cap F^{t+1}
\cap V(\bar C \setminus \{ (i_\gamma^t, 0) \} )$
have~$x_{i_\gamma^t} = 1$,
and by minimality also such a vector exists.
By~%
$
C \subseteq \bar C \setminus \{ (i_\gamma^t, 0) \}
\subsetneq \bar C
$
and~$\xymretope\gamma^{t + 1} \cap F^{t+1}
\subseteq \xymretope\gamma^{t} \cap F^t $,
we have~$
x \in
\xymretope\gamma^{t + 1} \cap F^{t+1}
 \cap V(\bar C \setminus \{ (i_\gamma^t, 0)\})
\subseteq
\xymretope\gamma^{t} \cap F^t \cap V(C)
$.
Hence,~$(i_\gamma^t, 0) \in \bar C \setminus C$
implies that not all~$x \in \xymretope\gamma^t \cap F^t \cap V(C)$
have~$x_{i_\gamma^t} = 0$.
Conversely, suppose that all~%
$x \in \xymretope\gamma^t \cap F^t \cap V(C)$
have~$x_{i_\gamma^t} = 0$.
If~$(i_\gamma, 0) \in \bar C \setminus C$,
then
by~$\bar C \in \infconj{\gamma}{t + 1}$
there is an~$x
\in \xymretope\gamma^{t + 1} \cap F^{t+1}
\cap V(\bar C \setminus \{ (i_\gamma, 0) \})$
with~$x_{i_\gamma^t} = 1$,
which contradicts~%
$
x \in \xymretope\gamma^{t + 1} \cap F^{t+1}
\cap V(\bar C \setminus \{ (i_\gamma, 0) \})
\subseteq
\xymretope\gamma^t \cap F^t \cap V(C)
$.
Hence, together with the previous part,
we find that either~$(i_\gamma^t, 0) \in \bar C \setminus C$
or that all~$x \in \xymretope\gamma^t \cap F^t \cap V(C)$
have~$x_{i_\gamma^t} = 0$.
Analogously one can show that either~%
$(\gamma^{-1}(i_\gamma^t), 1) \in \bar C \setminus C$,
or that~$\gamma(x)_{i_\gamma^t} = 1$
for all~$x \in \xymretope\gamma^t \cap F^t \cap V(C)$.

Consistent with the definition of~%
$\infconj{\gamma, \mathrm{eq}}{t + 1}$,
let~$S \subseteq \{ (i_\gamma^t, 0), (\gamma^{-1}(i_\gamma^t), 1) \}$
with~$(i_\gamma^t, 0) \in S$
if there is an~$x \in \xymretope\gamma^t \cap F^t \cap V(C)$
with~$x_{i_\gamma^t} = 1$,
and~$(\gamma^{-1}(i_\gamma^t), 1) \in S$
if there is an~$x \in \xymretope\gamma^t \cap F^t \cap V(C)$
with~$\gamma(x)_{i_\gamma^t} = 0$.
Then~$C \cup S$ is an~$i_\gamma^{t + 1}$-inf-conjunction,
because~$C$ is an~$i_\gamma^t$-eq-conjunction,
and the fixings from~$S$ additionally ensure
that~$x_{i_\gamma^t} < \gamma(x)_{i_\gamma^t}$
for all~$x \in \xymretope\gamma^t \cap F^t \cap V(C \cup S)$.
As a result of the previous paragraph, we thus have
that~$f_0 \define (i_\gamma^t, 0) \in S$
if and only if~$f_0 \in \bar C \setminus C$,
and analogously
that~$f_1 \define (\gamma^{-1}(i_\gamma^t), 1) \in S$
if and only if~$f_1 \in \bar C \setminus C$.
Hence,~$S \subseteq \bar C \setminus C$,
showing that~$C \cup S \subseteq \bar C$.
Since~$\bar C$
is a minimal~$i_\gamma^{t + 1}$-inf-conjunction
and~$C \cup S$ is a~$i_\gamma^{t + 1}$-inf-conjunction
it must thus hold that~$\bar C = C \cup S$.
This shows that~$\bar C \in \infconj{\gamma, \mathrm{eq}}{t + 1}$.

To prove \ref{prop:varfixeqconj:proof:2},
let~$\bar C \in \infconj{\gamma, \mathrm{eq}}{t + 1}$
be represented by~$C \cup S$
for~$C \in \eqconj{\gamma}{t}$,
and~$S \subseteq \{ (i_\gamma^t, 0), (\gamma^{-1}(i_\gamma^t), 1) \}$,
where~$S$ is consistent with the definition.
Since~$C \in \eqconj{\gamma}{t}$,
all~$x \in \xymretope\gamma^{t} \cap F^t \cap V(C \cup S)$
have~$x \eqp{i_\gamma^t} \gamma(x)$,
and by the definition of~$S$,
also~$x_{i_\gamma^t} < \gamma(x)_{i_\gamma^t}$.
Hence, all~$x \in \xymretope\gamma^{t} \cap F^t \cap V(C \cup S)$
have~$x \precp{i_\gamma^{t + 1}} \gamma(x)$,
such that~%
$\xymretope\gamma^{t + 1} \cap F^{t+1} \cap V(C \cup S) = \emptyset$,
meaning that~$\bar C = C \cup S$ is an~$i_\gamma^{t + 1}$-inf-conjunction.
It is also minimal, since removing any element from~$\bar C$
yields a non-empty
set~$\xymretope\gamma^{t+1} \cap F^{t+1} \cap V(\bar C)$,
as we show next.
Note that, by the definition of~$S$, that~$C$ and~$S$ are non-overlapping.
Consider~$f \in C$.
Since~$C \in \eqconj{\gamma}{t}$, there is a vector~%
$x \in \xymretope\gamma^t \cap F^t \cap V(C \setminus \{ f \})$
with~$x \succp{i_\gamma^t} \gamma(x)$.
If~$x \notin V(S)$,
one can construct~%
$y \in \xymretope\gamma^t \cap F^t \cap V((C \cup S) \setminus \{ f \})$
by copying~$x$ and applying the fixings of~$S$.
The application of fixing~$(i_\gamma^t, 0)$ does not affect the first~%
$i_\gamma^t - 1$ entries of~$x$,
and can only make vector~$\gamma(x)$
partially lexicographically smaller up to~$i_\gamma^t$,
as it changes a~$1$-entry to a~$0$-entry.
Likewise, the application of~$(\gamma^{-1}(i_\gamma^t), 0)$ does not
affect the first~$i_\gamma^t - 1$ entries of~$\gamma(x)$
and can only make vector~$x$
partially lexicographically larger up to~$i_\gamma^t$.
Hence,~$
y
\succeqp{i_\gamma^t} x
\succp{i_\gamma^t} \gamma(x)
\succeqp{i_\gamma^t} \gamma(y)
$,
showing that~$y \succp{i_\gamma^t} \gamma(y)$,
so in particular~%
$y \succp{i_\gamma^{t + 1}} \gamma(y)$,
meaning that~%
$y \in \xymretope\gamma^{t + 1} \cap F^{t+1}
\cap V((C \cup S) \setminus \{ f \})$.
Conversely, consider~$f \in S$,
and by analogy between the two possibilities,
suppose without loss of generality~$f = (i_\gamma^t, 0) \in S$.
By the definition of~$S$,
there exists~$x \in \xymretope\gamma^t \cap F^t \cap V(C)$
with~$x \eqp{i_\gamma^t} \gamma(x)$ and~$x_{i_\gamma^t} = 1$,
that is~$x \succeqp{i_\gamma^{t + 1}} \gamma(x)$.
Also here, if~$x \notin V(S \setminus \{ f \})$,
one can apply the fixing from~$S \setminus \{ f \}$
and find~%
$y \in \xymretope\gamma^{t + 1} \cap F^{t+1}
\cap V((C \cup S) \setminus \{ f \})$
with~$y \succeqp{i_\gamma^{t + 1}} \gamma(y)$.
Hence, no~$i_\gamma^{t + 1}$-inf-conjunction could be found by dropping
any fixing from~$\bar C = C \cup S$,
showing that~$\bar C \in \infconj{\gamma}{t + 1}$,
and thus~%
$\infconj{\gamma, \mathrm{eq}}{t + 1} \subseteq \infconj{\gamma}{t + 1}$.

Next, to show \ref{prop:varfixeqconj:proof:3},
let~$\bar C \in \infconj{\gamma}{t + 1}$
be a minimal~$i_\gamma^{t + 1}$-inf-conjunction
that is also a~$i_\gamma^{t}$-inf-conjunction.
Since~$\bar C$ is a minimal~$i_\gamma^{t + 1}$-inf-conjunction,
for all~$C \subsetneq \bar C$
there exists~$x \in F(I_0^t, I_1^t) \cap V(C)$
with~$x \succeqp{i_\gamma^{t + 1}} \gamma(x)$,
so especially with~$x \succeqp{i_\gamma^{t}} \gamma(x)$.
This shows that~$\bar C$ is also a minimal~$i_\gamma^t$-inf-conjunction,
meaning that~$\bar C \in \infconj{\gamma}{t}$.
Since by \ref{prop:varfixeqconj:proof:2} we have~%
$\infconj{\gamma, \mathrm{eq}}{t + 1} \subseteq \infconj{\gamma}{t + 1}$,
a contradiction with minimality of~$\bar C$ at time~$t + 1$
would follow if there was a~%
$C \in \infconj{\gamma, \mathrm{eq}}{t + 1}$
with~$C \subsetneq \bar C$.
Hence,~$\bar C \in \infconj{\gamma, \mathrm{inf}}{t + 1}$.
The combination of part~\ref{prop:varfixeqconj:proof:1}
and~\ref{prop:varfixeqconj:proof:3} shows~$
\infconj{\gamma}{t + 1}
\subseteq
\infconj{\gamma, \mathrm{eq}}{t + 1}
\cup
\infconj{\gamma, \mathrm{inf}}{t + 1}
$.

Last, to show part \ref{prop:varfixeqconj:proof:4},
let~$\bar C \in \infconj{\gamma, \mathrm{inf}}{t + 1}$.
Then~$\bar C \in \infconj{\gamma}{t}$
so it is a minimal~$i_\gamma^t$-inf-conjunction.
In particular it is a~$i_\gamma^{t + 1}$-inf-conjunction.
Suppose that it is no minimal~$i_\gamma^{t + 1}$-inf-conjunction.
Then there is a~$C \in \infconj{\gamma}{t + 1}$
with~$C \subsetneq \bar C$.
By the previous parts,~%
$C \in \infconj{\gamma, \mathrm{eq}}{t + 1}
\cup \infconj{\gamma, \mathrm{inf}}{t + 1}$.
If~$C \in \infconj{\gamma, \mathrm{inf}}{t + 1}$,
then~$C \in \infconj{\gamma}{t}$, and that contradicts
minimality of~$\bar C \in \infconj{\gamma}{t}$.
If~$C \in \infconj{\gamma, \mathrm{eq}}{t + 1}$,
then this contradicts
the definition of~$\bar C \in \infconj{\gamma, \mathrm{inf}}{t + 1}$.
Since either case yields a contradiction,~%
$\bar C$ must be a minimal~$i_\gamma^{t + 1}$-inf-conjunction,
meaning that~$\bar C \in \infconj{\gamma}{t + 1}$,
proving~$\infconj{\gamma, \mathrm{inf}}{t + 1}
\subseteq \infconj{\gamma}{t + 1}$.
\end{proof}

\varfixinfconj*
\begin{proof}
Denote the right-hand side of Equation~\eqref{eq:varfixinfconj}
by~$H$.
We prove the correctness of Equation~\eqref{eq:varfixinfconj}
by showing~$\infconj{\gamma}{t + 1} \subseteq H$
and~$\infconj{\gamma}{t + 1} \supseteq H$, respectively.

Let~$\bar C \in \infconj{\gamma}{t + 1}$.
As fixing~$f$ is applied and~$\bar C$ is minimal,~$f \notin \bar C$.
On the one hand,
suppose that~$\bar C$ is an~$i_\gamma^t$-inf-conjunction at time~$t$.
It is necessarily minimal, as otherwise there exists
a~$C \in \infconj{\gamma}{t}$ with~$C \subsetneq \bar C$.
$C$ is also an~$i_\gamma^{t}$-inf-conjunction at time~$t + 1$,
contradicting minimality of $\bar C$.
Hence, $\bar C \in \infconj{\gamma}{t}$.
Suppose that~$C' \in \infconj{\gamma}{t}$ with~%
$C' \setminus \{ f \} \subsetneq \bar C \setminus \{ f \} = \bar C$.
Then~$C' \setminus \{ f \}$ is an~$i_\gamma^t$-inf-conjunction
at time~$t + 1$, and inclusionwise smaller than~$\bar C$,
contradicting $\bar C \in \infconj{\gamma}{t + 1}$.
Consequently, $\bar C \setminus \{ f \} = \bar C \in H$.

On the other hand, suppose that~$\bar C$
is no~$i_\gamma^t$-inf-conjunction at time~$t$.
As the only difference from time~$t$ to~$t + 1$ is the application of
fixing~$f$,~$C \define \bar C \cup \{ f \}$
is an~$i_\gamma^t$-inf-conjunction at time~$t$.
This is also minimal,
since~$\bar C$ is no~$i_\gamma^t$-inf-conjunction at time~$t$,
meaning that~$f$ cannot be removed from~$C$,
and~$\bar C$ is a minimal~$i_\gamma^t$-inf-conjunction at time~$t + 1$,
meaning that no element from~$\bar C$ can be removed from~$C$, as well.
Hence,~$C \in \infconj{\gamma}{t}$.
Now suppose that~$C' \in \infconj{\gamma}{t}$
with~$C' \setminus \{ f \} \subsetneq C \setminus \{ f \} = \bar C$.
Then~$C' \setminus \{ f \}$ is an~$i_\gamma^t$-inf-conjunction
at time~$t + 1$, which violates the minimality
of~$\bar C \in \infconj{\gamma}{t + 1}$.
Concluding,~$C \setminus \{ f \} = \bar C \in H$.
Together this proves that~$\infconj{\gamma}{t + 1} \subseteq H$.

Finally, to show the reverse inclusion, let~$\bar C \in H$
represented by~$C \setminus \{ f \}$ for some~$C \in \infconj{\gamma}{t}$.
Since~$C$ is an~$i_\gamma^t$-inf-conjunction at time~$t$,
so is~$\bar C$ one at time~$t + 1$.
Hence, if~$\bar C \notin \infconj{\gamma}{t + 1}$,
there exists a~$\bar C' \in \infconj{\gamma}{t + 1}$
with~$\bar C' \subsetneq \bar C$.
By the first part of the proof,
there exists~$C' \in \infconj{\gamma}{t}$
with~$\bar C' = C' \setminus \{ f \}$,
contradicting minimality of $C$.
Consequently, $\bar C \in \infconj{\gamma}{t + 1}$
showing $\infconj{\gamma}{t + 1} \supseteq H$.
\end{proof}

\varfixeqconj*
\begin{proof}
Denote the right-hand side of Equation~\eqref{eq:varfixeqconj} by~$H$.
We prove the correctness of Equation~\eqref{eq:varfixeqconj} by
showing~$\eqconj{\gamma}{t+1} \subseteq H$
and~$\eqconj{\gamma}{t+1} \supseteq H$, respectively.

First, let~$\bar C \in \eqconj{\gamma}{t + 1}$.
Suppose we have already established the existence
of~$C \in \eqconj{\gamma}{t}$ with~$\bar C = C \setminus \{ f \}$.
We show that in this case the conditions of the right-hand side
of Equation~\eqref{eq:varfixeqconj} hold.
If there is a~$C' \in \eqconj{\gamma}{t}$
with~$C' \setminus \{ f \} \subsetneq C \setminus \{ f \} = \bar C$,
then~$C' \setminus \{ f \}$ is also an~$i_\gamma^t$-eq-conjunction at
time~$t + 1$, which is not possible as~$\bar C$ is inclusionwise minimal.
Likewise, if there is~$C' \in \infconj{\gamma}{t}$
with~$C' \setminus \{ f \} \subseteq C \setminus \{ f \} = \bar C$,
then~$\bar C$ is both an~$i_\gamma^t$-inf-conjunction at time~$t + 1$
and an~$i_\gamma^t$-eq-conjunction at time~$t + 1$,
which is not possible.

We turn the missing proof of existence of~$C \in \eqconj{\gamma}{t}$.
On the one hand, suppose that~$\bar C$ is an~$i_\gamma^t$-eq-conjunction
at time~$t$.
Analogous to the proof of Proposition~\ref{prop:varfixinfconj},
we conclude that~$\bar C$ is necessarily a minimal
$i_\gamma^t$-eq-conjunction, so~$\bar C \in \eqconj{\gamma}{t}$.
On the other hand,
suppose that~$\bar C$ is no~$i_\gamma^t$-eq-conjunction at time~$t$.
Then~$C \define \bar C \cup \{ f \}$ is an~$i_\gamma^t$-eq-conjunction
at time~$t$.
This is minimal, as removing~$f$ from~$C$ yields~$\bar C$, which is
no~$i_\gamma^t$-eq-conjunction at time~$t$,
and if a valid~$i_\gamma^t$-eq-conjunction at time~$t$ is found when
removing any element from~$\bar C$, then this would violate minimality
of~$\bar C$ at time~$t + 1$.
Hence,~${C \in \eqconj{\gamma}{t}}$.
So in both cases there is a~$C \in \eqconj{\gamma}{t}$
with~$\bar C = C \setminus \{ f \}$.
Consequently $\eqconj{\gamma}{t + 1} \subseteq H$.

Second and last, to show the reverse inclusion,
let~$\bar C \in H$ represented
by~$C \setminus \{ f \}$ for some~$C \in \eqconj{\gamma}{t}$.
As~$C$ is an~$i_\gamma^t$-eq-conjunction at time~$t$,
also~$\bar C$ is an~$i_\gamma^t$-eq-conjunction at time~$t + 1$.
Suppose that~$\bar C \notin \eqconj{\gamma}{t + 1}$,
meaning that minimality does not hold.
Then there is~$\bar C' \in \eqconj{\gamma}{t + 1}$
with~$\bar C' \subsetneq \bar C$.
By~$\eqconj{\gamma}{t + 1} \subseteq H$,
we have that~$\bar C' = C' \setminus \{ f \}$
for some~$C' \in \eqconj{\gamma}{t}$.
But then~%
$C' \setminus \{ f \} = \bar C' \subsetneq \bar C = C \setminus \{ f \}$
contradicts~$\bar C \in H$.
Hence,~$\bar C \in \eqconj{\gamma}{t + 1}$,
and thus~$\eqconj{\gamma}{t + 1} \supseteq H$.
\end{proof}

\clearpage
\section{Supplements}
\label{app:supplements}

This appendix provides the per-instance results of the experiments
described in  Section~\ref{sec:computational}.

\begin{table}[h]
\caption{3-edge coloring on flower snark graphs $J_n$, original
relabeling.}
\scriptsize
\centering
\begin{tabular}{l*{10}{rr}}
\toprule
& \multicolumn{2}{c}{\tt nosym}& \multicolumn{2}{c}{\tt gen}&
\multicolumn{2}{c}{\tt group}& \multicolumn{2}{c}{\tt nopeek}&
\multicolumn{2}{c}{\tt peek}\\
\cmidrule(l{1pt}r{1pt}){2-3}
\cmidrule(l{1pt}r{1pt}){4-5}
\cmidrule(l{1pt}r{1pt}){6-7}
\cmidrule(l{1pt}r{1pt}){8-9}
\cmidrule(l{1pt}r{1pt}){10-11}
instance & time(s) & S & time(s) & S & time(s) & S & time(s) & S & time(s)
& S \\
\midrule
3                    &     0.01 &  5 &     0.00 &  5 &     0.00 &  5 &
0.00 &  5 &     0.00 &  5 \\
5                    &     0.02 &  5 &     0.01 &  5 &     0.01 &  5 &
0.01 &  5 &     0.01 &  5 \\
7                    &     0.37 &  5 &     0.03 &  5 &     0.03 &  5 &
0.03 &  5 &     0.03 &  5 \\
9                    &     0.65 &  5 &     0.20 &  5 &     0.19 &  5 &
0.17 &  5 &     0.17 &  5 \\
11                   &     2.68 &  5 &     0.58 &  5 &     0.53 &  5 &
0.55 &  5 &     0.55 &  5 \\
13                   &    16.51 &  5 &     1.39 &  5 &     1.21 &  5 &
1.32 &  5 &     1.34 &  5 \\
15                   &    48.17 &  5 &     5.29 &  5 &     5.11 &  5 &
4.68 &  5 &     4.86 &  5 \\
17                   &    95.36 &  5 &    15.98 &  5 &    13.15 &  5 &
12.08 &  5 &    12.32 &  5 \\
19                   &   198.58 &  5 &    30.62 &  5 &    20.92 &  5 &
28.42 &  5 &    29.56 &  5 \\
21                   &   339.21 &  5 &    51.57 &  5 &    41.70 &  5 &
40.66 &  5 &    40.84 &  5 \\
23                   &  3141.73 &  4 &    52.46 &  5 &    59.23 &  5 &
58.71 &  5 &    57.97 &  5 \\
25                   &  7200.00 &  0 &    71.86 &  5 &    73.22 &  5 &
86.04 &  5 &    86.84 &  5 \\
27                   &  7200.00 &  0 &    81.67 &  5 &   110.99 &  5 &
132.82 &  5 &   134.37 &  5 \\
29                   &  7200.00 &  0 &   290.86 &  5 &   317.28 &  5 &
294.04 &  5 &   229.28 &  5 \\
31                   &  7200.00 &  0 &   297.63 &  5 &   348.92 &  5 &
423.39 &  5 &   425.73 &  5 \\
33                   &  7200.00 &  0 &  1659.06 &  2 &   708.61 &  4 &
324.87 &  5 &   349.46 &  5 \\
35                   &  7200.00 &  0 &  1177.03 &  4 &  3689.75 &  2 &
1834.47 &  3 &  1356.33 &  3 \\
37                   &  7200.00 &  0 &  1588.07 &  3 &  3603.63 &  2 &
3701.40 &  2 &  2523.97 &  2 \\
39                   &  7200.00 &  0 &  2122.42 &  2 &  3375.47 &  1 &
3797.39 &  1 &  3277.88 &  2 \\
41                   &  7200.00 &  0 &  3798.78 &  1 &  2679.75 &  2 &
5067.89 &  1 &  1923.04 &  3 \\
43                   &  7200.00 &  0 &  3553.34 &  1 &  6878.06 &  1 &
5734.13 &  2 &  6885.16 &  1 \\
45                   &  7200.00 &  0 &  7200.00 &  0 &  7200.00 &  0 &
4892.76 &  1 &  2481.88 &  3 \\
47                   &  7200.00 &  0 &  7200.00 &  0 &  7200.00 &  0 &
2322.81 &  2 &  2507.41 &  2 \\
49                   &  7200.00 &  0 &  7200.00 &  0 &  7200.00 &  0 &
4666.12 &  1 &  4616.06 &  1 \\
\cmidrule{2-11}
All instances combined                      &   730.78 & 54 &   172.35 &
88 &   187.56 & 87 &   169.79 & 93 &   153.23 & 97 \\
Total time                                  &
\multicolumn{2}{r}{136:26:10     } & \multicolumn{2}{r}{ 69:19:48     } &
\multicolumn{2}{r}{ 74:10:26     } & \multicolumn{2}{r}{ 63:51:23     } &
\multicolumn{2}{r}{ 54:15:29     } \\
Symmetry time                               & \multicolumn{2}{r}{
0:00:00     } & \multicolumn{2}{r}{  1:37:28     } & \multicolumn{2}{r}{
2:18:15     } & \multicolumn{2}{r}{  2:20:04     } & \multicolumn{2}{r}{
9:03:07     } \\
Percentage time                             & \multicolumn{2}{r}{
0.0\%} & \multicolumn{2}{r}{         2.3\%} & \multicolumn{2}{r}{
3.1\%} & \multicolumn{2}{r}{         3.7\%} & \multicolumn{2}{r}{
16.7\%} \\
\bottomrule
\end{tabular}
\end{table}

\begin{table}[h]
\caption{3-edge coloring on flower snark graphs $J_n$,
{\tt max}-relabeling.}
\scriptsize
\centering
\begin{tabular}{l*{10}{rr}}
\toprule
& \multicolumn{2}{c}{\tt nosym}& \multicolumn{2}{c}{\tt gen}&
\multicolumn{2}{c}{\tt group}& \multicolumn{2}{c}{\tt nopeek}&
\multicolumn{2}{c}{\tt peek}\\
\cmidrule(l{1pt}r{1pt}){2-3}
\cmidrule(l{1pt}r{1pt}){4-5}
\cmidrule(l{1pt}r{1pt}){6-7}
\cmidrule(l{1pt}r{1pt}){8-9}
\cmidrule(l{1pt}r{1pt}){10-11}
instance & time(s) & S & time(s) & S & time(s) & S & time(s) & S & time(s)
& S \\
\midrule
3                    & --       &       &     0.00 &  5 &     0.00 &  5
&     0.00 &  5 &     0.00 &  5 \\
5                    & --       &       &     0.02 &  5 &     0.01 &  5
&     0.01 &  5 &     0.01 &  5 \\
7                    & --       &       &     0.05 &  5 &     0.04 &  5
&     0.04 &  5 &     0.04 &  5 \\
9                    & --       &       &     0.44 &  5 &     0.34 &  5
&     0.35 &  5 &     0.34 &  5 \\
11                   & --       &       &     0.80 &  5 &     0.80 &  5
&     0.62 &  5 &     0.58 &  5 \\
13                   & --       &       &     2.52 &  5 &     1.58 &  5
&     1.29 &  5 &     1.35 &  5 \\
15                   & --       &       &     8.88 &  5 &     4.48 &  5
&     4.92 &  5 &     4.22 &  5 \\
17                   & --       &       &    28.51 &  5 &    12.36 &  5
&    11.17 &  5 &    10.82 &  5 \\
19                   & --       &       &    60.61 &  5 &    31.00 &  5
&    35.13 &  5 &    30.12 &  5 \\
21                   & --       &       &    95.28 &  5 &    55.67 &  5
&    52.78 &  5 &    52.47 &  5 \\
23                   & --       &       &   143.54 &  5 &   108.25 &  5
&    84.91 &  5 &   108.19 &  5 \\
25                   & --       &       &   215.28 &  5 &   122.76 &  5
&   170.44 &  5 &   140.46 &  5 \\
27                   & --       &       &  1511.94 &  3 &   340.70 &  5
&   211.61 &  5 &   290.88 &  5 \\
29                   & --       &       &  4975.37 &  2 &  1180.29 &  3
&   529.31 &  5 &   481.28 &  5 \\
31                   & --       &       &  7200.00 &  0 &  6579.03 &  1 &
1922.91 &  4 &  1381.55 &  3 \\
33                   & --       &       &  7200.00 &  0 &  7200.00 &  0 &
3435.54 &  2 &  4207.49 &  2 \\
35                   & --       &       &  7200.00 &  0 &  4618.61 &  1 &
6872.39 &  1 &  4162.52 &  2 \\
37                   & --       &       &  7200.00 &  0 &  7200.00 &  0 &
7200.00 &  0 &  6535.44 &  1 \\
\multicolumn{11}{c}{\emph{[Results
for~$n \in \{39, \dots, 49\}$ are omitted, since no setting could
solve the instance within the time limit.]}} \\
\cmidrule{2-11}
All instances combined                      & --       &       &   407.02
& 65 &   312.93 & 70 &   278.00 & 77 &   270.19 & 78 \\
Total time                                  & --            & &
\multicolumn{2}{r}{113:32:08     } & \multicolumn{2}{r}{102:47:17     } &
\multicolumn{2}{r}{ 94:51:24     } & \multicolumn{2}{r}{ 90:47:44     } \\
Symmetry time                               & --            & &
\multicolumn{2}{r}{  2:41:02     } & \multicolumn{2}{r}{  3:31:28     } &
\multicolumn{2}{r}{  3:02:04     } & \multicolumn{2}{r}{ 16:23:01     } \\
Percentage time                             & --          & &
\multicolumn{2}{r}{         2.4\%} & \multicolumn{2}{r}{         3.4\%} &
\multicolumn{2}{r}{         3.2\%} & \multicolumn{2}{r}{        18.0\%} \\
\bottomrule
\end{tabular}
\end{table}

\begin{table}[h]
\caption{3-edge coloring on flower snark graphs $J_n$,
{\tt min}-relabeling.}
\scriptsize
\centering
\begin{tabular}{l*{10}{rr}}
\toprule
& \multicolumn{2}{c}{\tt nosym}& \multicolumn{2}{c}{\tt gen}&
\multicolumn{2}{c}{\tt group}& \multicolumn{2}{c}{\tt nopeek}&
\multicolumn{2}{c}{\tt peek}\\
\cmidrule(l{1pt}r{1pt}){2-3}
\cmidrule(l{1pt}r{1pt}){4-5}
\cmidrule(l{1pt}r{1pt}){6-7}
\cmidrule(l{1pt}r{1pt}){8-9}
\cmidrule(l{1pt}r{1pt}){10-11}
instance & time(s) & S & time(s) & S & time(s) & S & time(s) & S & time(s)
& S \\
\midrule
3                    & --       &       &     0.00 &  5 &     0.00 &  5
&     0.00 &  5 &     0.00 &  5 \\
5                    & --       &       &     0.01 &  5 &     0.01 &  5
&     0.02 &  5 &     0.02 &  5 \\
7                    & --       &       &     0.03 &  5 &     0.03 &  5
&     0.03 &  5 &     0.03 &  5 \\
9                    & --       &       &     0.05 &  5 &     0.05 &  5
&     0.07 &  5 &     0.07 &  5 \\
11                   & --       &       &     0.69 &  5 &     0.55 &  5
&     0.56 &  5 &     0.55 &  5 \\
13                   & --       &       &     1.22 &  5 &     1.38 &  5
&     1.14 &  5 &     1.14 &  5 \\
15                   & --       &       &     2.52 &  5 &     1.96 &  5
&     1.59 &  5 &     1.61 &  5 \\
17                   & --       &       &     7.77 &  5 &     5.17 &  5
&     3.99 &  5 &     3.38 &  5 \\
19                   & --       &       &    19.32 &  5 &    13.54 &  5
&    14.61 &  5 &    14.21 &  5 \\
21                   & --       &       &    28.24 &  5 &    31.52 &  5
&    28.03 &  5 &    28.31 &  5 \\
23                   & --       &       &    51.41 &  5 &    46.26 &  5
&    35.49 &  5 &    38.36 &  5 \\
25                   & --       &       &    73.62 &  5 &    62.88 &  5
&    61.79 &  5 &    65.63 &  5 \\
27                   & --       &       &    73.12 &  5 &    83.37 &  5
&    85.86 &  5 &    88.03 &  5 \\
29                   & --       &       &    99.74 &  5 &    95.35 &  5
&   112.66 &  5 &   117.75 &  5 \\
31                   & --       &       &   167.98 &  5 &   218.68 &  5
&   145.33 &  5 &   125.67 &  5 \\
33                   & --       &       &   179.86 &  5 &   250.65 &  5
&   264.39 &  5 &   192.32 &  5 \\
35                   & --       &       &   280.75 &  5 &   441.17 &  5
&   241.95 &  5 &   399.40 &  5 \\
37                   & --       &       &   233.34 &  5 &  1124.95 &  3 &
1685.18 &  2 &   863.11 &  4 \\
39                   & --       &       &  1333.45 &  3 &  4479.43 &  1 &
3955.85 &  1 &  5817.81 &  1 \\
41                   & --       &       &  1278.07 &  3 &  4517.44 &  1 &
7200.00 &  0 &  2239.94 &  2 \\
43                   & --       &       &  3865.27 &  1 &  4249.17 &  2 &
3396.23 &  1 &  3527.04 &  1 \\
45                   & --       &       &  1519.78 &  3 &  7200.00 &  0 &
3385.45 &  2 &  4479.20 &  1 \\
47                   & --       &       &  4357.15 &  1 &  3742.17 &  2 &
7200.00 &  0 &  4414.05 &  1 \\
49                   & --       &       &  4161.40 &  1 &  6796.76 &  1 &
6900.11 &  1 &  7200.00 &  0 \\
\cmidrule{2-11}
All instances combined                      & --       &       &    97.99
& 102 &   131.58 & 95 &   127.48 & 92 &   119.67 & 95 \\
Total time                                  & --            & &
\multicolumn{2}{r}{ 39:18:45     } & \multicolumn{2}{r}{ 58:01:08     } &
\multicolumn{2}{r}{ 60:22:17     } & \multicolumn{2}{r}{ 55:16:02     } \\
Symmetry time                               & --            & &
\multicolumn{2}{r}{  1:07:30     } & \multicolumn{2}{r}{  2:06:51     } &
\multicolumn{2}{r}{  2:02:49     } & \multicolumn{2}{r}{  3:48:44     } \\
Percentage time                             & --          & &
\multicolumn{2}{r}{         2.9\%} & \multicolumn{2}{r}{         3.6\%} &
\multicolumn{2}{r}{         3.4\%} & \multicolumn{2}{r}{         6.9\%} \\
\bottomrule
\end{tabular}
\end{table}

\begin{table}[h]
\caption{3-edge coloring on flower snark graphs $J_n$,
{\tt respect}-relabeling.}
\scriptsize
\centering
\begin{tabular}{l*{10}{rr}}
\toprule
& \multicolumn{2}{c}{\tt nosym}& \multicolumn{2}{c}{\tt gen}&
\multicolumn{2}{c}{\tt group}& \multicolumn{2}{c}{\tt nopeek}&
\multicolumn{2}{c}{\tt peek}\\
\cmidrule(l{1pt}r{1pt}){2-3}
\cmidrule(l{1pt}r{1pt}){4-5}
\cmidrule(l{1pt}r{1pt}){6-7}
\cmidrule(l{1pt}r{1pt}){8-9}
\cmidrule(l{1pt}r{1pt}){10-11}
instance & time(s) & S & time(s) & S & time(s) & S & time(s) & S & time(s)
& S \\
\midrule
3                    & --       &       &     0.00 &  5 &     0.00 &  5
&     0.00 &  5 &     0.00 &  5 \\
5                    & --       &       &     0.01 &  5 &     0.01 &  5
&     0.01 &  5 &     0.02 &  5 \\
7                    & --       &       &     0.03 &  5 &     0.03 &  5
&     0.03 &  5 &     0.03 &  5 \\
9                    & --       &       &     0.21 &  5 &     0.23 &  5
&     0.21 &  5 &     0.21 &  5 \\
11                   & --       &       &     0.59 &  5 &     0.50 &  5
&     0.51 &  5 &     0.50 &  5 \\
13                   & --       &       &     1.06 &  5 &     1.39 &  5
&     1.00 &  5 &     0.99 &  5 \\
15                   & --       &       &     3.93 &  5 &     2.56 &  5
&     3.17 &  5 &     3.26 &  5 \\
17                   & --       &       &     9.01 &  5 &     7.51 &  5
&     8.36 &  5 &     7.75 &  5 \\
19                   & --       &       &    23.34 &  5 &    15.71 &  5
&    16.11 &  5 &    16.60 &  5 \\
21                   & --       &       &    48.78 &  5 &    30.36 &  5
&    30.54 &  5 &    33.10 &  5 \\
23                   & --       &       &    66.14 &  5 &    47.37 &  5
&    45.44 &  5 &    44.85 &  5 \\
25                   & --       &       &    85.33 &  5 &    66.92 &  5
&    73.00 &  5 &    74.54 &  5 \\
27                   & --       &       &    94.18 &  5 &    99.08 &  5
&    92.91 &  5 &   112.44 &  5 \\
29                   & --       &       &   150.11 &  5 &   167.78 &  5
&   157.78 &  5 &   232.94 &  5 \\
31                   & --       &       &   315.77 &  5 &   208.05 &  5
&   399.83 &  5 &   226.83 &  5 \\
33                   & --       &       &  1130.03 &  4 &   882.55 &  5
&   599.91 &  5 &   662.99 &  5 \\
35                   & --       &       &  1401.11 &  5 &  1723.64 &  3 &
3068.62 &  3 &  1429.65 &  4 \\
37                   & --       &       &  7200.00 &  0 &  6042.10 &  1 &
1375.05 &  3 &  4953.76 &  2 \\
39                   & --       &       &  5072.36 &  1 &  3873.88 &  1 &
5507.11 &  1 &  4105.52 &  1 \\
41                   & --       &       &  5875.18 &  1 &  4071.27 &  1 &
7200.00 &  0 &  7200.00 &  0 \\
43                   & --       &       &  4127.42 &  1 &  7200.00 &  0 &
7200.00 &  0 &  7200.00 &  0 \\
45                   & --       &       &  7200.00 &  0 &  7200.00 &  0 &
7200.00 &  0 &  7200.00 &  0 \\
47                   & --       &       &  4642.87 &  1 &  7200.00 &  0 &
6160.08 &  1 &  7200.00 &  0 \\
49                   & --       &       &  7200.00 &  0 &  7200.00 &  0 &
7200.00 &  0 &  7200.00 &  0 \\
\cmidrule{2-11}
All instances combined                      & --       &       &   184.44
& 88 &   173.41 & 86 &   174.32 & 88 &   178.46 & 87 \\
Total time                                  & --            & &
\multicolumn{2}{r}{ 71:53:32     } & \multicolumn{2}{r}{ 72:39:08     } &
\multicolumn{2}{r}{ 70:44:12     } & \multicolumn{2}{r}{ 72:09:40     } \\
Symmetry time                               & --            & &
\multicolumn{2}{r}{  1:53:37     } & \multicolumn{2}{r}{  2:27:03     } &
\multicolumn{2}{r}{  2:24:02     } & \multicolumn{2}{r}{  4:13:49     } \\
Percentage time                             & --          & &
\multicolumn{2}{r}{         2.6\%} & \multicolumn{2}{r}{         3.4\%} &
\multicolumn{2}{r}{         3.4\%} & \multicolumn{2}{r}{         5.9\%} \\
\bottomrule
\end{tabular}
\end{table}

\begin{table}[h]
\caption{Relevant MIPLIB 2010 and MIPLIB 2017 benchmark instances,
original relabeling.}
\scriptsize
\centering
\begin{tabular}{l*{10}{rr}}
\toprule
& \multicolumn{2}{c}{\tt nosym}& \multicolumn{2}{c}{\tt gen}&
\multicolumn{2}{c}{\tt group}& \multicolumn{2}{c}{\tt nopeek}&
\multicolumn{2}{c}{\tt peek}\\
\cmidrule(l{1pt}r{1pt}){2-3}
\cmidrule(l{1pt}r{1pt}){4-5}
\cmidrule(l{1pt}r{1pt}){6-7}
\cmidrule(l{1pt}r{1pt}){8-9}
\cmidrule(l{1pt}r{1pt}){10-11}
instance & time(s) & S & time(s) & S & time(s) & S & time(s) & S & time(s)
& S \\
\midrule
cod105               &  7200.04 &  0 &    65.94 &  5 &    61.84 &  5 &
59.14 &  5 &    57.45 &  5 \\
cov1075              &  4761.35 &  5 &   117.17 &  5 &    47.48 &  5 &
32.71 &  5 &    33.15 &  5 \\
fastxgemm-n2r6s0t2   &  1628.36 &  5 &   286.01 &  5 &   209.40 &  5 &
141.51 &  5 &   145.82 &  5 \\
fastxgemm-n2r7s4t1   &  6394.82 &  2 &   971.15 &  5 &   812.15 &  5 &
768.81 &  5 &   776.56 &  5 \\
neos-1324574         &  6251.90 &  5 &  2398.62 &  5 &  2080.58 &  5 &
2064.22 &  5 &  2060.96 &  5 \\
neos-3004026-krka    &   129.58 &  5 &   262.07 &  5 &   144.61 &  5 &
685.34 &  4 &  1286.86 &  4 \\
neos-953928          &  2332.51 &  5 &  1471.52 &  5 &  1166.06 &  5 &
975.68 &  5 &   975.07 &  5 \\
neos-960392          &   850.24 &  5 &  1954.62 &  5 &  2278.05 &  5 &
2451.05 &  4 &  2454.52 &  4 \\
supportcase29        &   282.05 &  5 &   196.09 &  5 &   253.30 &  5 &
388.96 &  5 &   662.54 &  5 \\
wachplan             &  2713.70 &  5 &   906.09 &  5 &   939.21 &  5 &
849.65 &  5 &   849.25 &  5 \\
\cmidrule{2-11}
All instances combined                      &  1853.78 & 42 &   491.69 &
50 &   407.76 & 50 &   449.34 & 48 &   506.23 & 48 \\
Total time                                  & \multicolumn{2}{r}{
45:46:14     } & \multicolumn{2}{r}{ 13:24:05     } & \multicolumn{2}{r}{
12:12:07     } & \multicolumn{2}{r}{ 15:57:33     } & \multicolumn{2}{r}{
18:43:22     } \\
Symmetry time                               & \multicolumn{2}{r}{
0:00:00     } & \multicolumn{2}{r}{  0:03:44     } & \multicolumn{2}{r}{
0:03:47     } & \multicolumn{2}{r}{  0:06:56     } & \multicolumn{2}{r}{
1:56:05     } \\
Percentage time                             & \multicolumn{2}{r}{
0.0\%} & \multicolumn{2}{r}{         0.5\%} & \multicolumn{2}{r}{
0.5\%} & \multicolumn{2}{r}{         0.7\%} & \multicolumn{2}{r}{
10.3\%} \\
\bottomrule
\end{tabular}
\end{table}

\begin{table}[h]
\caption{Relevant MIPLIB 2010 and MIPLIB 2017 benchmark instances,
{\tt max}-relabeling.}
\scriptsize
\centering
\begin{tabular}{l*{10}{rr}}
\toprule
& \multicolumn{2}{c}{\tt nosym}& \multicolumn{2}{c}{\tt gen}&
\multicolumn{2}{c}{\tt group}& \multicolumn{2}{c}{\tt nopeek}&
\multicolumn{2}{c}{\tt peek}\\
\cmidrule(l{1pt}r{1pt}){2-3}
\cmidrule(l{1pt}r{1pt}){4-5}
\cmidrule(l{1pt}r{1pt}){6-7}
\cmidrule(l{1pt}r{1pt}){8-9}
\cmidrule(l{1pt}r{1pt}){10-11}
instance & time(s) & S & time(s) & S & time(s) & S & time(s) & S & time(s)
& S \\
\midrule
cod105               & --       &       &  6635.01 &  3 &  1450.64 &  5 &
1563.47 &  5 &  1629.04 &  5 \\
cov1075              & --       &       &   883.55 &  5 &   156.74 &  5
&   157.80 &  5 &   194.83 &  5 \\
fastxgemm-n2r6s0t2   & --       &       &   312.75 &  5 &   229.14 &  5
&   145.76 &  5 &   151.16 &  5 \\
fastxgemm-n2r7s4t1   & --       &       &   975.69 &  5 &   813.14 &  5
&   777.50 &  5 &   787.83 &  5 \\
neos-1324574         & --       &       &  1847.38 &  5 &  1473.25 &  5 &
1201.85 &  5 &  1288.09 &  5 \\
neos-3004026-krka    & --       &       &   485.01 &  5 &   746.75 &  4
&   519.21 &  4 &   417.15 &  4 \\
neos-953928          & --       &       &   518.51 &  5 &   426.36 &  5 &
1219.34 &  5 &   927.31 &  5 \\
neos-960392          & --       &       &  1855.01 &  4 &  1475.19 &  5 &
1740.15 &  5 &  1799.96 &  5 \\
supportcase29        & --       &       &   578.87 &  5 &  1888.74 &  4 &
1060.72 &  4 &  1679.39 &  5 \\
wachplan             & --       &       &  1965.03 &  5 &  2125.04 &  5 &
1767.51 &  5 &  1691.19 &  5 \\
\cmidrule{2-11}
All instances combined                      & --       &       &  1061.29
& 47 &   812.31 & 48 &   771.56 & 48 &   796.69 & 49 \\
Total time                                  & --            & &
\multicolumn{2}{r}{ 24:29:35     } & \multicolumn{2}{r}{ 19:32:41     } &
\multicolumn{2}{r}{ 19:18:10     } & \multicolumn{2}{r}{ 17:59:25     } \\
Symmetry time                               & --            & &
\multicolumn{2}{r}{  0:04:15     } & \multicolumn{2}{r}{  0:08:27     } &
\multicolumn{2}{r}{  0:11:46     } & \multicolumn{2}{r}{  1:31:29     } \\
Percentage time                             & --          & &
\multicolumn{2}{r}{         0.3\%} & \multicolumn{2}{r}{         0.7\%} &
\multicolumn{2}{r}{         1.0\%} & \multicolumn{2}{r}{         8.5\%} \\
\bottomrule
\end{tabular}
\end{table}

\begin{table}[h]
\caption{Relevant MIPLIB 2010 and MIPLIB 2017 benchmark instances,
{\tt min}-relabeling.}
\scriptsize
\centering
\begin{tabular}{l*{10}{rr}}
\toprule
& \multicolumn{2}{c}{nosym}& \multicolumn{2}{c}{default}&
\multicolumn{2}{c}{symre}& \multicolumn{2}{c}{nopeek}&
\multicolumn{2}{c}{peek}\\
\cmidrule(l{1pt}r{1pt}){2-3}
\cmidrule(l{1pt}r{1pt}){4-5}
\cmidrule(l{1pt}r{1pt}){6-7}
\cmidrule(l{1pt}r{1pt}){8-9}
\cmidrule(l{1pt}r{1pt}){10-11}
instance & time(s) & S & time(s) & S & time(s) & S & time(s) & S & time(s)
& S \\
\midrule
cod105               & --       &       &  6724.34 &  3 &  1450.69 &  5 &
1541.47 &  5 &  1629.39 &  5 \\
cov1075              & --       &       &   419.35 &  5 &   144.54 &  5
&   200.75 &  5 &   195.73 &  5 \\
fastxgemm-n2r6s0t2   & --       &       &   324.24 &  5 &   199.63 &  5
&   150.58 &  5 &   168.10 &  5 \\
fastxgemm-n2r7s4t1   & --       &       &   969.75 &  5 &   812.81 &  5
&   779.84 &  5 &   789.20 &  5 \\
neos-1324574         & --       &       &  1812.54 &  5 &  1962.48 &  5 &
2351.54 &  5 &  1980.54 &  5 \\
neos-3004026-krka    & --       &       &   237.08 &  5 &   131.02 &  5 &
1366.37 &  3 &   511.05 &  4 \\
neos-953928          & --       &       &   457.24 &  5 &   828.82 &  5
&   534.57 &  5 &   535.76 &  5 \\
neos-960392          & --       &       &  2166.11 &  5 &  1900.46 &  5 &
1799.27 &  5 &  1825.84 &  5 \\
supportcase29        & --       &       &   440.40 &  5 &   282.07 &  5
&   753.98 &  4 &   179.23 &  5 \\
wachplan             & --       &       &  2020.76 &  5 &  1702.72 &  5 &
1873.75 &  5 &  1844.22 &  5 \\
\cmidrule{2-11}
All instances combined                      & --       &       &   901.65
& 48 &   612.10 & 50 &   837.48 & 47 &   657.29 & 49 \\
Total time                                  & --            & &
\multicolumn{2}{r}{ 23:29:14     } & \multicolumn{2}{r}{ 14:28:11     } &
\multicolumn{2}{r}{ 20:57:14     } & \multicolumn{2}{r}{ 16:24:51     } \\
Symmetry time                               & --            & &
\multicolumn{2}{r}{  0:05:54     } & \multicolumn{2}{r}{  0:04:36     } &
\multicolumn{2}{r}{  0:08:17     } & \multicolumn{2}{r}{  0:22:54     } \\
Percentage time                             & --          & &
\multicolumn{2}{r}{         0.4\%} & \multicolumn{2}{r}{         0.5\%} &
\multicolumn{2}{r}{         0.7\%} & \multicolumn{2}{r}{         2.3\%} \\
\bottomrule
\end{tabular}
\end{table}

\begin{table}[h]
\caption{Relevant MIPLIB 2010 and MIPLIB 2017 benchmark instances,
{\tt respect}-relabeling.}
\scriptsize
\centering
\begin{tabular}{l*{10}{rr}}
\toprule
& \multicolumn{2}{c}{\tt nosym}& \multicolumn{2}{c}{\tt gen}&
\multicolumn{2}{c}{\tt group}& \multicolumn{2}{c}{\tt nopeek}&
\multicolumn{2}{c}{\tt peek}\\
\cmidrule(l{1pt}r{1pt}){2-3}
\cmidrule(l{1pt}r{1pt}){4-5}
\cmidrule(l{1pt}r{1pt}){6-7}
\cmidrule(l{1pt}r{1pt}){8-9}
\cmidrule(l{1pt}r{1pt}){10-11}
instance & time(s) & S & time(s) & S & time(s) & S & time(s) & S & time(s)
& S \\
\midrule
cod105               & --       &       &  6729.60 &  3 &  1449.94 &  5 &
1561.65 &  5 &  1629.89 &  5 \\
cov1075              & --       &       &   847.73 &  5 &   170.26 &  5
&   158.31 &  5 &   213.03 &  5 \\
fastxgemm-n2r6s0t2   & --       &       &   263.54 &  5 &   180.64 &  5
&   139.46 &  5 &   147.08 &  5 \\
fastxgemm-n2r7s4t1   & --       &       &   970.69 &  5 &   815.20 &  5
&   777.01 &  5 &   788.21 &  5 \\
neos-1324574         & --       &       &  2688.34 &  5 &  2216.11 &  5 &
1836.15 &  5 &  1759.82 &  5 \\
neos-3004026-krka    & --       &       &   374.60 &  4 &   578.92 &  4
&   189.27 &  5 &   418.20 &  4 \\
neos-953928          & --       &       &  1013.14 &  5 &  1710.39 &  5 &
1351.12 &  5 &  1346.81 &  5 \\
neos-960392          & --       &       &  1717.12 &  5 &  2152.39 &  5 &
2101.22 &  5 &  2096.46 &  5 \\
supportcase29        & --       &       &   269.59 &  5 &   272.41 &  5
&   420.73 &  5 &   467.09 &  5 \\
wachplan             & --       &       &  1019.60 &  5 &  1017.95 &  5
&   928.59 &  5 &   947.01 &  5 \\
\cmidrule{2-11}
All instances combined                      & --       &       &   970.19
& 47 &   743.01 & 49 &   639.20 & 50 &   723.93 & 49 \\
Total time                                  & --            & &
\multicolumn{2}{r}{ 25:20:14     } & \multicolumn{2}{r}{ 17:57:44     } &
\multicolumn{2}{r}{ 14:57:39     } & \multicolumn{2}{r}{ 16:48:27     } \\
Symmetry time                               & --            & &
\multicolumn{2}{r}{  0:04:26     } & \multicolumn{2}{r}{  0:19:34     } &
\multicolumn{2}{r}{  0:05:03     } & \multicolumn{2}{r}{  0:55:03     } \\
Percentage time                             & --          & &
\multicolumn{2}{r}{         0.3\%} & \multicolumn{2}{r}{         1.8\%} &
\multicolumn{2}{r}{         0.6\%} & \multicolumn{2}{r}{         5.5\%} \\
\bottomrule
\end{tabular}
\end{table}

\end{document}